%% file: Final_Version.tex
\documentclass[12pt,reqno]{amsart}
\usepackage{amscd,amssymb,graphicx,color,a4wide,cite}
\usepackage{listings}
\usepackage{color}
\usepackage{booktabs}
\usepackage[framemethod=tikz]{mdframed}
\usepackage{hyperref}

\usepackage{mathrsfs}

\usepackage{epstopdf}

\usepackage[all]{xy}
\let\oldtocsection=\tocsection
 
\let\oldtocsubsection=\tocsubsection
 
\let\oldtocsubsubsection=\tocsubsubsection
 
\renewcommand{\tocsection}[2]{\hspace{0em}\oldtocsection{#1}{#2}}
\renewcommand{\tocsubsection}[2]{\hspace{1em}\oldtocsubsection{#1}{#2}}
\renewcommand{\tocsubsubsection}[2]{\hspace{2em}\oldtocsubsubsection{#1}{#2}}

\newtheorem{theorem}{Theorem}[section]
\newtheorem{lemma}[theorem]{Lemma}
\newtheorem{proposition}[theorem]{Proposition}
\newtheorem{corollary}[theorem]{Corollary}

\theoremstyle{definition}
\newtheorem{definition}[theorem]{Definition}
\newtheorem{construction}[theorem]{Construction}
\newtheorem{discussion}[theorem]{Discussion}

\newtheorem{example}[theorem]{Example}

\theoremstyle{remark}
\newtheorem{remark}[theorem]{Remark}

\numberwithin{equation}{section}
\numberwithin{figure}{section}

\newcommand{\NN} {\mathbb{N}}
\newcommand{\ZZ} {\mathbb{Z}}
\newcommand{\QQ} {\mathbb{Q}}
\newcommand{\RR} {\mathbb{R}}
\newcommand{\CC} {\mathbb{C}}

\newcommand{\PP} {\mathbb{P}}
\renewcommand{\AA} {\mathbb{A}}
\newcommand{\GG} {\mathbb{G}}

\newcommand {\ev} {\operatorname{ev}}

\newcommand {\gp} {{\operatorname{gp}}}

\newcommand {\Hom} {\operatorname{Hom}}

\newcommand {\id} {\operatorname{id}}
\newcommand {\im} {\operatorname{im}}

\newcommand {\Int} {\operatorname{int}}

\newcommand {\kk} {\Bbbk}

\newcommand {\lra} {\longrightarrow}

\newcommand {\M} {\mathcal{M}}

\newcommand {\maxid} {\mathfrak{m}}

\renewcommand{\O} {\mathcal{O}}

\renewcommand {\P} {\mathscr{P}}

\newcommand {\pr} {\operatorname{pr}}

\newcommand {\Spec} {\operatorname{Spec}}

\newcommand {\T} {\mathfrak{T}}

\newcommand{\bbfamily}{\fontencoding{U}\fontfamily{bbold}\selectfont}
\newcommand{\textbb}[1]{{\bbfamily#1}}
\newcommand {\lfor} {\mbox{\rm\textbb{[}}}
\newcommand {\rfor} {\mbox{\rm\textbb{]}}}

\newcommand {\trcone} {\overline{C}}
\newcommand {\ul} {\underline}
\newcommand {\ol} {\overline}

\DeclareMathOperator{\Map}{Map}



\newcommand\restr[2]{{
  \left.\kern-\nulldelimiterspace 
  #1 
  \vphantom{\big|} 
  \right|_{#2} 
  }}

\newcommand {\Trop}  {\operatorname{Trop}}

\makeatletter
\newcommand{\xMapsto}[2][]{\ext@arrow 0599{\Mapstofill@}{#1}{#2}}
\def\Mapstofill@{\arrowfill@{\Mapstochar\Relbar}\Relbar\Rightarrow}
\makeatother

\newtheoremstyle{cited}%
  {3pt}
  {3pt}
  {\itshape}
  {}
  {\bfseries}
  {.}
  {.5em}
  {\thmname{#1} \thmnumber{#2} \thmnote{\normalfont#3}}

\theoremstyle{cited}

\usepackage{tcolorbox}

\usepackage{mathtools}

\usepackage{tikz}
\usetikzlibrary{shapes,arrows}
\usetikzlibrary{positioning}
\usepackage{fixltx2e}

\usepackage{epstopdf}

\begin{document}

\title[A Mirror for the Tate Curve]{Mirror Symmetry for the Tate Curve\\ via Tropical and Log Corals}


\author{H\"ulya Arg\"uz}
\address{Universit\'e de Versailles Saint-Quentin-en-Yvelines, Paris--Saclay \\78000, France}
\email{harguz@ic.ac.uk}

\date{\today}
\maketitle

\begin{abstract}
We introduce tropical corals, balanced trees in a half-space, and show that they correspond to holomorphic polygons capturing the product rule in Lagrangian Floer theory for the elliptic curve. We then prove a correspondence theorem equating counts of tropical corals to punctured log Gromov--Witten invariants of the Tate curve. This implies that the homogeneous coordinate ring of the mirror to the Tate curve is isomorphic to the degree-zero part of symplectic homology, confirming a prediction of homological mirror symmetry.
\end{abstract}

\tableofcontents


\section{Introduction.}

\subsection{Background and context}
Mirror symmetry, first observed by physicists studying string theory, posits that Calabi--Yau manifolds exist in pairs $(X,\Breve{X})$, where the symplectic geometry of $X$ controls the complex geometry of $\Breve{X}$ and vice-versa. This phenomenon has received a great deal of attention from mathematicians over the last thirty years, and has been given a compelling categorical formulation by Kontsevich~\cite{Ko} known as \emph{homological mirror symmetry}. Although still conjectural, homological mirror symmetry serves as an important organising principle in the field, and suggests answers to many fundamental questions. For example: given a Calabi--Yau manifold $X$, how can we construct the mirror manifold $\breve{X}$? 

Homological mirror symmetry implies that, in good circumstances, the homogeneous coordinate ring of the mirror $\breve{X}$ should be isomorphic to the degree-zero part of the \emph{symplectic cohomology} $SH^0(X)$ \cite{Se}. Proving this is challenging, because tools for computing symplectic cohomology -- which encodes part of the Floer theory of $X$ -- are still in their infancy. 
Indeed, so far symplectic cohomology has been computed to verify manifestations
of homological mirror symmetry only in a handful of cases, in the situation of two dimensional
Liouville domains 
\cite{P}. In this paper, we investigate it in the first known case where $X$ is not a Liouville domain: the Tate curve, a family 
\[T\longrightarrow \Spec\CC\lfor u \rfor\]
with generic fiber a smooth elliptic curve and central fiber a cycle $T_0$ of $b$ rational curves, for a fixed integer $b$.


Our work builds directly on the Gross--Siebert program \cite{affinecomplex,GSCanScat}, which aims at an algebro-geometric understanding of mirror symmetry.
This program is motivated by the Strominger--Yau--Zaslow (SYZ) conjecture, suggesting that a mirror pair $(X, \breve{X})$ of Calabi--Yau manifolds should carry dual special Lagrangian torus fibrations $X \to B$, $\breve{X} \to B$ over the same base~$B$~\cite{SYZ}. 
A key ingredient in the Gross--Siebert approach is tropical geometry, which one should think of as occurring on the base $B$ of the fibration -- this is viewed as the ``tropicalization'' of a degenerate limit of $X$ (or $\breve{X}$).

From the point of view of the Gross--Siebert program, the homogeneous coordinate ring of the mirror to the open Calabi--Yau manifold $T \setminus T_0$ -- or rather, the mirror to the log Calabi--Yau pair $(T,T_0)$ -- arises as a \emph{ring of theta functions} \cite{affinecomplex,gross2021canonical}. Structure constants in this ring count maps from algebraic curves into $T$ that satisfy prescribed tangency conditions~\cite{gross2019intrinsic}: these are the punctured log Gromov--Witten invariants defined by Abramovich--Chen--Gross--Siebert~\cite{ACGS}. Our contribution is to establish the correct notion of tropicalization for such curves in this setting. We introduce tropical objects, called \emph{corals}, that morally speaking arise as the tropical limits of punctured log curves in $T$. We then prove a correspondence between punctured log Gromov--Witten invariants of $T$ and counts of corals (Theorem \ref{Main Theorem for the Tate curve} and Theorem \ref{PuncturedCorrespondence}). 

On the other hand, the discussion in \S\ref{Sec: SH} below expresses structure constants of the symplectic cohomology $SH^0(T \setminus T_0)$ in terms of counts of holomorphic polygons in $E$, and further in terms of counts of combinatorial objects called \emph{tropical Morse trees} introduced by Abouzaid--Gross--Siebert \cite[\S 8]{Clay}. By establishing a correspondence between tropical Morse trees and corals (Theorem \ref{SurjectionOfCoralsToTMT}) we verify that the product structure in the degree-zero part of symplectic cohomology $SH^0(T \setminus T_0)$ agrees with the one in the homogeneous coordinate ring of the mirror to $T \setminus T_0$, given by the ring of theta functions.

\subsection{Outline of the paper}
We now give a more detailed overview of the argument, following the outline in Figure~\ref{results}. 

\tikzstyle{decision} = [diamond, draw, fill=blue!20,
    text width=6.5em, text badly centered, node distance=3cm, inner sep=0pt]
\tikzstyle{block} = [rectangle, draw,
    text width=14.8em, text centered, rounded corners, minimum height=4em]
\tikzstyle{mycircle} = [circle, thick, draw=orange, minimum height=4mm]

\tikzstyle{line} = [draw, -latex']
\tikzstyle{cloud} = [draw, ellipse,fill=red!20, node distance=3cm,
    minimum height=2em]  
    
   \begin{figure}
\begin{tikzpicture}[align=center,node distance = 2cm, auto]
 
    \node [block] (ProductSH) {Product rule in $SH^0(T \setminus T_0)$}
    ;
    \node [block, right of=ProductSH, node distance=10.0cm] (ProductTheta) {Product rule in the ring of theta functions 
    $\hat{R}_{T \setminus T_0}$};
    
     \node [block, below of=ProductSH, node distance=3cm] (Hol) { \# Holomorphic polygons in $E$  };
      
       \node [block, below of=ProductTheta, node distance=3cm] (Punctured)
      {\# Punctured log maps in $T$ }      ; 
             
      \node [block, below of=Hol, node distance=3cm] (Morse)
         {\# Tropical Morse trees on $S^1$}    ; 
             
      \node [block, below of=Punctured, node distance=3cm] (LogCorals)
 {\# Log corals in $T_0 \times \AA^1$}             ; 
           
           \node [block, below right=1.5cm and -1.5 cm of Morse] (TropCorals)
 {\# Tropical corals in $\overline{C}S^1$}  ;
      
\draw[double,<->] (TropCorals) -- node [text width=2cm, midway, right=0.5em] { Thm \ref{Main Theorem for the Tate curve}  } (LogCorals);

\draw[double,<->] (ProductSH) -- node [text width=1.5cm, midway, left=0.5em] {\hspace{0.4cm} \S\ref{Sec: SH} } (Hol);

\draw[double,<->] (ProductTheta) --  node [text width=1cm, midway, right=0.5em] { \cite{gross2019intrinsic} } (Punctured); 

\draw[double,<->] (Hol) -- 
node [text width=1.5cm, midway, left=0.5em] {
\cite[\S 8]{Clay}}  (Morse); 
  
\draw[double,<->] (Punctured) --   node [text width=2cm, midway, right=0.5em] {Thm \ref{PuncturedCorrespondence}} (LogCorals); 

\draw[double,<->] (Morse) --   node [text width=2cm, midway, left=1em ] {Thm \ref{SurjectionOfCoralsToTMT}} (TropCorals); 

\draw[dashed,<->] (ProductSH) --  (ProductTheta);           
\end{tikzpicture} 
\caption{}
   \label{results} 
   \end{figure}

\subsubsection{The symplectic cohomology of $T \setminus T_0$}

By definition, symplectic cohomology is a version of Hamiltonian Floer theory for Liouville domains \cite{CFH,FH,GP,Se}. In \S \ref{Sec: SH} we investigate the symplectic cohomology of the open Calabi--Yau manifold $T \setminus T_0$. Topologically, $T \setminus T_0$ is the mapping cylinder of the Dehn twist $\tau:E\to E$ on the elliptic curve; it is not a Liouville domain. Nonetheless one can define (a version of) the symplectic cohomology ring $SH^\star(T\setminus T_0)$ -- see \S \ref{Sec: SH}, \cite{F} and \cite{McLean}. Based on discussions with Pomerleano--Tonkonog and Abouzaid--Siebert, we have that
\begin{equation}
\label{SHandFloer2}
SH^0(T\setminus T_0) \cong \bigoplus_{k \in \ZZ} HF^0 \big(L(0), L(k)\big)
\end{equation}
where $HF$ stands for Lagrangian Floer cohomology and $L(k)\subset E$ is the Lagrangian which lifts to a line of slope $k$ on the universal cover of $E$. Thus the ring structure on $SH^0(T\setminus T_0)$ is isomorphic to the ring structure on Lagrangian Floer cohomology of the elliptic curve $E$. In particular, the product rule in $SH^0(T\setminus T_0)$ can be described by counts of holomorphic polygons in $E$ bounded by Lagrangians.

\subsubsection{Holomorphic polygons in $E$ and tropical Morse trees}

Gross has established a correspondence between holomorphic polygons in $E$ and tropical Morse trees \cite[\S 8]{Clay}. This is discussed in detail in \S\ref{Sec: Tropical Morse trees}, but roughly speaking a tropical Morse tree is a map $\phi:\mathcal{R} \rightarrow S^1$, where $\mathcal{R}$ is a ribbon graph, whose edges are decorated with integer weights. The map $\phi$ is asked to satisfy a balancing condition in terms of these weights and lengths of the images of edges of $\mathcal{R}$. This enables us to construct a holomorphic polygon in $E$ from a tropical Morse tree. The correspondence between tropical Morse trees in $S^1$ and holomorphic polygons in $E$, together with \eqref{SHandFloer2}, allows us to describe the ring structure on $SH^0(T\setminus T_0)$ using counts of tropical Morse trees. 

From the SYZ point of view on mirror symmetry, the target $S^1$ of a tropical Morse tree $\phi:\mathcal{R} \rightarrow S^1$ should be thought of as the SYZ base $B$ for the elliptic curve $E$. Below we will also consider tropical Morse trees $\phi:\mathcal{R} \rightarrow \RR$ with target $\RR$, and these $\RR$ should be thought of as the SYZ base for the algebraic torus $\CC^\times$.

\subsubsection{Unfolding the Tate curve}

 Up to this point, we have considered the Tate curve $T \to \Spec \CC \lfor u \rfor$ with generic fiber a smooth elliptic curve and special fiber $T_0$ given by a cycle of rational curves. In what follows, it will be convenient to consider the \emph{unfolded Tate curve} $\widetilde{T} \to \Spec \CC[u]$ with general fiber an algebraic torus $\CC^\times$ and special fiber $\widetilde{T}_0$ given by an infinite chain of rational curves. We describe in \S\ref{Taking the quotient} how to pass from statements about the unfolded Tate curve $(\widetilde{T}, \widetilde{T_0})$ to the corresponding statements for the Tate curve $(T, T_0)$, by taking an appropriate quotient.

\subsubsection{Tropical corals}

The main new objects introduced in this paper are \emph{tropical corals}. These, unlike tropical Morse trees, are tropical objects for which one can obtain a geometric interpretation in terms of certain log maps, as explained below. Roughly, tropical corals are balanced trees contained in the truncated cone $\trcone \RR := \RR \times [1,\infty)$, and which satisfy some additional conditions where they meet the boundary of the truncated cone. A tropical coral is given by a map 
\[ h: \Gamma \longrightarrow \trcone \RR  \]
from an edge-weighted graph $\Gamma$ that satisfies a balancing condition at the images of the vertices of $\Gamma$. 
Tropical corals are similar to usual tropical curves in the plane ~\cite{M,NS}, except that the unbounded edges meeting the truncated cone at its boundary get cut off, and moreover we require these edges to lie on lines passing through the origin. Figure \ref{ExCoral} shows the image of a tropical coral in the truncated cone $\trcone \RR$; for more details on tropical corals see~\S\ref{subsec: tropical corals}.
\begin{figure}
\resizebox{.6\linewidth}{!}{
\input{HulyasCoral.pspdftex}}
	\caption{A tropical coral in the truncated cone $\trcone \RR$ over $\RR$}
\label{ExCoral}
\end{figure}

A tropical coral has a \emph{degree}, which is the set of vectors in $\ZZ^2$ defined by the (weighted) directions of unbounded edges together with the directions of the edges that meet the boundary of the truncated cone. We set up a well-defined counting problem for tropical corals in \S \ref{The count of tropical corals}, by considering tropical corals of a fixed degree $\Delta$ that satisfy a constraint~$\lambda$; here $\lambda$ specifies the position of all but one 
unbounded edge. We define the multiplicity of a tropical curve (Definition~\ref{def mult labelled}), and set $N^{\mathrm{trop}}_{\Delta,\lambda}$ to be the number of tropical corals of degree $\Delta$ that satisfy $\lambda$, counted with multiplicity.

\begin{theorem}[see Lemma~\ref{stable range non-empty} and Theorem~\ref{Independence Of Constraint} for details]

For a general constraint~$\lambda$, chosen from a non-empty set called the \emph{stable range}, the corresponding count of tropical corals $N^{\mathrm{trop}}_{\Delta,\lambda}$ is independent of~$\lambda$.
\end{theorem}

From the SYZ point of view on mirror symmetry, the target $\trcone \RR$ of a tropical coral $h: \Gamma \to \trcone \RR$ should be thought of as the SYZ base $B$ for the total space of the unfolded Tate curve $\widetilde{T}$. That is, corals are tropical curves in the tropicalization of a degenerate limit of $\widetilde{T}$ (see Theorem \ref{base change trcone} and Proposition \ref{prop_trop_Y})).

\subsubsection{Tropical corals and tropical Morse trees}

In \S \ref{From Tropical Morse Trees to Tropical Corals} we show that a tropical coral $h \colon \Gamma \to \trcone \RR$ without horizontal edges determines a tropical Morse tree $\phi \colon \mathcal{R} \to \RR$. Roughly speaking, the coral itself defines the ribbon graph $\mathcal{R}$, and the map $\phi$ is defined by taking an unbounded edge of $\Gamma$ to its direction in $\PP^1(\RR) = S^1$. Conversely a general (trivalent) tropical Morse tree can be lifted to a tropical coral -- see Theorem~\ref{SurjectionOfCoralsToTMT}. Thus, we obtain:
\begin{theorem}[see Theorem~\ref{SurjectionOfCoralsToTMT} and Corollary~\ref{thm: bijective correspondence corals and TMT} for details]
There is a surjection 
\begin{eqnarray}
\label{SurjectionTMT}
\Psi: \mathcal{T} & \lra &   \mathcal{TMT}\\
\nonumber
h & \longmapsto & \psi_{h} 
\end{eqnarray}
from the set $\mathcal{T}$ of tropical corals with no horizontal edges to the set of tropical Morse trees $\mathcal{TMT}$. The fiber over $\phi \in \mathcal{TMT}$ is the set of tropical corals of a fixed type determined by $\phi$. In particular, fixing appropriate constraints for the corals, we obtain a bijective correspondence between tropical Morse trees and tropical corals.
\end{theorem}
\noindent This allows us to describe the ring structure on $SH^0(T\setminus T_0)$ through counts of tropical corals $N^{\mathrm{trop}}_{\Delta,\lambda}$. 

\subsubsection{Tropical corals and log corals}

We now turn to the geometric curve-counting problem of which tropical corals give the tropicalization: this involves counting certain log curves. As a first step we introduce an intermediate space, which one can think of as a degeneration of the total space of the unfolded Tate curve. Consider the unfolded Tate curve as a family over $\Spec \CC [ u ]$, and take the base change $Y$ of this family along the map
\[ \Spec \CC [ s, t ] \to \Spec \CC [ u ] \]
given by $u = st$.  The general fiber of the degeneration given by the composition $Y \to \Spec \CC [ s, t ] \to \Spec \CC [ t ]$ is the total space $\widetilde{T}$ of the unfolded Tate curve. The degenerate limit of $\widetilde{T}$, given by the central fiber of $Y \to \Spec \CC [ t ] $ over $t=0$, is $\widetilde{T}_0 \times \AA^1$. We consider the total space of $Y$ as a log scheme, equipping it with the divisorial log structure with respect to the central fiber $\widetilde{T}_0 \times \AA^1$, and equipping the central fiber itself with the pullback log structure.

A \emph{log coral}, defined in \S\ref{A curve counting problem}, is a log map to $\widetilde{T}_0 \times \AA^1$ from a marked nodal genus-zero log curve $C$ that may have non-compact components. The non-compact components here are required to be \emph{parallel}, that is, they project dominantly to $\AA^1$ and project to a single smooth point in $\widetilde{T}_0$. In \S\ref{A curve counting problem} we set up a well-defined counting theory for log corals. We consider log corals of fixed degree $\Delta$, which satisfy an \emph{asymptotic constraint}~$\lambda$ and a \emph{log constraint}~$\rho$; here the asymptotic constraint is the same data as the tropical constraint $\lambda$ considered above, and the log constraint $\rho$, defined as in Definition~\ref{log constraint}, imposes additional incidence conditions at marked points. We write $N^{\log}_{\Delta,\lambda,\rho}$ for the number of log corals of degree $\Delta$ which satisfy the asymptotic constraint $\lambda$ and the log constraint $\rho$. The main result of \S \ref{log count equals tropical count} is the following correspondence theorem for (log) corals:
\begin{theorem}[see Theorem~\ref{main theorem} for details]
For general asymptotic constraint $\lambda$ in the stable range and any choice of log constraint $\rho$, we have
\[N^{\log}_{\Delta,\lambda,\rho}= N^{\mathrm{trop}}_{\Delta,\lambda}.   \]
In particular, the counts $N^{\log}_{\Delta,\lambda,\rho}$ are log Gromov--Witten invariants of the (unfolded) Tate curve, independent of $\lambda$ and $\rho$.
\end{theorem}

This is the technical heart of the paper. We argue as follows. We consider an open subset $U_0 \subset \widetilde{T}_0 \times \AA^1$ which contains the images of all log corals contributing to $N^{\log}_{\Delta,\lambda,\rho}$. Extending the truncated cone, we obtain a compact space $\overline{U}_0$ which contains $U_0$ and has an additional component $Z_0$ (which corresponds to the origin in the plane). Then, extending the edges of a tropical coral which meet the boundary of the truncated cone, so that they pass through the origin defines a tropical curve in the plane, and by Nishinou--Siebert~\cite{NS} this corresponds to a map from a (compact) curve $\tilde{C} \to \overline{U}_0$ that satisfies certain tangency conditions. 

Restricting this map $\tilde{C} \to \overline{U}_0$ to land in $\widetilde{T}_0 \times \AA^1$ defines a tropical coral $C \to \widetilde{T}_0 \times \AA^1$. The curve $\tilde{C}$ is obtained from $C$ by, for each non-compact component of $C$, compactifying it to $\PP^1$ and then gluing on an additional rational component which meets $\PP^1$ at the compactifying point at infinity and which maps to $Z_0$. The map from the additional rational components to $Z_0$ is standard, and is determined by the log constraint $\rho$. Thus our construction gives a bijective correspondence between log corals and tropical corals.

There is an important technical point here, which is discussed further in \S \ref{Reduction to the torically transverse case}. Nishinou--Siebert establish their correspondence theorem under a restrictive hypothesis, called toric transversality, that fails here. Lemma~\ref{Interpolating family} shows that their theorem in fact holds under weaker hypotheses, which include the case of the (unfolded) Tate curve. This is likely to have broader application.

\subsubsection{Log corals and punctured log maps}

As discussed in \S\ref{The punctured invariants of the central fiber of the Tate curve}, deleting all non-compact components of the source curve gives a one-to-one correspondence\footnote{By \S\ref{Taking the quotient}, we can consider either the Tate curve $(T,T_0)$ or the unfolded Tate curve $(\widetilde{T},\widetilde{T_0})$ here.} between log corals $C \to T_0 \times \AA^1$ and punctured log maps $C' \to T$, as defined in \cite{ACGS}. The image of such a punctured log map is necessarily contained in the central fiber $T_0 \subset T$. Since punctured log Gromov--Witten theory is unobstructed in this case. This implies:
\begin{theorem}[see Theorem~\ref{PuncturedCorrespondence} for details]
The count of log corals $N^{log}_{\Delta,\lambda,\rho}$ is a punctured log Gromov--Witten invariant of $T$. 
\end{theorem}
These punctured log invariants give the structure constants of the ring of theta functions, that is, of the homogeneous coordinate ring of the mirror to the Tate curve $T$ \cite{gross2019intrinsic}. This completes the connection between the mirror to the Tate curve and the degree-zero symplectic homology $SH^0(T \setminus T_0)$.

\subsection{Related work and future aspects}

The correspondences in Figure~\ref{results} are expected to generalize to higher dimensional Calabi–Yau varieties. Nonetheless a higher dimensional tropical Morse category, which would be needed for this generalisation, has not been worked out so far.
Similarly we lack a suitable log deformation theory for log corals, generalising \cite[\S 7]{NS}, that would allow us to write concretely deformations of log corals 
as solutions of the Floer equation.
These solutions are punctured Riemann surfaces with asymptotic cylindrical ends, which can be visualized as thickenings of tropical corals on $\trcone S^1$. 
Such punctured Riemann surfaces with several asymptotic cylindrical ends arise in the context of involutive Lie bialgebras, and capture an $IBL_{\infty}$-structure \cite{CL,CFL,MS,CMW} on symplectic cohomology. In general one should see algebro-geometric analogues of this and similar structures arising in the ring of theta functions.

\bigskip

\begin{flushleft}
\textbf{Acknowledgements.}
\end{flushleft}
This paper is based on my PhD thesis, which would not be possible without the support of my advisor Bernd Siebert. I also thank Dan Abramovich, Mohammed Abouzaid, Mark Gross, Tom Coates and Dimitri Zvonkine for many useful conversations. Finally, I thank the anonymous referees for their many insightful comments and valuable
suggestions which have resulted in major improvements to this article. This project has
received funding from the European Research Council (ERC) under the European Union's Horizon 2020 research and innovation programme (grant agreement No. 682603), and from Fondation Math\'ematique Jacques Hadamard.
\vspace{1 cm}
\begin{flushleft}
\textbf{Conventions.}
\end{flushleft} We fix 
\[N=\ZZ^n, \,\ N_\RR=N\otimes_\ZZ \RR, \,\ M=\Hom(N,\ZZ),\,\ M_\RR=M\otimes_\ZZ \RR\] 
If $\Sigma$ is a fan in $N_\RR$ then $X(\Sigma)$ denotes the associated toric $\CC$-variety with big torus $\mathrm{Int} X(\Sigma)\simeq \GG(N)\subset X(\Sigma)$, whose complement is referred to as the \emph{toric boundary} of
$X(\Sigma)$. For a subset $\Xi\subset N_\RR$, $L(\Xi) \subset N_\RR$ denotes the \emph{linear space associated to $\Xi$}, which is the linear subspace spanned by differences $v-w$ for $v,w$ in $\Xi$. Given a monoid $P$, define the monoid ring 
\[\CC[P]:=\bigoplus_{p\in P}\CC z^p  \]
where multiplication is determined by $z^p\cdot z^{p'}=z^{p+p'}$. 
\section{The Tate Curve $T$}
\label{The Tate curve}
The Tate curve $T$ is a curve over the complete discrete valuation ring $R:= \CC \lfor u \rfor$, where $\CC\lfor u \rfor$ is the ring of formal power series in $u$ with coefficients in $\CC$. Topologically it is a family 
\[T\to \Spec\CC\lfor u \rfor\] 
of elliptic curves degenerating into a nodal elliptic curve. The construction of the Tate curve is a special case of a construction of Mumford for degenerations of abelian varieties \cite{Mu}. For details of the construction see also \cite[$8.4$]{Clay}. We describe the Tate curve torically, obtained as a quotient of its unfolding in the next section.
\subsection{The Tate curve as a $\ZZ$-quotient of its unfolding $\widetilde{T}$}
\label{The Tate curve and its unfolding}
In this section, we construct the \emph{unfolded Tate curve} and its degeneration, as a particular case of a \emph{toric degeneration} of a toric variety \cite{invitation, NS}. The Tate curve will then be obtained as a quotient of the unfolded Tate curve by an action of $\ZZ$. The initial data to construct the unfolded Tate curve, consists of the pair $(\RR,\P_b)$ where $\RR$ is viewed as an integral affine manifold \cite[Defn $2.2$]{AffLog}, which we will denote by $B$ and $\P_b$, $b \in \ZZ_{> 0}$ is a \emph{b-periodic polyhedral decomposition} of $\RR$ defined as follows.
\begin{definition}
\label{poly decomp}
Let 
\[\Xi_j :=[jb,(j+1)b],~ j\in \ZZ\] 
be a closed interval of integral length $b$ in $\RR$. Let
\[ \Xi'_j:=\Xi_j\cap\Xi_{j+1} \]
be a common face of two such intervals.
The \emph{integral b-periodic polyhedral decomposition $\P_b$ of $\RR$} is the covering \[ \P_b:=\{ \Xi_j\} \cup \{\Xi'_j \} \] 
of $\RR$, where $j\in \ZZ$. We refer to each $\Xi_j$ and $\Xi'_j$ as a \emph{cell} of $\P_b$. The \emph{maximal cells} of $\P_b$ are the faces $\Xi_j$, $j\in \ZZ$, and the \emph{vertices} of $\P_b$ are the $0$-faces $\Xi'_j$, $j\in \ZZ$.
\end{definition}
For each cell $\Xi_j\in \P_b$ we define the convex polyhedral cone 
\[
C(\Xi_j) = \RR_{\geq 0}\cdot (\Xi_j\times \{ 1\}).
\]
We use the cone $C(\Xi_j)$ to define the toric fan 
\begin{equation}
\label{FanForX}
\widetilde{\Sigma}_{\P_b}:= \{\text{$\sigma$ a face of $C(\Xi_j)$}~|~\Xi_j\in \P_b \}
\end{equation}
associated to  $\P_b$. The fan $\widetilde{\Sigma}_{\P_b}$ has support in $(\RR\times \RR_{>0}) \cup \{ (0,0) \}$. The projection $\RR^2\to\RR^2/\RR=\RR$ onto
the second factor defines a map of fans $\widetilde{\Sigma}_{\P_b} \to \{0,\RR_{\ge0}\}$ which induces the morphism
\begin{equation}
\label{Eq:unfolded Tate curve}
 \pi:\widetilde{T}\lra \AA^1
\end{equation} 
referred to as the \emph{unfolded Tate curve}. We illustrate the map of fans in Figure \ref{Tate fan}.
\begin{figure}
\label{Fig:fan for unfolded tate curve}
\input{MumfordFan.pspdftex}
	\caption{The toric fan for the unfolded Tate curve}
\label{Tate fan}
\end{figure}
We can find the general and central fibers of  $\pi:\widetilde{T}\to \AA^1$ by examining the pre-images of the cones $\{0\}$ and $\RR_{\geq 0}$ in the fan for $\AA^1$ under the height function (see Lemma $3.4$ and Proposition $3.5$ in \cite{NS} for details). Since each $\Xi_j\in\P_b$ is bounded, we obtain the \emph{asymptotic fan}  $\Sigma_{\P_b}$ as the fan consisting of the single point $(0,0)\in \RR^2$ \cite[Defn. $3.1$]{NS}. Therefore, the degeneration $\pi:\widetilde{T}\to \AA^1$ associated to $(\RR,\P_b)$ satisfies 
\begin{equation}
\label{degeneration1}
\pi^{-1}(\AA^1\setminus \{0\}) = \GG_m \times (\AA^1\setminus \{0\})
\end{equation}
where $\GG_m$ is the algebraic torus and hence the general fiber of $\pi:\widetilde{T}\to\AA^1$ is $\widetilde{T}_t=\GG_m$. For details of the construction of the general fiber we refer to Lemma $3.4$ in \cite{NS}.

The central fiber $\widetilde{T}_0$ of $\pi:\widetilde{T}\to \AA^1$ is determined as follows. For each vertex $v\in \P_b$ let
\[\Sigma_v=\{ \RR_{\geq 0}\cdot(\Xi-v)\subset \RR~|~\Xi\in \P, v \in \Xi  \}\]
So, $\Sigma_v$ is the fan of the toric variety $\widetilde{T}_v=\PP^1$.
Similarly, each closed interval $\Xi\in \P_b$ defines a fan $\Sigma_\Xi$ for the toric variety $\widetilde{T}_{\Xi}$ which is a point of intersection of $\widetilde{T}_v$ and $\widetilde{T}_{v'}$ where $v$ and $v'$ denote the vertices adjacent to $\Xi$. Hence, we obtain $\widetilde{T}_0$ as an infinite chain of projective lines $\mathbb{P}^1$, glued pairwise together along an $A_{b-1}$ singularity in $\widetilde{T}$.

Before proceeding with the definition of the Tate curve, we would like to describe an affine cover for the total space of the unfolded Tate curve. For this, fix a $b$-periodic polyhedral decomposition $\P_b$ of $\RR$ and let $\Xi:=[a,a+b]\subset \RR $ be a maximal cell of $\P_b$, so that
\begin{equation}
\label{C(Xi)}
C(\Xi) =C((a,1),(a+b,1)) \subset N_{\RR} 
\end{equation} 
where we describe the cone $C(\Xi)$ over $\Xi$ in terms of the integral ray generators  $(a,1)$ and $(a+b,1)$ of its edges. We will use an analogous notation throughout this section.
\begin{equation} \label{C(Xi)vee} 
C(\Xi)^{\vee}=C( (1,-a),(-1,a+b) ) \subset M_{\RR} 
\end{equation}
The generators of the monoid ring $\CC[C(\Xi)^{\vee}\cap M]$ associated to $C(\Xi)^{\vee}$ are 
\[\{ z^{(1,-a)},z^{(-1,a+b)}, z^{(0,1)} \}\] 
The isomorphism
\begin{eqnarray}\label{xyu}
\varphi: \CC[C(\Xi)^{\vee}\cap M] & \lra &\CC[x,y,u]/(xy-u^b)
\nonumber \\
z^{(1,-a)}& \longmapsto & x
\nonumber\\
z^{(-1,a+b)} & \longmapsto & y
\nonumber\\
z^{(0,1)} & \longmapsto & u
\nonumber
\end{eqnarray}
gives us an affine cover for the total space of the unfolded Tate curve, given by a countable number of copies of
\begin{equation}
\label{ub}
 \Spec \CC[x,y,u]/(xy-u^b)  
\end{equation}
Now to define the Tate curve, first define
\[ \hat{\widetilde{T}}=
\lim_{\longrightarrow} \widetilde{T}^k = 
( \widetilde{T}_0, \lim_{\longleftarrow} \mathcal{O}_{\widetilde{T}} / u^{k+1})  
\]
where $\widetilde{T}^k$ is the $k$-th order thickening of $\pi^{-1}(0)$, that is, the subscheme of the unfolded Tate curve $\pi:\widetilde{T}\to \Spec \CC [u]$ defined by the equation $u^{k+1}=0$. There is  an action of $\ZZ$ on $\hat{\widetilde{T}}$, induced by the action of $\ZZ$ on the fan $\widetilde{\Sigma}_{\P_b}$, given by translation by $b$, taking a ray $\RR_{\geq 0}(i, 1)$  to $\RR_{\geq 0}(i+d, 1)$. This has the effect of acting on the big torus orbit $(\CC^*)^2 \subseteq \widetilde{T}$, with coordinates $z,u$ by
\[ (z,u) \mapsto  (zu^b,u).  \]
Hence if we fix a non-zero value of $u$, the action on the $\CC^*$ parametrized by $z$ is $z \mapsto z \cdot u^b$. As explained in \cite[\S 8.4.1]{Clay}, the quotient $\hat{\widetilde{T}}/\ZZ$ makes sense as a formal scheme. There is a map of formal schemes $\hat{\widetilde{T}}/\ZZ \to \hat{\AA}$ induced by $\pi$. Here $\hat{\AA}=\mathrm{Spf}\CC \lfor u \rfor$. Since, there is an ample line  bundle $\mathcal{L}$ over $\hat{\widetilde{T}}/\ZZ$ \cite[pg. $620$]{Clay}, Grothendieck's Existence Theorem \cite[EGA III, $5.4.5$]{Gr} ensures that $\hat{\widetilde{T}}/\ZZ$ arises as the formal completion of a scheme 
\begin{equation}
\label{Eq: The tate curve}
T\to \Spec\mathbb{C} \lfor u \rfor
\end{equation}  
which we refer to as the \emph{Tate curve}. Note that the generic fibre of the Tate curve, is an elliptic curve over $\mathbb{C}(\!(u)\!)$ and the central fiber $T_0$ is a nodal elliptic curve with a $A_{b-1}$ singularity, whose resolution (obtained by adding all rays connecting the origin to all integral points in $B=\RR$ in Figure \ref{Tate fan}) is a cycle of $b$ rational curves -- see\cite[\S $8.4.1$]{Clay}. In \S \ref{Taking the quotient}, we will see that the invariants that we compute on the Tate curve lift uniquely to the unfolded Tate curve. Therefore, for computational convenience we will disregard the $\ZZ$-quotient in the next sections.
\begin{remark}
\label{Rem: analytic Tate curve}
When working in the analytic category, since the $\ZZ$-action on the unfolded Tate curve $\pi:\widetilde{T}\to \AA^1$ is properly discontinuous in the analytic topology once restricting to the unit disc 
\[D=\{u\in\mathbb{C}\,|\, |u|<1\},\] 
we can define 
the \emph{analytic Tate curve}
\begin{equation}
\label{Eq: the analytic Tate curve}
T^\mathrm{an}\lra D
\end{equation} 
with fiber $T^\mathrm{an}_u$ over $u\in
D\setminus\{0\}$ the elliptic curve 
\[E:=\mathbb{C}^*/(z\sim u^b\cdot
z),\] 
viewed as a complex manifold. The central fiber, $T_0^\mathrm{an}$ again a cycle of $b$ rational curves. Note that the complement of the analytic Tate curve $T^\mathrm{an} \setminus T_0^\mathrm{an}$ is the mapping cylinder $\Map(\tau)$ of the Dehn twist $\tau \colon E \to E$ of an elliptic curve along a meridian. We use this description while investigating the symplectic cohomology of $\Map(\tau)$ in \S \ref{Sec: SH}
\end{remark}

\section{A Toric Degeneration of $\widetilde{T}$ Obtained From The Truncated Cone $\trcone \RR$}
\label{The unfolded Tate curve}
In this section we construct a toric degeneration $\tilde{\pi}: Y\to \AA^2$ of the unfolded Tate curve $\widetilde{T}$ defined in \S\ref{The Tate curve and its unfolding}. This construction is similar to the construction of the unfolded Tate curve from $(B,\P_b)$. However, rather than $B$, we obtain the degeneration of the total space by from the \emph{truncated cone} $\trcone B$ over $B$, defined as in \cite[Defn $3.14$]{Theta} as an integral affine manifold with boundary. We first review the construction of the truncated cone, and describe the polyhedral decomposition on it, induced from the one on $B$. For details we refer to \cite{Theta}. 

Let $B$ be an integral affine manifold, endowed with a polyhedral decomposition $\P$. For a cell $\Xi\in \P$, let $C(\Xi)$ be the closure of the cone spanned by $\Xi \times
\{1\}$ in $N_\RR\times\RR$:
\begin{equation}
\label{cone over Xi}
C(\Xi) :=  \overline{ \big\{a \cdot (n, 1)\,\big|\, a \ge 0, n\in
\Xi\big\} }
\end{equation}
Define the cone $C(B)$ over $B$ as the polyhedral decomposition
\[ C(B) = \bigcup_{\Xi\in\P} C(\Xi)\]
with cells $C(\Xi)$ for $\Xi \in \P$. Note that $C(B)\subseteq N_{\RR}\times \RR$ admits an integral affine structure with a singularity at the origin in $N_{\RR}\times \RR$ \cite[Construction $4.11$]{GHS}. 
\begin{definition}
\label{truncated cone}
The \emph{truncated cone} $\trcone B$ over $B\subset N_{\RR}$ is the manifold with boundary with underlying topological space
\[\trcone B := \{ (x, h) \in C(B)~|~ h \geq 1 \} \]
in the cone $C(B)$, endowed with the induced affine structure. It admits a polyhedral decomposition  $\trcone \P$ with cells 
\[\trcone \Xi := \{ (x, h) \in C(\Xi)~|~\text{$h \geq 1$, $\Xi\in \P$}\} \] 
The maximal cells of $\trcone \P$ are the cells $\trcone \Xi$ such that $\Xi$ is a maximal cell of $\P$.
\end{definition}
In what follows, focusing attention to the case $B=\RR$, we construct the degeneration of the unfolded Tate curve from the pair $(\trcone \RR,\trcone \P_b)$. Here, $\trcone \RR$ is the truncated cone over $\RR$ endowed with the polyhedral decomposition $\trcone \P_b$ with maximal cells $\trcone \Xi$, where $\P_b$ is the $b$-periodic polyhedral decomposition of $\RR$. We identify $\RR$, with the line given by $y=-x+1, z=1$ in $\RR^3$, and let $C(\trcone \Xi)$ be the cone over the truncated cone $\trcone \Xi$. Define the toric fan
\begin{equation}
\label{FanForY}
\widetilde{\Sigma}_{\trcone \P_b}:=\{\text{$\sigma$ a face of $C(\trcone \Xi)$}~|~\trcone \Xi \in \trcone \P_b   \} 
\nonumber
\end{equation}
and let $Y$ be the toric variety associated to $\widetilde{\Sigma}_{\trcone \P_b}$. The projection map
\begin{eqnarray}
(\pr_2,\pr_3):\RR\times\RR\times\RR &\lra & \RR\times\RR
\nonumber \\
(x,y,z) &\longmapsto & (y,z)
\nonumber 
\end{eqnarray}
onto the second and third factors defines a map of fans
\begin{eqnarray}
(\pr_2,\pr_3):\widetilde{\Sigma}_{\trcone \P_b} &\lra & \{ 0, \RR_{\geq 0 } \} \times  \{ 0, \RR_{\geq 0 }\}
\nonumber \\
\end{eqnarray}
which induces a morphism
\[\tilde{\pi}:Y\rightarrow \Spec \CC [s,t] \]
referred to as the \emph{degeneration of the unfolded Tate curve}. As in \textsection \ref{The Tate curve and its unfolding} before, $Y$ viewed over 
$\Spec \CC[t]$ defines a toric degeneration, this time of the Tate curve $T$ into
\[  Y_0:= \AA^1_s \times \widetilde{T}_0 \]
where $Y_0$ is the central fiber of $Y\rightarrow \Spec \CC [ s,t ]\rightarrow \Spec \CC[t]$ over $t=0$. In a moment we will define an affine cover for $Y$. 

Let $\Xi:=[a,a+b]$ be a maximal cell of $\P_b$. Then, the set of ray generators the cone over the truncated cone over $\Xi$ is given by $\{ (a,1,1), (a+b,1,1), (a+b,1,0), (a,1,0) \}$ as illustrated in Figure \ref{Fig:CCXi};
\begin{figure}
\scalebox{.4}{\input{CCXi.pspdftex}}
	\caption{The ray generators (in red) of the cone over the truncated cone over a maximal cell $\Xi$ identified with $[(a,1,0),(a+b,1,0)]$}
\label{Fig:CCXi}
\end{figure}
\begin{equation}
\,\ \,\ \,\ C(\trcone (\Xi))=C((a,1,1), (a+b,1,1), (a+b,1,0), (a,1,0)).
\nonumber
\end{equation}
Its dual ${C(\trcone (\Xi))}^{\vee}$ is given by
\begin{equation} \label{C}
\,\ \,\ \,\ {C(\trcone (\Xi))}^{\vee} = C( (0,1,-1), (-1,a+b,0), (0,0,1), (1,-a,0) )
\end{equation} 
The isomorphism 
\begin{eqnarray}\label{xyst}
\widetilde{\varphi}:\CC[C(\trcone \Xi)^{\vee}\cap (M\oplus \ZZ)] & \lra & \CC[x,y,s,t]/(xy-(st)^b) 
\nonumber \\
z^{(1,-a,0)}& \longmapsto & x
\nonumber\\
z^{(-1,a+b,0)} &\longmapsto & y
\nonumber\\
z^{(0,1,-1)} &\longmapsto & s
\nonumber\\
z^{(0,0,1)} &\longmapsto & t
\nonumber
\end{eqnarray}
defines an affine cover for the total space $Y$ of the degeneration of the unfolded Tate curve, given by a countable number of copies of
\begin{equation}
\label{stb}
 \Spec \CC[x,y,s,t]/(xy-(st)^b)  
\end{equation}
\begin{theorem}
\label{base change trcone}
The degeneration $\tilde{\pi}:Y\to \Spec\mathbb{C}[s,t]$ of the unfolded Tate curve is obtained from the unfolded Tate curve $\pi: \widetilde{T}\to \Spec\CC[u]$ by the base change  $u \mapsto st$.
\begin{proof}
Let $\Xi:=[a,a+b] \subset \RR$ be a maximal cell of $\P_b$ and let $\trcone \P_b$ be the corresponding maximal cell in $\trcone \Xi$. Then, the projection map
\begin{eqnarray}
(\pr_1,\pr_2): N_{\RR} \times \RR \times \RR & \lra & N_{\RR} \times \RR
\nonumber \\
C(\trcone \Xi) & \longmapsto & C(\Xi)
\nonumber
\end{eqnarray} 
defines a map of fans from the fan $\widetilde{\Sigma}_{{\P}_b}$ defined in \eqref{FanForX} and  $\widetilde{\Sigma}_{{\trcone \P}_b}$ defined in \eqref{FanForY}. Hence, the compatibility of the gluing of affine patches follows. The dual of $(\pr_1,\pr_2)$ induces the embedding 
\begin{align*}
j:C(\Xi)^{\vee} &\hookrightarrow  C(\trcone \Xi)^{\vee}
\nonumber\\
(m_1,m_2)&\mapsto (m_1,m_2,0)
\nonumber
\end{align*}
From equations \eqref{C(Xi)vee} and \eqref{C}, it follows that
\[C(\trcone \Xi)^{\vee} = j({C(\Xi)}^{\vee})+\RR_{\geq 0}(0,1,-1)+\RR_{\geq 0}(0,0,1)    \]
Let 
\[\phi_j:\CC[{C(\Xi)}^{\vee} \cap M] \to \CC[{C(\trcone \Xi)}^{\vee} \cap M \oplus \ZZ] \] 
be the map induced by $j: {C(\Xi)}^{\vee} \hookrightarrow   {C(\trcone \Xi)}^{\vee}$ on the level of monoid rings. Explicitly, we have
\begin{eqnarray}
\phi_j:\CC[{C(\Xi)}^{\vee} \cap M] & \lra & \CC[{C(\trcone \Xi)}^{\vee} \cap M \oplus \ZZ]
\nonumber \\
x:=z^{(1,-a)} &\longmapsto & z^{(1,-a,0)} =x 
\nonumber \\
 y:=z^{(-1,a+b)} &\longmapsto & z^{(-1,a+b,0)} =y 
\nonumber \\
u:=z^{(0,1)} &\longmapsto & z^{(0,1,0)} = z^{(0,1,-1)}\cdot z^{(0,0,1)}=st 
\nonumber
\end{eqnarray}
Hence, we get 
\begin{eqnarray}
\widetilde{\varphi}\circ \phi_j \circ \varphi: \CC[x,y,u]/(xy-u^b) & \longrightarrow &  \CC[x,y,s,t]/(xy-(st)^b) 
\nonumber \\
 x &\longmapsto & x 
\nonumber \\
 y &\longmapsto & y 
\nonumber \\
u &\longmapsto & st 
\nonumber
\end{eqnarray}
where $\varphi$ is the isomorphism defined in \eqref{xyu} and $\widetilde{\varphi}$ is the isomorphism defined in \eqref{xyst}. Thus, $\tilde{\pi}:Y\to \Spec\mathbb{C}[s,t]$ is obtained from $\pi: \widetilde{T}\to \Spec\CC[u]$ by the base change  $u \mapsto st$.
\end{proof}
\end{theorem}

\section{Tropical Corals on $\trcone \RR$}
\label{A tropical counting problem}
\subsection{Tropical corals}
\label{subsec: tropical corals}
Throughout this section we fix $N=\ZZ^2$, and denote by $N_{0}$, $N_{>0}$ and $N_{<0}$, the set of elements $n\in N$ with  $\pr_2 (n) = 0$, $\pr_2 (n) > 0$, and $\pr_2 (n) < 0$ respectively, where $\pr_2$ is the projection map onto the second component. 
We also set the following conventions. Given a simplicial complex $\bar{\Gamma}$, we denote by
\begin{center}
  \begin{tabular}{ll}
    $V(\bar{\Gamma})$ & the set of vertices of $\bar{\Gamma}$ \\
  $V_k(\bar{\Gamma})$ & the subset of vertices of valency $k$  in $V(\bar{\Gamma})$ \\
    $E(\bar{\Gamma})$ & the set of edges of $\bar{\Gamma}$ \\
    $\partial e$ & the set of vertices adjacent to an edge $e\in E(\bar{\Gamma})$.
  \end{tabular}
\end{center}
\noindent A \emph{bilateral graph} is a finite, connected $1$-dimensional simplicial complex $\bar{\Gamma}$ such that:
\begin{itemize}
\item[(i)] $\bar{\Gamma}$ has no divalent vertices.
\item[(ii)]  There is a partition
  \[ V(\bar{\Gamma})= V^{+}(\bar{\Gamma}) ~ \textstyle\amalg V^{0}(\bar{\Gamma}) ~ \textstyle\amalg ~ V^{-}(\bar{\Gamma}) \]
  of the set of vertices into \emph{positive vertices} $V^{+}(\bar{\Gamma})$,
  \emph{interior vertices} $V^0(\bar{\Gamma})$, and \emph{negative vertices}  $V^{-}(\bar{\Gamma})$
  We assume that $V^{+}(\bar{\Gamma})$ and $V^{-}(\bar{\Gamma})$ are of non-empty.
\item[(iii)] All positive vertices are univalent, and no interior vertices are univalent:
  \[V_1(\bar{\Gamma})=  V^{-}_1(\bar{\Gamma}) \amalg V^{+}(\bar{\Gamma})  \]
\item[(iv)] Denoting by $\ell+1$ the cardinality of $V^+(\bar{\Gamma})$, we have a fixed labeling 
\begin{equation} v_0^+,\dots,v_\ell^+
\nonumber
\end{equation} 
of the positive vertices.
Denoting by $m$ the cardinality of 
$V^-(\bar{\Gamma})$, we have a fixed labeling 
\begin{equation} 
v_1^-,\dots, v_m^-
\nonumber
\end{equation}of the negative vertices.
\item[(v)] The first Betti number of $\bar{\Gamma}$ is zero.
\end{itemize}

In the remaining part of this section, we omit the case where the cardinalities of both $V_1^{-}(\bar{\Gamma})$ and of $E(\bar{\Gamma})$ are one, as this case can be treated easily in all the arguments that are used.
\begin{definition}
\label{Def: coral graph}
A \emph{coral graph} $\Gamma$ is the geometric realization (i.e. the underlying topological space \cite[Defn.1.39]{DK}) $|\bar{\Gamma}|$ of a bilateral graph $\bar{\Gamma}$, with positive vertices removed:
\[
\Gamma :=|\bar{\Gamma}|\setminus V^{+}(\bar{\Gamma})\]
\end{definition}
We call an edge $e$ a \emph{half-edge} if $\partial e$ is a single vertex. Note that a coral graph $\Gamma$ contains half-edges
\[ E^{+}(\Gamma):= \{ e ~|~  \text{$e = e^{+} \setminus \{v^+\}$ where $e^{+}\in E(\bar{\Gamma})$, $v^+\in V^{+}(\bar{\Gamma})$ such that  $\partial e^{+}=\{v^+\}$} \}\]
referred to as the set of \emph{positive edges} or \emph{unbounded edges} of $\Gamma$.
The fixed labeling 
of the positive vertices induces a fixed labeling 
\begin{equation}\label{positive_labelling}
\nonumber
e_0^+,\dots,e_\ell^+
\end{equation}
of the positive edges.


The set 
\[ E^b(\Gamma):= E(\Gamma) \setminus E^{+}(\Gamma) \]
is referred to as the set of \emph{bounded edges} of $\Gamma$, and the set 
\[ V^{-}(\Gamma):= V^{-}(\bar{\Gamma}) \] 
is referred to as the set of \emph{negative vertices} of $\Gamma$. 
A \emph{weighted graph} $\Gamma$ is a graph endowed with a \emph{weight function} $w_{\Gamma}: E(\Gamma)\to \mathbb{N}\setminus
\{0\}$. The image $w_\Gamma(e)$ of $e\in E(\Gamma)$ under $w_\Gamma$ is referred to as the \emph{weight of} the edge $e$. 
\begin{definition}
\label{parameterized tropical coral}
A \emph{(parameterized) tropical coral} in $\overline{C}\RR$ is a \emph{proper} map 
\[h:\Gamma \to \overline{C}\RR\,,\]
where $\Gamma$ is a weighted coral graph,
satisfying the following:
\begin{enumerate}
\item[(i)] For all $e \in E(\Gamma)$, the restriction $h|_e$ is an embedding and $h(E)$ is contained in an integral affine submanifold of $\trcone \RR$.
\item[(ii)] For all $v \in V^{0}(\bar{\Gamma})$; 
\[h(v)\in \overline{C}\RR \setminus \partial \overline{C}\RR\]
where $\partial \overline{C}\RR$ denotes the boundary of the truncated cone $\overline{C}\RR$. Moreover, the following balancing condition holds: 
\begin{eqnarray}\label{balancing condition}
\nonumber
\sum_{j=1}^k w_\Gamma(e_j) u_j = 0
\end{eqnarray}
where $\{e_1,\ldots,e_k\} \subset E(\Gamma)$ is the set of edges adjacent to $v$, $w_\Gamma(e_j)\in \NN\setminus \{ 0\}$ is the weight on $e_j$ and $u_j\in N$ is the primitive integral vector emanating from $h(v)$ in the direction of $h(e_j)$ for $j=1,\ldots,k$.
\item[(iii)] For all $v \in V^{-}(\Gamma)$, we have $h(v) \in \partial\overline{C}\RR$
and there exists $w_v\in \NN\setminus \{0\}$ associated to $v$, referred to as the \emph{weight on $v$}, such that the following balancing condition holds: 
\begin{eqnarray}\label{negative balancing}
\nonumber
w_v\cdot u_v+\sum_{j=1}^k w_\Gamma(e_j) u_j = 0
\end{eqnarray} 
where $\{e_1,\ldots,e_k\} \subset E(\Gamma)$ is the set of edges adjacent to $v$ and $u_v\in N$ is the primitive integral vector emanating from $h(v)$ in the direction of the origin in $N_\RR$.
\item[(iv)] For all 
$e \in E^+(\Gamma)$, the restriction to $h(e)$ of the projection map 
\[{\pr}_2:\trcone \RR \rightarrow [1,\infty)\] 
onto the second factor is proper. 
\item[(v)] For all $v \in V(\Gamma)$; $h(v) \in N_{\QQ}=N\otimes_\ZZ \QQ$, where $\QQ$ is the set of rational numbers.
\end{enumerate}
We illustrate the image of a tropical coral in Figure \ref{ExCoral}. An \emph{isomorphism} of tropical corals $h: \Gamma \to \overline{C}\RR$ and
$h': \Gamma' \to \overline{C}\RR$ is a homeomorphism $\Phi: \Gamma \to \Gamma'$
respecting the weights of the edges and such that $h = h' \circ
\Phi$. A \emph{tropical coral} is an isomorphism class of
parameterized tropical corals. 

\end{definition}
The definition of a tropical coral can be generalized to tropical corals in $\trcone B$, for any integral affine manifold $B$. Indeed in \S\ref{Taking the quotient}, to study the invariants of the Tate curve, we shall apply the quotient given by the $\ZZ$-action on $(\RR, \P_b)$ and work over 
\[B:=\trcone (\RR/\ZZ) =\trcone S^1 \] 
Then, condition (iv) of the definition \ref{parameterized tropical coral} ensures that there is no infinite wrapping of the unbounded edges of tropical corals in $\trcone S^1$. 
\begin{definition}
\label{def_general}
We call a tropical coral $h:\Gamma \rightarrow \trcone \RR$ \emph{general} if all interior vertices of $\Gamma$ are trivalent and all negative vertices are univalent.
Otherwise it is called \emph{degenerate}. 
\end{definition}
\subsection{Incidences for tropical corals}
\label{Incidences for tropical corals}
\begin{definition}
\label{type of a coral graph}
For $\Gamma$ a coral graph, we denote by
\[F(\Gamma)=\{(v,e)\,|\, e\in E(\Gamma)~\mathrm{and}~v\in\partial e\}\] the set of \emph{flags} of $\Gamma$. 

A \emph{type} is a pair $(\Gamma,u)$ consisting of a weighted coral graph $\Gamma$ and a map
\begin{align*}
u:F(\Gamma)~ \amalg ~ V^{-}(\Gamma)& \lra N \\
F(\Gamma) \ni (v,e)  & \longmapsto  u_{v,e}\\
V^{-}(\Gamma) \ni v  &  \longmapsto  u_v
\end{align*}
where $u_v$ and $u_{v,e}$ are primitive integral vectors in $N$.
\end{definition}
\begin{definition}
\label{type of a tropical coral}
The \emph{type of a tropical coral} $h: \Gamma \to \trcone \RR$ is the
type $(\Gamma, u)$ where the map $u:F(\Gamma) ~ \amalg ~ V^{-}(\Gamma) \to N$ is given by assigning to each $(v, e)\in F(\Gamma)$ the primitive integral vector $u_{v,e} \in N$ emanating from $h(v)$ in the direction of $h(e)$ and assigning to each $v\in V^{-}(\Gamma)$ the primitive integral vector $u_v\in N$ emanating from $h(v)$ in the direction of the origin. We denote the set of tropical corals of type $(\Gamma,u)$ by
$\T_{(\Gamma,u)}$.
\end{definition}
To set up our counting problem for tropical corals, we define tropical incidence conditions, given by \emph{degree} and \emph{asymptotic constraints} defined as follows. 

\begin{definition}
\label{degree}
A \emph{degree} with $\ell+1$ positive and $m$ negative entries is a tuple 
\[\Delta:=(\overline{\Delta},\underline{\Delta})\in N^{\ell+1}\times N^m\]
of elements in the lattice $N$, where
$\overline{\Delta}=(\overline{\Delta}^0,\ldots,\overline{\Delta}^\ell)$, and $\underline{\Delta}=(\underline{\Delta}^1,\ldots,\underline{\Delta}^m)$, with
\begin{align}
\pr_2(\overline{\Delta}^i) & >  0
\nonumber \\
\pr_2(\underline{\Delta}^j) & < 0
\nonumber
\end{align}
for all $0\leq i\leq \ell$ and $1\leq j\leq m$, 
where $\pr_2:N= \ZZ \oplus \ZZ\to \ZZ$ is the projection onto the second factor.
\end{definition}

\begin{definition}
\label{tropical degree}
Let $(\Gamma,u)$ be a type where $\Gamma$
has $\ell+1$ positive edges and $m$ negative vertices.
The \emph{degree of $(\Gamma,u)$} is the degree 
\[\Delta=(\overline{\Delta},\underline{\Delta})\in N^{\ell+1}\times N^m\]
where, for every $0 \leq i \leq \ell$,
denoting $\partial e_i^+=\{v_i\}$,
\[ \overline{\Delta}^i := w_\Gamma(e_i^+) \cdot u_{v_i,e_i^+}\]
and, for every 
$1 \leq j \leq m$,
\[ \underline{\Delta}^j:=w_{v_j^-} \cdot u_{v_j^-}\,.\]

\end{definition}
\begin{definition}
\label{degree of a tropical coral}
The \emph{degree of a tropical coral} $h:\Gamma \to \trcone \RR$ is the degree of its type. 
\end{definition}




\begin{definition}
\label{asymptotic constraint}
Let $\Delta$ be a degree with $\ell+1$
positive and $m$ negative entries.
An \emph{asymptotic constraint}
for $\Delta$ is a $\ell$-tuple 
\[ \lambda=(\lambda_1,\ldots,\lambda_\ell)\in \prod_{i=1}^{\ell} N_{\QQ}/ (\QQ \cdot \overline{\Delta}^i)\,.\]


\end{definition}

\begin{definition}
\label{trop_coral_match}
Let $\Delta$ be a degree with 
$\ell+1$ positive entries, and $\lambda$ an asymptotic constraint for $\Delta$.
We say that a tropical coral 
$h: \Gamma \rightarrow \overline{C}\RR$
\emph{matches} 
$(\Delta,\lambda)$ if the degree of $h: \Gamma \rightarrow \overline{C}\RR$ is 
$\Delta$, and for every $1 \leq i\leq \ell$, 
the image of $h(e_i^+)$ under 
the quotient map $N_{\QQ}\lra N_{\QQ}/\QQ \cdot \overline{\Delta}^i$ is $\lambda_i$.

\end{definition}

\begin{remark}
An asymptotic constraint $\lambda$ constrains the positive edges $e_1^+,\dots,e_\ell^+$, but not the positive edge $e_0^+$.
\end{remark}

\begin{definition}
Let $\Delta$ be a degree. We say that an
an asymptotic constraint $\lambda$ for $\Delta$
is \emph{general} if every tropical coral  matching  $(\Delta,\lambda)$  is general as in Definition \ref{def_general}. 

\end{definition}
For a type $(\Gamma,u)$ of degree $\Delta$, 
and  $\lambda$ a general asymptotic constraint for $\Delta$, 
we denote the set of tropical corals of type $(\Gamma,u)$ matching $(\Delta,\lambda)$ by $\T_{(\Gamma,u)}(\lambda)$.
\begin{remark}
\label{existence of general constraints}
A non-general tropical coral can always be deformed into a general tropical coral analogously to the case of tropical curves \cite[\S2]{M}. This is possible since non-general tropical corals are obtained by taking the limit of the lengths of some edges in general tropical corals to zero. It follows that the non-general tropical corals of a given type $(\Gamma,u)$ form a lower dimensional strata of the moduli space $\mathcal{\T}_{(\Gamma,u)}$ \cite[Prop $2.14$]{M}. Hence, the types of the non-general corals form a nowhere dense subset in the space of constraints. This ensures the existence of general asymptotic constraints.
\end{remark}

\subsection{Extending a tropical coral to a tropical curve}
\label{Extending a tropical coral to a tropical curve}

We first explain how the notions of degree and asymptotic constraints introduced in the previous sections for tropical corals can be interpreted as degree and asymptotic constraints for tropical curves as in \cite{NS}.

\begin{definition}
\label{trop_degree}
Let $\Delta$ be a degree as in Definition \ref{degree}. Let $\widetilde{h}:\widetilde{\Gamma} \to \RR^2$
be a tropical curve as in \cite{NS}, with 
$(\ell+1+m)$-unbounded edges which are labelled 
$\tilde{e}_0^+,\dots, \tilde{e}_\ell^+, \tilde{e}_1^-, \dots,\tilde{e}_m^-$.
We say that the tropical curve $\widetilde{h}:\widetilde{\Gamma} \to \RR^2$
is of degree $\Delta$ if 
for every $0 \leq i\leq \ell$, denoting 
$\partial \tilde{e}_i^+=\{\tilde{v}_i\}$, 
\[ \overline{\Delta}^i=w_{\tilde{\Gamma}}(\tilde{e}_i^+) \cdot u_{\tilde{v}_i, \tilde{e}_i^+}\]
and for every $1 \leq j \leq m$, 
denoting 
$\partial \tilde{e}_j^-=\{\tilde{v}_j\}$, 
\[ \underline{\Delta}^j=w_{\tilde{\Gamma}}(\tilde{e}_j^-) \cdot u_{\tilde{v}_j, \tilde{e}_j^-}\,.\]
\end{definition}

\begin{definition}
\label{trop_const}
Let $\Delta$ be a degree and 
$\lambda$ an asymptotic constraint for 
$\Delta$ as in Definition \ref{asymptotic constraint}. Let $\widetilde{h}:\widetilde{\Gamma} \to \RR^2$
be a tropical curve as in \cite{NS}, with 
$(\ell+1+m)$-unbounded edges which are labelled 
$\tilde{e}_0^+,\dots, \tilde{e}_\ell^+, \tilde{e}_1^-, \dots,\tilde{e}_m^-$.
We say that $\widetilde{h}:\widetilde{\Gamma} \to \RR^2$ \emph{matches} $(\Delta,\lambda)$
if it is of degree $\Delta$, and:
\begin{itemize}
    \item[(i)] for every $1 \leq i\leq \ell$, the image of $\tilde{h}(\tilde{e}_i^+)$ under the quotient map $N_\QQ \rightarrow N_\QQ/\QQ \cdot \overline{\Delta}^i$ is $\lambda_i$,
    \item[(ii)] for every $1 \leq j\leq m$, the image of $\tilde{h}(\tilde{e}_j^-)$ under the quotient map $N_\QQ \rightarrow N_\QQ/\QQ \cdot \underline{\Delta}^j$ is $0$.
\end{itemize}
\end{definition}

We then explain how to extend tropical corals to tropical curves.

\begin{definition}
We define the \textit{extension} of a tropical coral $h:\Gamma\to \trcone \RR$ as the 
following tropical curve $\widetilde{h}:\widetilde{\Gamma} \to \RR^2$
in the sense of \cite{NS}: 
\begin{itemize}
\item[(i)]$\widetilde{\Gamma}$ is obtained from 
$\Gamma$ by first adding for each negative vertex 
$v_j^- \in V^-(\Gamma)$ an unbounded edge $E_{v_j^-}$ incident to $v_j^-$ and then by adding a $2$-valent vertex $V_{v_j^-}$ on each $E_{v_j^-}$,
\item[(ii)]the weight function is extended to 
$\tilde{\Gamma}$ by $w_{\Tilde{\Gamma}}(E_{v_j^-})=w_{v_j^-}$, where 
$w_{v_j^-}$ is the weight on $v_j^-$ as in
Definition \ref{parameterized tropical coral}(iii), 
\item[(iii)] $\tilde{h}$ extends $h$ by mapping
$E_{v_j^-}$ onto the half-line emanating from $h(v_j^-)$ and passing through the origin $0$ in 
$\RR^2$, in such a way that $h(V_{v_j^-})=0$. 
\end{itemize}
\end{definition}
Note that the primitive integral direction of $h(E_{v_j^-})$
emanating from $h(v_j^-)$ is $u_{v_j^-}$
as in Definition \ref{parameterized tropical coral}(iii), and so the balancing condition for the tropical curve $\widetilde{h}:\widetilde{\Gamma} \to \RR^2$
at the vertex $v_j^-$ follows from the balancing condition in 
Definition \ref{parameterized tropical coral}(iii) for the tropical coral $h:\Gamma\to \trcone \RR$.
For details of this construction see \S $2.3$ of \cite{MyThesis}. 
The tropical curve 
$\tilde{h} \colon \tilde{\Gamma} \rightarrow \RR^2$ contains $(\ell+1+m)$-unbounded edges,
that we label 
$\tilde{e}_0^+,\dots, \tilde{e}_\ell^+, \tilde{e}_1^-, \dots,\tilde{e}_m^-$, where 
$\tilde{e}_i^+:=e_i^+$ for $0 \leq i\leq \ell$ and $\tilde{e}_j^-=E_{v_j^-}$ for 
$1 \leq j \leq m$. By construction, the tropical curve $\tilde{h} \colon \tilde{\Gamma} \rightarrow \RR^2$ matches 
$(\Delta, \lambda)$ as in Definition \ref{trop_const}.


\begin{example}
In Figure \ref{ExtensionOfTheCoral}, we illustrate a tropical coral on the truncated cone $\trcone \RR$ and its extension to a tropical curve in $\RR^2$.
\begin{figure}
\resizebox{.8\linewidth}{!}{
\input{ExtensionOfTheCoral.pspdftex}}
\caption{A tropical coral and its extension.}
\label{ExtensionOfTheCoral}
\end{figure}
\end{example}

\subsection{Tropical corals of a fixed type}
\label{The space of general tropical corals of fixed type}
In this section, we show that the space $\mathcal{T}_{(\Gamma,u)}$ of isomorphism classes of general tropical corals of given type $(\Gamma,u)$ forms a convex polyhedron of dimension $l-1$ where $l$ is the number of unbounded edges of $\Gamma$. 

\begin{proposition}
\label{EmbeddingOfT}
The set of general tropical corals of type $(\Gamma,u)$ is embedded into $\RR^{l-1}_{>0}$, where $l$ is the number of unbounded edges of $\Gamma$.
\end{proposition}
\begin{proof}
Let $(h:\Gamma\rightarrow \trcone \RR)\in \T_{(\Gamma,u)}$ be a general tropical coral of type $(\Gamma,u)$ with $l$ unbounded edges. Assume $\Gamma$ has $m$ negative vertices. Since $h$ is general, $\Gamma$ has $m+l-2$ interior vertices and $l-2$ among them, which we will denote by
\[ \{ v_1^0,\ldots,v_{l-2}^0  \} \subset V^0(\Gamma)\]
are not adjacent to any negative vertex of $\Gamma$. 

For any vertex $v\in V^0(\Gamma)$, let $\rho(v)=\pr_2(h(v))$,
where $\pr_2:\RR^2 \lra \RR$ is the projection map onto the second factor. Fix an interior vertex 
\[v\in V^0(\Gamma) \setminus \{ v_1^0,\ldots,v_{l-2}^0  \} \] 
which is adjacent to a negative vertex, and define
\begin{eqnarray}
\label{Eq:EmbeddingPhi}
\Phi: \T_{(\Gamma,u)} &\hookrightarrow&  \RR_{>0}^{l-1} 
\nonumber \\
h & \mapsto & (\rho (v), \rho(v_1^0),\ldots,\rho (v_{l-2}^0) )
\nonumber
\end{eqnarray}
We will show that fixing $(\rho (v), \rho(v_1^0),\ldots,\rho (v_{l-2}^0) ) \in \RR_{>0}^{l-1} $ determines $h\in \T_{(\Gamma,u)}$ uniquely. First, observe that the positions of images of all negative vertices $V^-(\Gamma)$ under $h$ are fixed by the degree of $(\Gamma,u)$. The position of $h(v)$ is also fixed since $v$ is connected to a negative vertex $v^-\in V^-(\Gamma)$, and we also had fixed the type of $\Gamma$. Proceeding inductively, we can determine the images of all vertices of $\Gamma$ under $h$. Hence, $\Phi$ is injective, and the result follows. 
\end{proof}
We next describe how to obtain a tropical coral by a gluing construction of \emph{coral blocks}, which are more general objects than tropical corals, defined as follows.
\begin{definition}
Let $\Gamma_{m,l}$ be a coral graph with $m$ negative vertices and $l$ positive edges.
A \emph{coral block} is a proper map $h:\Gamma_{m,l}\rightarrow \RR^2$ which satisfies all conditions of Definition \ref{parameterized tropical coral}, except Item $(iv)$. We furthermore require $0 \notin h(e)~\mathrm{ ~for~ any}~ e\in E^{+}(\Gamma)$, where $0$ denotes the origin in $\RR^2$.
\end{definition}
We call a coral block $h:\Gamma_{m,l} \rightarrow \RR^2$ \emph{general} if all vertices $v\in V^{0}(\Gamma_{m,l})$ are trivalent, and all vertices $v\in V^-(\Gamma_{m,l})$ are univalent. Note that tropical corals are particular types of coral blocks, in which the image of all unbounded edges lie in $\trcone \RR$ and the projection of each unbounded edge onto the second factor is proper. 

The \emph{type} of a coral block $h:\Gamma_{m,l} \rightarrow \RR^2$ analogously as the type of a tropical coral and denoted by $(\Gamma_{m,l},u)$. The set of isomorphism classes of coral blocks of a given type $(\Gamma_{m,l},u)$ is denoted by ${\mathcal{\T}_{(\Gamma_{m,l},u)}}$. 
\begin{construction}
\label{constr:gluing}
We describe the gluing of coral blocks as follows. Let 
\begin{eqnarray}
(h_1:\Gamma_{m_1,l_1}^1 \rightarrow \RR^2) & \in & \mathcal{\T}_{(\Gamma_{m_1,l_1},u_1)}
\nonumber \\
(h_2:\Gamma_{m_2,l_2}^2 \rightarrow \RR^2) & \in & \mathcal{\T}_{(\Gamma_{m_2,l_2},u_2)}
\nonumber
\end{eqnarray}
be two coral blocks, and assume there exists edges
\begin{eqnarray}
e_1 & \in & E^{+} (\Gamma_{m_1,l_1})~\mathrm{such~that}~e_1~\mathrm{is~adjacent~to}~v_1\in V^{0}(\Gamma_{m_1,l_1})
\nonumber \\
e_2 & \in & E^{+} (\Gamma_{m_2,l_2})~\mathrm{such~that}~e_2~\mathrm{is~adjacent~to}~v_2\in V^{0}(\Gamma_{m_2,l_2})
\nonumber
\end{eqnarray}
Let
\[ w_1:E(\Gamma_{m_1,l_1}) \rightarrow \NN\setminus \{0\}
\,\ \,\ \mathrm{and} \,\ \,\
w_2:E(\Gamma_{m_2,l_2})  \rightarrow  \NN\setminus \{0\}  \] 
be the weight functions on edges of $\Gamma_{m_1,l_1}$ and $\Gamma_{m_2,l_2}$ respectively, and assume we have  $w_1(e_1)=w_2(e_2)$. Assume furthermore that
\[u_{e_1}=-u_{e_2}\]
where $u_{e_i}$ denotes the primitive integral vector emanating from $v_i$ in the direction of $e_i$. Let
\[ N_{\RR}/\RR u:= N_{\RR}/\RR u_{e_1}=N_{\RR}/\RR u_{e_2} \]
Then, for $i\in \{1,2\}$ the maps
\begin{align}
f_i:\mathcal{\T}_{(\Gamma_{m_1,l_1},u_1)} &\rightarrow N_{\RR}/\RR u_{e_1}
\nonumber \\
h_i &\mapsto [h_i(v_i)] 
\nonumber
\end{align}
induce the map
\begin{align}
f:\mathcal{\T}_{(\Gamma_{m_1,l_1},u_1)} \times \mathcal{\T}_{(\Gamma_{m_2,l_2},u_2)} &\rightarrow N_{\RR}/\RR u
\nonumber \\
(h_1,h_2) &\mapsto [h_1(v_1)-h_2(v_2)]. 
\nonumber
\end{align}
Assume 
\[f(h_1,h_2)= 0 \in N_{\RR}/\RR u\] and
\[ h(v_1)-h(v_2)=\lambda\cdot u_1 ~ \mathrm{for} ~ \lambda\in\RR_{>0} \]
Then we can define a glued coral block $h_{12}:\Gamma_{m,l}\rightarrow \RR^2$ as follows. Define the vertex set $V(\Gamma_{m,l})$ as the disjoint union
\[V(\Gamma_{m,l}):=V(\Gamma_{1})\amalg V(\Gamma_{2})  \]
and the edge set $E(\Gamma_{m,l})$ as
\[E(\Gamma_{m,l}):=E(\Gamma_{1})\setminus \{ e_1 \}\amalg E(\Gamma_{2}) \setminus \{ e_2 \} \amalg \{e_{12}\} \]
where $e_{12}$ is the edge such that $\partial e_{12}=\{ v_1,v_2 \}$. Define $E_{12}$ to be the line segment in $\trcone \RR$ such that $\partial E_{12}=\{ h(v_1),h(v_2)\}$. 
Define the weight function 
$w_{12}:E(\Gamma_{m,l}) \rightarrow \NN\setminus \{0\}$
by 
\[ w_{12}:= \begin{cases} 
      w_1 & \mathrm{on}~E(\Gamma_{m_1,l_1}^1) \setminus \{e_1\} \\
      w_1(e_1)=w_2(e_2) & \mathrm{on}~ e_{12} \\
     w_2 & \mathrm{on}~ E(\Gamma_{m_2,l_2}^2) \setminus \{e_2\}
   \end{cases}
\]
Define the coral block $h_{12}:\Gamma_{m,l}\rightarrow \RR^2$ by
\[ h_{12}:= \begin{cases} 
      h_1 & \mathrm{on}~ V(\Gamma_{m_1,l_1}^1)\cup E(\Gamma_{m_1,l_1}^1) \setminus \{e_1\} \\
      E_{12} & \mathrm{on}~ e_{12} \\
     h_2 & \mathrm{on}~ V(\Gamma_{m_2,l_2}^2)\cup E(\Gamma_{m_2,l_2}^2) \setminus \{e_2\} 
   \end{cases}
\]
We say  $h_{12}:\Gamma_{m,l}\rightarrow \RR^2$ is obtained by \emph{gluing} $h_1$ and $h_2$ along the edges $e_1$ and $e_2$. 
\end{construction}
The following is an immediate corollary of Construction \ref{constr:gluing}.
\begin{lemma}
\label{glueing of bouquets}
Any coral block $(h_{12}:\Gamma_{m,l}\rightarrow\RR^2)\in \mathcal{\T}_{(\Gamma_{m,l},u)}$ where $m>1$,  can be obtained by gluing coral blocks
\[ h_1\in \mathcal{\T}_{(\Gamma_{m_1,l_1},u_1)} \,\ \,\ \mathrm{and} \,\ \,\  h_2\in \mathcal{\T}_{(\Gamma_{m_2,l_2},u_2)} \]
where
\[m=m_1+m_2 \,\ \,\ \mathrm{and} \,\ \,\ l=(l_1-1)+(l_2-1)=l_1+l_2-2.\] 
\end{lemma}
\begin{figure}
\resizebox{.9\linewidth}{!}{
\input{buke.pspdftex}}
	\caption{A tropical coral obtained by gluing two coral blocks along the labelled edges}
\label{Fig:CoralBlocks}
\end{figure}
Now we are ready to prove the following main theorem of this section.
\begin{theorem}
\label{coral moduli space is polyhedral}
Let $(\Gamma,u)$ be a general type of tropical corals of fixed degree $\Delta$ with $l$ unbounded edges
\[ e_1,\ldots,e_l  \]
with $ \partial e_i= v_i $ for $v_i\in V(\Gamma )$ \footnote{Not all $v_i$ may be distinct, repetitions are allowed}. Let $u_i$ be the primitive integral vector in $N_\RR$ emanating from $v_i$ in the direction of $h(e_i)$. Assume $\mathcal{\T}_{(\Gamma,u)}$ is non-empty. Then 
for any sequence of indices $1\le i_1<\ldots<i_k\le l$ with $k\le l-1$ the map
\begin{eqnarray}
\ev_{i_1,\ldots,i_k}: \mathcal{\T}_{(\Gamma,u)} & \lra & \prod_{\mu=1}^{k} N_\RR/ (\RR\cdot u_{i_\mu})
\nonumber \\
h & \longmapsto & \big( [h(v_{i_1})],\ldots, [h(v_{i_k})]\big)
\end{eqnarray}
is an integral affine submersion.
\end{theorem}
\begin{proof}
We will prove the theorem for coral blocks of general type $(\Gamma_{m,l},u)$ of degree $\Delta$. Since a tropical coral is a special type of a coral block the result will follow.   
By Lemma \ref{glueing of bouquets}, any coral block  \[(h:\Gamma_{m,l}\rightarrow\RR^2)\in \mathcal{T}_{(\Gamma_{m,l},u)}\] is obtained by gluing two coral blocks $h_1\in \mathcal{\T}_{(\Gamma_{m_1,l_1},u_1)}$ and $h_2\in \mathcal{\T}_{(\Gamma_{m_2,l_2},u_2)}$
along edges $e_1\in  E^{+}(\Gamma_{m_1,l_1},u_1)$ with $ \partial e_1=v_1$ and $e_2\in  E^{+}(\Gamma_{m_2,l_2},u_2)$ with $\partial e_2=v_2$ so that 
\[ E(\Gamma_{m,l},u)=E(\Gamma_{m_1,l_1},u_1)\setminus \{e_1\} \amalg E(\Gamma_{m_2,l_2},u_2)\setminus \{e_2\} \amalg {e} \]
with $\partial e= \{v_1,v_2\}$. Let
\begin{eqnarray}
\{e_1\} \cup \{e_{i_r},\ldots, e_{i_r} ~ | ~ r \leq l_1-1 \} & \subseteq & E^{+}(\Gamma_{m_1,l_1},u_1)  
\nonumber \\
\{e_2\} \cup \{e_{i_{r+1}},\ldots, e_{i_k} ~ | ~ k-r \leq l_2-1 \} & \subseteq & E^{+}(\Gamma_{m_2,l_2},u_2)  \} 
\nonumber
\end{eqnarray}
We will use induction on $l$. For $l=1$, we need to have a unique negative vertex. Let $v$ be the negative vertex and $e\in E(\Gamma_{m,l})$ be the edge with $\partial e=v$. Extend $e$, to obtain the tropical curve $\widetilde{h}:\widetilde{\Gamma_l^m}\rightarrow \RR^2$. In this case the result follows from \cite[Proposition $2.14$]{M}, or \cite[Proposition $2.4$]{NS}. 

Assume the theorem is true for any $2 \leq l_i < l$. Let $u_i$ be the direction vector for $e_i$, that is the primitive integral vector emanating from $h_i(v_i)$ in the direction of $h_i(e_i)$ for $i=1,2$. Define the direction vectors $u_{i_\mu}$ for the unbounded edges $e_{i_\mu}$ analogously. Then, by the induction hypothesis we have submersions
\begin{eqnarray}
 \mathcal{T}_{(\Gamma_{m_1,l_1},u_1)} & \lra &   \prod_{\mu=1}^{r} N_\RR/ (\RR\cdot u_{i_\mu})
\nonumber \\
h_1 &\longmapsto & ([h_1(v_{i_1})],\ldots, [h_1(v_{i_{r}})] )
\nonumber\\
 \mathcal{T}_{(\Gamma_{m_2,l_2},u_2)} & \lra &  N_\RR/\RR\cdot u_{2}  \times \prod_{\mu=r+1}^{k} N_\RR/ (\RR\cdot u_{i_\mu})
\nonumber \\
h_2 &\longmapsto & \big([h_2(v_2)],([h_2(v_{i_{r+1}})],\ldots, [h_2(v_{i_{k}})] ) \big)
\nonumber
\end{eqnarray}
Hence, we obtain a submersion
\begin{eqnarray}
\mathcal{F}  : \mathcal{\T}_{(\Gamma_{m_1,l_1},u_1)}\times_{N_\RR / (\RR \cdot u)}  \mathcal{\T}_{(\Gamma_{m_2,l_2},u_2)} & \lra &  \prod_{\mu=1}^{r} N_\RR/ (\RR\cdot u_{i_\mu}) \times \prod_{\mu=r+1}^{k-1} N_\RR/ (\RR\cdot u_{i_\mu}) 
\nonumber \\
(h_1,h_2) & \longmapsto & \big( [h_1(v_{i_1})],\ldots, [h_1(v_{i_{r}})],[h_2(v_{i_{r+1}})],\ldots, [h_2(v_{i_{k-1}})] \big)
\nonumber
\end{eqnarray}
where
\[ N_{\RR}/ (\RR \cdot u) := N_{\RR}/ (\RR \cdot u_1 )=N_{\RR}/ (\RR \cdot u_2) \]
and the fibered coproduct is defined via the morphisms $f_i$ in Construction \ref{constr:gluing}. Define
\begin{eqnarray}
 \mathcal{G} :   \mathcal{\T}_{(\Gamma_{m_1,l_1},u_1)}\times_{N_\RR / \RR u }  \mathcal{\T}_{(\Gamma_{m_2,l_2},u_2)} & \lra & \RR
 \nonumber \\
 (h_1,h_2) & \lra & \lambda 
\nonumber
\end{eqnarray}
where $\lambda \in \RR$ is defined by 
\[ h(v_2)-h(v_2)=\lambda \cdot u_{e_1} \]
Then, by the construction of gluing of coral blocks we obtain 
\begin{equation}
\label{PositiveDistance}
 \mathcal{\T}_{(\Gamma_{m,l},u)} = \mathcal{G}^{-1}(\RR_{> 0}) 
\end{equation}
Hence, the inclusion $ \mathcal{G}^{-1}(\RR_{> 0}) \subset  \mathcal{\T}_{(\Gamma_{m_1,l_1},u_1)}\times_{N_\RR / \RR u}  \mathcal{\T}_{(\Gamma_{m_2,l_2},u_2)}$ followed by the submersion $\mathcal{F}$ gives the desired submersion $\ev_{i_1,\ldots,i_k}$.
\end{proof}

\begin{corollary}
\label{dimension count for corals}
The set $\T_{(\Gamma,u)}$ of isomorphism classes of general tropical corals of a given type $(\Gamma,u)$ forms the interior of a convex polyhedron of dimension $l-1$ where $l$ is the number of unbounded edges of $\Gamma$. 
\end{corollary}
\begin{proof}
We will first prove the result for the set $\T_{(\Gamma_{m,l},u)}$ of isomorphism classes of general coral blocks of a given type $(\Gamma_{m,l},u)$. The fact that $\T_{(\Gamma_{m,l},u)}$ is a convex polytope follows by the Equation \eqref{PositiveDistance} and the induction hypothesis. By Theorem \ref{coral moduli space is polyhedral} we obtain a submersion  
\[ \mathcal{F}|_{\mathcal{G}^{-1}(\RR_{>0})}: \mathcal{T}_{(\Gamma,u)}  \lra  \prod_{i=1}^{k} N_\RR/ (\RR\cdot u_i)
\]
where $1\leq k \leq l-1$. Hence, for $k=l-1$ the result follows. Since tropical coral is a special type of a coral block in which all unbounded edges have direction vectors satisfying condition (iv) in Definition \ref{parameterized tropical coral} the result follows.  
\end{proof}
\begin{remark}
\label{no dependence}
From Theorem \ref{coral moduli space is polyhedral} and Corollary \ref{dimension count for corals} it follows that there is no linear dependence among general asymptotic constraints. Hence, if 
the type $(\Gamma,u)$ is of degree $\Delta$
and $\lambda$ 
is a general asymptotic constraint for $\Delta$, then the set 
$\T_{(\Gamma,u)}(\lambda)$ of labelled tropical corals of type $(\Gamma,u)$ matching $(\Delta,\lambda)$
is either empty of a cardinality $1$.
\end{remark}

\subsection{The count of tropical corals}
\label{The count of tropical corals}
In this section we define the count of tropical corals matching 
$(\Delta,\lambda)$, where $\Delta$ is a degree as in Definition \ref{degree} and 
$\lambda$ is a general asymptotic constraint for $\Delta$ as in Definition \ref{asymptotic constraint}.
\begin{proposition}\label{finiteness of types}
There are only finitely many types of tropical corals of a fixed degree $\Delta$.
\end{proposition}
\begin{proof}
By \S \ref{Extending a tropical coral to a tropical curve}, every tropical coral $h:\Gamma \to \trcone \RR$ uniquely extends to a tropical curve $\widetilde{h}:\widetilde{\Gamma}\rightarrow \RR^2$ in the sense of \cite[\S1]{NS}.
The degree of $\widetilde{h}$
is $\Delta$ as in Definition \ref{trop_degree}. By Proposition $2.1$ in \cite{NS} there are only finitely many types of tropical curves of degree $\Delta$. Since $h$ is obtained by the restriction of $\widetilde{h}$, the result follows immediately.
\end{proof}

\begin{definition}
\label{GoodConstraint}
For $\Delta$ be a degree with 
$\ell+1$ positive entries.
We denote by $\mathcal{C}_{\Delta}
\subset N_\RR$ the cone spanned by 
$\overline{\Delta}^0,\dots,\overline{\Delta}^\ell$.
Let 
$\lambda=(\lambda_1,\dots,\lambda_\ell)$
be an asymptotic constraint for $\Delta$.
We say that $\lambda$ is \emph{good} if for every $1 \leq i\leq \ell$, 
\begin{equation}
\label{eq_good}
\lambda_i\in \Int (\mathcal{C}_{\Delta}/ (\RR \cdot \overline{\Delta}^i))\subset N_{\RR}/ (\RR \cdot \overline{\Delta}^i) \,.
\end{equation}
\end{definition}
\begin{remark}
If $\overline{\Delta}^i \in \Int(\mathcal{C}_\Delta)$, then 
$\mathcal{C}_{\Delta}/ (\RR \cdot \overline{\Delta}^i)=N_\RR/(\RR\cdot \overline{\Delta}^i)$ and so 
\eqref{eq_good} is automatically satisfied. But if
$\overline{\Delta}^i \in \partial (\mathcal{C}_\Delta)$, 
then $\mathcal{C}_{\Delta}/ (\RR \cdot \overline{\Delta}^i)$
is a half-space in $N_\RR/(\RR\cdot 
\overline{\Delta}^i)$ and so \eqref{eq_good} is a non-trivial condition.
\end{remark}

\begin{remark}
\label{good general exists}
For a given degree $\Delta$, 
general asymptotic constraints exist by Remark \ref{existence of general constraints}. Since, good constraints form an open set inside the set of constraints, the existence of general and good asymptotic constraints for 
$\Delta$ follows.
\end{remark}
Let $(\Gamma,u)$ be a type of degree $\Delta$ and $\lambda$ an asymptotic constraint for $\Delta$. Even if we assume that the set of tropical corals of type $(\Gamma,u)$ matching
$(\Delta,\lambda)$ is empty, after rescaling 
$\lambda$ by $s\in \QQ_{\geq 1}$,
it is possible to obtain a non-empty set of tropical corals of type $(\Gamma,u)$ matching $(\Delta,\lambda)$
as illustrated in Figure \ref{fig:rescale}.
\begin{figure}
	 \resizebox{.9\linewidth}{!}{\input{Rescale.pspdftex}}
\label{fig:rescale}
\caption{Tropical corals appearing after rescaling of asymptotic constraints}
\end{figure}
To obtain $\lambda$-independent counts of tropical corals,
we want to choose our constraints such that we also avoid the possibility of obtaining new tropical corals after rescaling.
This is done by the notion of stable range of constraints introduced in Definition \ref{stable range}.
\begin{lemma}
\label{stability}
Let $(\Gamma,u)$ be a type of degree $\Delta$
and $\lambda$ a general asymptotic constraint for $\Delta$.
Then one of the following holds.
\begin{itemize}
\item[(i)] $\forall s \in \QQ_{\geq 1}$, $\T_{(\Gamma,u)}(s \cdot \lambda)= \emptyset$.
\item[(ii)] $\exists s_0 \in \QQ_{\geq 1}$ such that $\forall s \geq s_0$, $|\T_{(\Gamma,u)}(s \cdot \lambda)|= 1$.
\end{itemize}
\end{lemma}
\begin{proof}
By Remark \ref{no dependence}, for all 
$s \in \QQ_{\geq 1}$, we have either $\T_{(\Gamma,u)}(s \cdot \lambda)= \emptyset$ or $|\T_{(\Gamma,u)}(s \cdot \lambda)|= 1$.
If there exists a tropical coral $h\in {\mathcal{T}}_{(\Gamma,u)}$ 
matching a general constraint $\lambda$, then the rescaled coral $s\cdot h$ with 
$s\geq 1$ matches $s \cdot \lambda$. This operation clearly does not change the type. Moreover, since $\lambda$ is general for $h$, $s \cdot \lambda$ is also general for $s\cdot h$. Hence, the result follows.
\end{proof}
\begin{definition}
\label{stable range}
Let $\Delta$ be a degree.
We define the \emph{stable range of constraint $\mathcal{S}_\Delta$} as the set of asymptotic constraints $\lambda$ for $\Delta$ such that:
\begin{itemize}
\item[(i)] $\lambda$ is a good general asymptotic constraint for $\Delta$.
\item[(ii)] For every type $(\Gamma,u)$ of degree $\Delta$, if $\T_{(\Gamma,u)}(\lambda)= \emptyset$, then $\T_{(\Gamma,u)}(s\cdot \lambda)= \emptyset$, $\forall s\in \QQ_{\geq 1}$.
\end{itemize}
\end{definition}

\begin{lemma}
\label{stable range non-empty}
For every degree $\Delta$, the stable range of constraints
$\mathcal{S}_{\Delta}$ is non-empty.
\end{lemma}

\begin{proof}
By Remark \ref{good general exists}, the set of good general asymptotic constraints for $\Delta$ is not empty. Fix 
a good general asymptotic constraint $\lambda$
for $\Delta$. 
By Proposition \ref{finiteness of types}, there are 
finitely many types of tropical corals $(\Gamma_1,u_1),\cdots,(\Gamma_n,u_n)$ 
of degree $\Delta$. 
By Lemma \ref{stability}, for each type  
$(\Gamma_i,u_i)$, either $\T_{(\Gamma,u)}( s \cdot \lambda)= \emptyset$ for all $s \in \QQ_{\geq 1}$ and in this case we set $s_i :=1$, or
there exists $s_i \in \QQ_{\geq 1}$ such that for every $s \geq s_i$, $|\T_{(\Gamma,u)}(s \cdot \lambda)|= 1$.
By construction, if we consider the maximum \[s_0=\mathrm{max}\{ s_i~|~i=1,\ldots,n \}\,,\]
we have $s_0 \cdot \lambda \in \mathcal{S}_{\Delta}$.
\end{proof}

\begin{definition}
\label{def mult}
Let $h: \Gamma \rightarrow \overline{C}\RR$ be a general tropical coral as in Definition \ref{def_general}; in particular all interior vertices of $\Gamma$ are trivalent. 
The \emph{multiplicity} of $h: \Gamma \rightarrow \overline{C}\RR$ is defined as in Definition 2.16,\cite{M} by
\begin{equation}\label{eq_mult}
\mathrm{Mult}(\Gamma, h):= \prod_{v\in V^0(\Gamma)}  \mathrm{Mult}(v) \,,     \end{equation}
where the product is over interior vertices, and for every
interior vertex $v \in V^0(\Gamma)$, choosing two arbitrary edges $e_1$, $e_2$ adjacent to $v$, denoting $u_1$,$u_2$ the primitive integral vectors emanating from $h(v)$ in the direction of $h(e_1)$ and $h(e_2)$, and $w_\Gamma(e_1)$, $w_\Gamma(e_2)$ the weights of $e_1$ and $e_2$,
\begin{equation}
\label{mult_vertex}\mathrm{Mult}(v):= w_\Gamma(e_1)\cdot w_\Gamma(e_2) \cdot |\mathrm{det}(u_1,u_2)|\,. 
\end{equation}
\end{definition}

\begin{definition}
\label{def mult labelled}
If $h: \Gamma \rightarrow \overline{C}\RR$
is a tropical coral with $\ell+1$ unbounded edges 
$e_0^+,\dots,e_\ell^+$ 
and $m$ negative vertices
$v_1^-,\dots,v_m^-$, we
define the \emph{modified multiplicity} of $h: \Gamma \rightarrow \overline{C}\RR$ by 
\begin{align}\label{eq_mult_labelled}
\widetilde{\mathrm{Mult}}(\Gamma,h)&:= \frac{1}{
\prod_{i=1}^\ell w_\Gamma(e_i^+)} \frac{1}{\prod_{j=1}^m w_{v_j^-}}
\mathrm{Mult}(\Gamma,h)
\\&= \frac{1}{
\prod_{i=1}^\ell w_\Gamma(e_i^+)} \frac{1}{\prod_{j=1}^m w_{v_j^-}}
\prod_{v\in V^0(\Gamma)}  \mathrm{Mult}(v)\,, \nonumber    \end{align}
where $w_\Gamma(e_i^+)$ are the weights of the positive edges
$e_1^+,\dots,e_\ell^+$
(all of them except $e_0^+$)
and 
$w_{v_j^-}$ are the weights on the negative vertices $v_j^-$. 
\end{definition}


We are ready to define the count of tropical corals. 
Let  $\Delta$
be a degree and $\lambda$
an asymptotic constraint in the stable range 
$\mathcal{S}_\Delta$.
By Proposition \ref{finiteness of types} and 
Remark \ref{no dependence}, there are finitely many 
tropical corals $h_i: \Gamma_i \rightarrow \overline{C}\RR$, 
$1 \leq i\leq n$,  matching $(\Delta,\lambda)$. They are all general, as $\lambda$ is general by Definition \ref{stable range} of $\mathcal{S}_\Delta$, and so they all have a modified multiplicity 
defined by Definition \ref{def mult labelled}.
Then, the tropical count of tropical corals is defined by
\begin{equation}
\label{Ntrop}
N^{\mathrm{trop}}_{\Delta,\lambda}:= \sum_{i=1}^n \widetilde{\mathrm{Mult}}(\Gamma_i, h_i)\,.
\end{equation}    

\begin{remark}
Note that in the definition \eqref{eq_mult_labelled} of the modified multiplicities $\widetilde{\mathrm{Mult}}(\Gamma, h)$, we have the factors given by the weights $w_\Gamma(e_i^+)$ and $w_{v_j^-}$. The reason to divide out the multiplicities with these weights in  \eqref{Ntrop} is that in the next sections of this article the constraints we will be imposing on the stable maps will be on the boundary divisor, rather than in the dense torus orbit of the Tate curve. This is the same set-up as in \cite{GPS}, and tropically amounts to modifying the usual Mikhalkin multiplicity, by dividing out the weights on labelled unbounded edges as explained in \cite[Appendix]{GPS}. 
\end{remark}

The notion of good constraint is ultimately justified by the following key Lemma, which allows us to reduce the study of tropical corals to the study of tropical curves. 
\begin{lemma}
\label{curves are extensions}
Let $\Delta$ be a degree and $\lambda$ a good general asymptotic constraint for $\Delta$. Then
every tropical curve $\widetilde{h}:\widetilde{\Gamma}\to \RR^2$ 
matching $(\Delta,\lambda)$
as in Definition \ref{trop_const}
is obtained after a possible rescaling as an extension in the sense of \S \ref{Extending a tropical coral to a tropical curve} of a tropical coral $h:\Gamma\to \trcone \RR$ 
matching $(\Delta,\lambda)$.
\end{lemma}
\begin{proof}
It is enough to show that the images $\widetilde{h}(V(\widetilde{\Gamma}))$ of vertices of $\widetilde{\Gamma}$ lie inside the cone $\mathcal{C}_{\Delta}$ defined in \ref{GoodConstraint}, since then either
\[\widetilde{h}(V(\widetilde{\Gamma})) \subset \trcone \RR\cap \mathcal{C}_{\Delta}\] 
and the restriction of $\widetilde{h}$ is already a tropical coral, or by rescaling $\widetilde{h}$ with some $s\in \RR_{\geq 1}$, we can ensure all vertices will be in $\trcone \RR\cap \mathcal{C}_{\Delta}$. 

Now assume there exist a vertex $v$ of $\widetilde{\Gamma}$ such that $h(v)\notin \mathcal{C}_{\Delta}$. Then it follows that there exists at least one unbounded edge $e$ of $\widetilde{\Gamma}$ with direction vector $u_e$ emanating from $h(v)$ in the direction of $h(e)$ with $u_e\notin \mathcal{C}_{\Delta}$. If $v$ is the only vertex of $\widetilde{\Gamma}$ with $h(v)$ not included in $\mathcal{C}_{\Delta}$ then this is obvious by the balancing condition at $v$. If not, then take a longest path from $v$ to a vertex $v'$ of $\widetilde{\Gamma}$ such that $h(v')\notin \mathcal{C}_{\Delta}$, which exists since $\Gamma$ is connected. In this case $v'$ must be adjacent to an unbounded edge $e$ of $\widetilde{\Gamma}$ such that the direction vector $u_e\notin \mathcal{C}_{\Delta}$ by the balancing condition at $v'$. But the existence of such an edge $e$ contradicts that  $\widetilde{h}:\widetilde{\Gamma}\rightarrow \RR^2$ has degree $\Delta$. Hence, the result follows. 
\end{proof}

Let $\Delta$ be a degree with 
$\ell+1$ positive and $m$ negative entries, 
and $\lambda$
an asymptotic constraint in the stable range 
$\mathcal{S}_{\Delta}$.
By \cite{M,NS}, there are finitely many general tropical curves 
$\tilde{h}_i: \tilde{\Gamma}_i \rightarrow \RR^2$ 
matching $(\Delta,\lambda)$
as in Definition \ref{trop_const}.
Recall that the unbounded edges of such tropical curve $\tilde{h}: \tilde{\Gamma} \rightarrow \RR^2$ are labelled
$\tilde{e}_0^+,\dots, \tilde{e}_\ell^+,
\tilde{e}_1^-,\dots,\tilde{e}_m^-$
and we define its 
\emph{modified multiplicity}
as in \eqref{eq_mult_labelled}:
\begin{equation}
\mathrm{Mult}(\tilde{\Gamma}, 
\tilde{h}):= \frac{1}{
\prod_{i=1}^\ell w_{\tilde{\Gamma}}(\tilde{e}_i^+)}
\frac{1}{\prod_{j=1}^m w_{\tilde{\Gamma}}(\tilde{e}_j^-)}
\prod_{v\in V(\tilde{\Gamma})}  \mathrm{Mult}(v)\,,
\end{equation}
where $w_{\tilde{\Gamma}}(e)$ are the weights of the labelled 
edges $e \in \tilde{\mathbb{E}}$, and the multiplicities $\mathrm{Mult}(v)$
of trivalent vertices are given by \eqref{mult_vertex}
Then, the tropical count of tropical curves is defined by
\begin{equation}
\label{NtropCurve}
\tilde{N}^{\mathrm{trop}}_{\Delta,\lambda}:= \sum_{i=1}^n 
\mathrm{Mult}(\tilde{\Gamma}_i, 
\tilde{h})_i\,.
\end{equation}

\begin{theorem}
\label{Independence Of Constraint}
Let  $\Delta$ be a degree and 
$\lambda$ an asymptotic constraint in the stable range 
$\mathcal{S}_{\Delta}$.
Then, the extension 
\[ (h:\Gamma \rightarrow 
\overline{C}\RR) \mapsto (\tilde{h}: \tilde{\Gamma}\rightarrow \RR^2)\]
of \S \ref{Extending a tropical coral to a tropical curve} 
defines a bijection between tropical corals 
matching $(\Delta, \lambda)$
and tropical curves matching 
$(\Delta,\lambda)$.
This bijection preserves the modified multiplicities:
\[ \widetilde{\mathrm{Mult}}(\Gamma,h) = \widetilde{\mathrm{Mult}}(\tilde{\Gamma},\tilde{h})\]
and so the count of tropical corals coincides with the count of tropical curves:
\[ N^{\mathrm{trop}}_{\Delta,\lambda}
=\tilde{N}^{\mathrm{trop}}_{\Delta, \lambda}\,.\]
Moreover, the count of tropical corals $N^{\mathrm{trop}}_{\Delta,\lambda}$ is independent of the choice of the asymptotic constraint $\lambda$ in the stable range 
$\mathcal{S}_\Delta$.
\end{theorem}
\begin{proof}
The fact that the tropical extension defines a bijection 
when the asymptotic constraint 
$\lambda$ is in the stable range 
$\mathcal{S}_{\Delta}$
follows from Lemma \ref{curves are extensions}.
Multiplicities agree by the explicit description of the tropical extension in \S
\ref{Extending a tropical coral to a tropical curve}.
Finally, the count of tropical curves $\tilde{N}^{\mathrm{trop}}_{\Delta, \lambda}$ is independent of the choice of 
$\lambda$ by \cite{GM}
and so the count of log corals $N^{\mathrm{trop}}_{\Delta,\lambda}$ is also independent of $\lambda$ as long as $\lambda$
is in the stable range $\mathcal{S}_\Delta$.
\end{proof}

\section{A tropical approach to homological mirror symmetry}
\label{sec: A tropical view on hms}
\subsection{The ring of theta functions $\hat{R}$}
\label{sec: The mirror to the Tate curve} 
The program developed by Gross and Siebert in the last twenty years, aims at an algebro-geometric approach to the Strominger--Yau--Zaslow conjecture \cite{SYZ} in mirror symmetry, and provides a technique to construct the homogeneous coordinate ring for the mirror to a family of Calabi--Yau varieties, referred to as \emph{the ring of theta functions} using tropical and log geometric techniques \cite{affinecomplex}. This construction is generalised to a large class of families in \cite{GHS}, and it is shown that the underlying vector space of this coordinate ring admits a canonical basis given in terms of theta functions -- for previous work in this direction in two dimensional cases, see also \cite{GHK}. 

In this section we investigate the product rule in the ring of theta functions, focusing attention on the mirror to the Tate curve -- for details we refer to \cite[8.4.2]{Clay}. We first fix an ample line bundle $\mathcal{L} \to X $
over the unfolded Tate curve, as in \cite[\S 8.4.1]{Clay} (where the unfolded Tate curve is denoted by $X_{\Sigma}$), which is obtained from the data of a piecewise-linear function 
\begin{eqnarray}
\label{Eq: phi the PL fnc}
\phi : \RR & \lra & \RR \\
\nonumber
x & \longmapsto & ix -
\frac{i(i+1)}{
2,}
\end{eqnarray}
which has slope $i$ on the interval $[i,i+ 1]$ for $i \in \ZZ$. The upper convex hull of the graph of $\phi$, denoted by $\Delta$ is depicted in \cite[Figure 24]{Clay}, and it is shown that the normal fan to $\Delta$ is the toric fan to the unfolded Tate curve, illustrated in Figure \ref{Fig:fan for unfolded tate curve}. A basis for the sections of $\mathcal{L}^{\otimes n}$ is represented by integral points of $n \Delta$. Moreover, there is a $\ZZ$-action on 
\[ \bigoplus_{n=0}^{\infty} H^0(X,\mathcal{L}^{\otimes n}). \]
A generator of this action, for any non-negative integer $n \geq 0$ is given by the map
\begin{equation}
\label{Eq:Z action on L}
 T_n(x,y) := (x + nb,xb + y + \frac{b(b - 1)}{
2}
n)
\end{equation}
where $b \in \ZZ_{>0}$ is fixed, so that the action of $\ZZ$ on the toric fan for the unfolded Tate curve is given by translations by $b$, as in \S\ref{The Tate curve} (in \cite[\S8.4.1]{Clay} $b$ is denoted by $d$). Taking the quotient with respect to the $\ZZ$ action obtained on $\mathcal{L}$ given by by \eqref{Eq:Z action on L}, we obtain an ample line bundle
\begin{equation}
\label{E:line bundle}
  \mathcal{\hat{L}} \to T 
  \end{equation}
on the Tate curve (in \cite{Clay}, the Tate curve $T$ is denoted by $\hat{\mathcal{X}}$). A basis for the sections of basis of sections of $H^0(T, \mathcal{\hat{L}}^{\otimes n})$ is given by the \emph{theta functions}
\begin{equation}
\label{Eq: theta functions}
\theta_{n,p} := \sum_{s= - \infty}^{s= \infty} z^{T^s_n(np,n\phi(p))} 
\end{equation}
where $T_n$ is defined as in \eqref{Eq:Z action on L}, $\phi$ is the PL function defined in \eqref{Eq: phi the PL fnc},
and $p \in B((1/n)\ZZ)$. The \emph{ring of theta functions} is a $C \lfor u \rfor$-algebra, 
\begin{equation}
\label{Eq: ring of theta functions}
\hat{R} = \bigoplus_{n=0}^{\infty} H^0 (T,\hat{\mathcal{L}}^{\otimes n})
\end{equation}
and the set of $C \lfor u \rfor$-module generators is given by 
\[ \{1\} \cup \{\theta_{n,p}~|~p \in B((1/n)\ZZ),n \geq 1\}.\]
as explained in \cite[\S8.4.1]{Clay}. The product in $\hat{R}$ is given by
\begin{equation}
\label{Eq:multiplying theta functions}
\theta_{n_1,p_1} \cdot \theta_{n_2,p_2} = \sum_{\alpha=-\infty}^{\alpha=\infty} \theta_{n_1+n_2,\frac{n_1p_1+n_2(p_2+\alpha b)}{
n_1+n_2}}u^{\mathrm{deg}(p_1,p_2+\alpha b)}
\end{equation}
where $b\in \ZZ_{>0}$ is fixed as in \S\ref{The Tate curve}, and for any $p_1 \in \frac{1}{n_1}\ZZ, p_2 \in \frac{1}{n_2}\ZZ$, 
\begin{equation}
\label{Eq: degree}
\mathrm{deg}(p_1,p_2):=  n_1\phi(p_1) + n_2\phi(p_2) - (n_1+n_2) \bigg(  \frac{n_1p_1+n_2(p_2+\alpha b)}{
n_1+n_2} \bigg).
\end{equation}
as defined in \cite[8.4.2]{Clay}. It is shown in \cite[8.4.2]{Clay}, that the multiplication in the ring of theta functions $\hat{R}$ for the Tate curve, defined in \eqref{Eq: ring of theta functions}, agrees with the Floer multiplication in the Lagrangian Floer cohomology ring associated to the elliptic curve. This is shown by working with a combinatorial analogue of the Fukaya category, introduced by Abouzaid--Gross--Siebert as \emph{tropical Morse category}, where the product is described by counts of \emph{tropical Morse trees} which correspond to holomorphic disks bounded by Lagrangian submanifolds in the elliptic curve. We describe tropical Morse trees, and show the correspondence of such trees with tropical corals in the next sections.

\subsection{A view toward symplectic cohomology}\label{ViewOnSH} 
\label{Sec: SH}
The ring of theta functions $\hat{R}$ defined as in \eqref{Eq: ring of theta functions} is proposed to be the coordinate ring to the mirror of $T\setminus T_0$, the complement of the central fiber $T_0$ in the Tate curve \cite{GHS}. The mirror to the Tate curve $T$ is then analogously obtained by a compactification, as explained in \cite[\S4]{GHS}. Note that in the analytic category there is an elliptic fibration on $T\setminus T_0$ over the punctured disc $D^*$. The mirror as proposed by Strominger--Yau--Zaslow of $T\setminus T_0$ is also a fibration over $D^*$ formed by dual elliptic curves -- for a survey of SYZ mirror symmetry see for instance \cite{G3}.
Hence, $T\setminus T_0$ is in this sense self-mirror. It is proposed by \emph{homological mirror symmetry} that the mirror to $T\setminus T_0$ is given by the $\mathrm{Spec}$ of the degree zero part of the symplectic cohomology ring of it. Note that, topologically $T\setminus T_0$ is the mapping cylinder of the Dehn twist $\tau \colon E \to E$ of an elliptic curve along a meridian
\[
  \Map(\tau) := \frac{E \times [0,1] \times \mathbb{R}}{\big(e,0,q\big) \sim \big(\tau(e), 1, q\big)}.
\]
Although symplectic cohomology is described for open manifolds with a Liouville structure \cite{Se}, and $\Map(\tau)$ does not fall into this category, we nevertheless can define an appropriate modification symplectic cohomology in this set-up. Similar descriptions of the symplectic cohomology for mapping cylinders of Hamiltonian symplectomorphisms can be found in \cite{F}. In a related context, for a study of the Floer homology of Dehn twists see also \cite{HS}.

We equip $\Map(\tau)$ with the symplectic form $\omega_E + dp \wedge dq$, where $\omega_E$ is the symplectic form on $E$ and $p$ and $q$ are the coordinates on the second and third factors of $E \times [0,1] \times \mathbb{R}$, and 
choose a tame almost complex structure on $\Map(\tau)$, such that the projection $\pi: \Map(\tau) \to \RR \times S^1$ is holomorphic. We then consider a Hamiltonian function, given by the composition of $\pi$ with the projection to the first factor, whose time $1$-periodic orbits after suitable perturbations are non-degenerate, and thus we obtain a well-defined Floer complex, which allows us to describe the symplectic cohomology ring for $T\setminus T_0$ as 
\[SH^\star\big(T\setminus T_0\big) \cong \bigoplus_{k \in \ZZ} HF^\star(\tau^k),\]
where $HF^\star$ denotes the Hamiltonian Floer cohomology ring. Details of this construction, as well as checking maximum principle for the perturbations of the Hamiltonians require a rather large set-up of a symplectic geometric framework, which beyond the scope of this paper, and will be the focus of future work.  
Focusing attention at the degree zero piece, we obtain an isomorphism
\begin{equation}
\label{Eq HandF}
    \bigoplus_{k \in \ZZ} HF^0(\tau^k) \cong \bigoplus_{k \in \ZZ} HF^0\big(L(0), L(k)\big), 
\end{equation}
where $HF$ on the right hand side denotes the Lagrangian Floer cohomology. Based on discussions with Abouzaid--Siebert and Pomerleano--Tonkonog, the isomorphism \eqref{Eq HandF} follows by showing both sides are isomorphic to the wrapped Floer cohomology $HW^0(L(0)\times \RR, L(0)\times \RR)$. To do this, one needs to observe that the closed open map 
\[SH^0 \cong \bigoplus_{k \in \ZZ} HF^0(\tau^k)  \longrightarrow HW^0(L(0)\times \RR, L(0)\times \RR)\] is an isomorphism analogous to \cite[Proposition 7.2]{P}. This shows that the product in the symplectic cohomology of $T\setminus T_0$, can be understood from the product in the Lagrangian Floer cohomology of the elliptic curve, which agrees with the product in the ring of theta functions by \cite[\S8.4.2]{Clay}.

Our contribution in this paper is to interpret the structure constants contributing to the Floer product for the elliptic curve, namely holomorphic polygons bounded by Lagrangians, on the algebro-geometric side by counts of certain punctured curves defined using log geometric techniques. To pass from holomorphic polygons to punctured log curves, we are using the correspondences summarised in Figure~\ref{results}. In particular, use combinatorial analogues of these holomorphic polygons, given by \emph{tropical Morse trees}, as defined in the next section.
\begin{remark}
It is worthwhile emphasizing that, comparing algebraic and symplectic virtual fundamental classes on the moduli of punctured log curves is highly technical. Generally, to relate the symplectic and algebraic formalisms there are various approaches -- see \cite{V} for instance. Thus, the isomorphism between the ring of theta functions \cite{GHS} and symplectic cohomology ring, is so far conjectural \cite{Utah}. 
\end{remark}

\subsection{Tropical Morse trees}
\label{Sec: Tropical Morse trees}
The product in the Lagrangian Floer cohomology is, roughly, determined by counting holomorphic polygons bounded by Lagrangians. Such polygons, capturing the product rule in Floer cohomology for the case of the elliptic curve $E$ correspond to tropical Morse trees on $S^1$, as shown in \cite[\S $8.4.4$]{Clay}. Tropical Morse trees, introduced by Abouzaid--Gross--Siebert, are tropical analogues of the gradient flow trees in Morse theory. In the remaining part of this section we review the definition of tropical Morse trees along with examples illustrating their correspondence to holomorphic polygons. For details we refer to \cite[\S8]{Clay}.

A \emph{rooted ribbon tree} $\mathcal{R}$ is a connected tree with a finite number of vertices and edges, with no divalent vertices, together with the additional data of a cyclic ordering of edges at each vertex and a distinguished vertex referred to as the \emph{root vertex}.  We call the univalent vertices of $\mathcal{R}$ \emph{external} and the other vertices \emph{internal}, and insist that the root is an external vertex.  We orient all edges towards the root vertex, referring to the root vertex as the \emph{outgoing vertex} and all other external vertices as \emph{incoming vertices} or \emph{leaves}. We refer to edges adjacent to a vertex $v$ as \emph{incoming edges} if the assigned direction points towards $v$ and \emph{outgoing edges} otherwise.
Let $\mathcal{R}$ be a rooted ribbon tree with $d+1$ external vertices, for $d>0$.  There is a unique isotopy class of embeddings of $\mathcal{R}$ into the unit disc $D \subset \RR^2$ such that each external vertex maps to $S^1 \subset D$.  This embedding divides $D$ into $d+1$ regions, each of which meets the boundary $S^1$ in a segment.  We label the regions by $0,\ldots,d$, proceeding anticlockwise around $S^1$ from the root vertex, and assign distinct integers $n_0,\ldots,n_d$ to these regions. We then also assign to each edge~$e_{i,j}$, separating regions labelled by $i$ and $j$, the integer 
\begin{equation}
\label{Eq: acceleration}
 n_{e_{i,j}} := \begin{cases} 
     n_j-n_i & \mathrm{on}~ e_{ij} ~ \mathrm{for} ~ 0 < i,j < d  \\
      n_d - n_0 & \mathrm{on}~ e_{d,0}
   \end{cases}
\end{equation}
called the \emph{acceleration of $e_{i,j}$}.
A rooted ribbon tree whose edges are labelled with accelerations is called \emph{decorated}. We also label each external vertex of a decorated ribbon tree by $v_{[i][j]}$, where the unique edge incident to the vertex lies between regions $i$ and $j$, and the square brackets indicate that we take the indices $i$ and $j$ modulo $d+1$.  Note that the vertex labels are $v_{0,1},v_{1,2},\ldots,v_{d-1,d},v_{d,0}$.

Now we are ready to define tropical Morse trees as maps from $\mathcal{R} \to S^1$. Note that we will denote the points on $S^1 = \mathbb{R}/d\mathbb{Z}$ with coordinates in $\frac{1}{n}\mathbb{Z}$ by $ S^1\big(\frac{1}{n}\ZZ\big)$ as in \cite[\S $8$]{Clay}.
\begin{definition}
\label{tropicalMorsetree}
Let $\mathcal{R}$ be a rooted ribbon tree with $d+1$ external vertices, $d>0$, and $n_0,\ldots,n_d \in \ZZ$ be the set of integers describing a decoration on $\mathcal{R}$. Identify each edge $e$ of $\mathcal{R}$ with $[0,1]$ with coordinate $s$ and the orientation on $e$ pointing from $0$ to $1$. A \emph{tropical Morse tree} is a map $\phi:\mathcal{R} \rightarrow S^1$ satisfying the following:
\begin{itemize}
\item[(i)] For any external vertex $v_{j,i}$ of $\mathcal{R}$, 
$$p_{j,i}:=\phi(v_{j,i}) \in S^1\big(\textstyle \frac{1}{n_{e_{i,j}}}\ZZ\big) $$
where $n_{e_{i,j}}$ is the acceleration associated to the edge adjacent to $v_{j,i}$, defined as in \eqref{Eq: acceleration}.
\item[(ii)] For an edge $e$ of $\mathcal{R}$, $\phi(e)$ is either an affine line segment or a point in $S^1$.
\item[(iii)] For each edge $e$, there is a section $v_e\in \Gamma(e,(\restr{\phi}{e})^{*}TS^1)$, called \emph{velocity of $e$}, satisfying 
\begin{itemize}
\item[1)] $v_e(v)=0$ for each external vertex $v$ adjacent to the edge $e$.
\item[2)] For each edge $e\cong[0,1]$ and $s\in [0,1]$, we have $v_e(s)$ is tangent to $\phi(e)$ at $\phi(s)$, pointing in the same direction as the orientation on $\phi(e)$ induced by that on $e$. By identifying $(\restr{\phi}{e})^{*}TS^1$ with the trivial bundle over $e$ using the affine structure on $S^1$, we have
\[\frac{d}{ds}v_e(s)=n_e\phi_{*}\frac{\partial}{\partial s} \]
\item[3)] For any internal vertex $v$ of $\mathcal{R}$ the following balancing condition holds. Let $e_1\ldots,e_p$ be the incoming edges and let $e_{out}$ be the outgoing edge adjacent to $v$. Then,
  \begin{equation}
    \label{eq:balancing_tropical_Morse}
    v_{e_{out}}(v)=\sum_{i=1}^p v_{e_{i}}(v) 
  \end{equation}
\end{itemize}
\end{itemize}
\end{definition}
\begin{remark}
We can generalise the definition of a tropical Morse tree in $S^1$, to a tropical Morse tree in any affine manifold as in \cite[Defn $4.1$]{Theta}. For computational convenience, in the remaining part of this article, we indeed consider lifts of tropical Morse trees $\phi: \mathcal{R} \to S^1$ to tropical Morse trees in $\phi: \mathcal{R} \to \RR$, which by abuse of notation are also denoted by $\phi:\mathcal{R}\to\RR$. 
\end{remark}
\begin{definition}
\label{GeneralTMT}
A tropical Morse tree $\phi:\mathcal{R}\to\RR$ is called \emph{general} if all interior vertices of $\mathcal{R}$ are trivalent and at any vertex $v$ of $\mathcal{R}$ there is at most one adjacent edge to $v$ is contracted under $\phi$. We denote the set of general tropical Morse trees by $\mathcal{TMT}$. 
\end{definition}
\begin{remark}
\label{contraction}
Let $\phi: \mathcal{R} \to \RR$ be a tropical Morse tree.
The condition $v_e(v)=0$ on all external vertices $v$ in Definition \ref{tropicalMorsetree} can be interpreted as follows. Assume $v$ is an external vertex that is not the root vertex, start with zero velocity $v_e(v)=0$, and then increase it while tracing the orientation towards the root vertex until the last vertex adjacent to the root, where the velocity will evaluate positively. Hence, to achieve zero velocity also at the root vertex we need to have negative acceleration at the edge adjacent to it. This describes a systematic procedure to contract certain external edges under $\phi$, to achieve $v_e(v)=0$ -- see \cite{Abouzaid} and \cite[pg $35$]{Theta}. Hence, the edges that are contracted under $\phi$ correspond to the following ones.
\begin{itemize}
\item[(i)] Edges $e$ adjacent to an incoming vertex, if $n_e<0$.
\item[(ii)] The outgoing edge adjacent to the root vertex, if $n_e>0$.
\end{itemize}
Note that there exists always at least one external edge which is contracted under $\phi$, and at least one external edge which is not contracted under $\phi$ by studying all possible algebraic inequalities between the integers $n_i$ and $n_j$ determining the acceleration $n_e=n_j-n_i$ on the edges of $\mathcal{R}$.
\end{remark}

\begin{remark}
To each tropical Morse tree we associate a sign, which agrees with the sign associated to a holomorphic polygon, and is defined as in \cite[8.3.3]{Clay}. It follows from \cite[\S8.4.4]{Clay} that there is a bijective correspondence between tropical Morse trees $\phi:\mathcal{R} \to \RR$ of a fixed sign, and holomorphic polygons $P$ in $\mathbb{R}^2$, such that;
\begin{itemize}
\item[i.] If $\mathcal{R}$ has $d$-external vertices $v_{0,1},\ldots,v_{d,0}$ with $\phi(v_{i,j})=p_{i,j}$, then $P$ has $d$-vertices $\tilde{p}_{0,1},\ldots,\tilde{p}_{d,0}$ such that under the vertical projection map each $\tilde{p}_{i,j}$ maps to $p_{i,j}$.
\item[ii] The points $\tilde{p}_{0,1}, \tilde{p}_{1,2},\ldots,\tilde{p}_{d,0}$ on the boundary of $P$
are oriented cyclically in a counter-clockwise fashion.
\item[(iii)] If $\mathcal{R}$ is decorated with $\{ n_0,\ldots , n_d \}$. Then, each edge $L_i$ of $P$ has slope $S_{L_i}= -n_i$ for $ i \in \{1,\ldots,d\}$.
\end{itemize}
\end{remark}
In the following examples we illustrate tropical Morse trees and the corresponding polygons, assuming that signs are chosen so that the orientation of the boundary traced in counter-clockwise fashion is compatible with the ordering of the tuple $(L_0, \ldots L_d)$. 
\begin{example}
\label{SimpleExample}
Let $\phi:\mathcal{R}\rightarrow B$ be a tropical Morse tree, where $\mathcal{R}$ is a rooted ribbon tree decorated with 
\[ n_0=0, \,\ \,\ n_1= 3, \,\ \,\ n_2=5, \]
so that $\mathcal{R}$ separates the unit disk into three regions each labelled by $n_i$, as illustrated in Figure \ref{fig:polygon to tmt}. The acceleration vectors on the edges $e_{ij}$ of $\mathcal{R}$, adjacent to the external vertices $v_{i,j}$ are given by
\[ n_{e_{0,1}}=n_1-n_0=3,\,\ \,\ n_{e_{1,2}}=n_2-n_1=2 \,\ \,\  n_{e_{0,2}}=n_2-n_0=5 \]
Hence, by Remark \ref{contraction}, the edge $e_{0,2}$ gets contracted under $\phi$. We let 
\[  \phi{(v_{0,1})}=p_{0,1}=2, \,\ \,\  \phi{(v_{1,2})}=p_{1,2}=-3 \,\ \,\   \phi{(v_{0,2})}=p_{0,2}=0   \] 
so that the balancing condition in Definition \ref{tropicalMorsetree} holds. We illustrate the corresponding polygon to $\phi$ on the right hand side in Figure \ref{fig:polygon to tmt}.
\begin{figure}
	 \resizebox{.9\linewidth}{!}{\input{polygon.pspdftex}}
	\caption{A tropical Morse tree in blue and the corresponding triangle}
\label{fig:polygon to tmt}
\end{figure}
\end{example}

\begin{example}
\label{non-convex example}
Let $\mathcal{R}$ be the decorated ribbon tree illustrated in Figure \ref{fig:patience}. The accelerations $n_{e_{i,j}}=n_i-n_j$  on the external edges $e_{i,j}$ are given by
\[n_{e_{0,1}}  =  -2,~~~~ \,\ n_{e_{1,2}} =  3,~~~~ \,\
n_{e_{2,3}} = 2,~~~~ \,\
n_{e_{3,4}} = -8,~~~~ \,\
n_{e_{4,0}} = -5 . \]
Let $\phi: \mathcal{R} \to \RR$ be a tropical Morse tree with
\[ \phi(v_{0,1}) = p_{0,1} =8,~~~~\,\ \phi(v_{1,2}) = p_{1,2} =12,~~~~\,\ \phi(v_{2,3}) = p_{2,3} =0, ~~~~\,\ \phi(v_{3,4}) = p_{3,4} =6 \] 
Note that, by Remark \ref{contraction}, two external edges, that are not adjacent to the root vertex, which have negative acceleration are contracted. Furthermore, as the image of the edges of $\mathcal{R}$ overlap, we illustrate them in Figure \ref{fig:patience} on three copies of the real line, and we label on each of them the associated acceleration.
\begin{figure}[ht]
\begin{minipage}[b]{0.45\linewidth}
\centering
	 \scalebox{.4}{\input{backforth.pspdftex}}
\end{minipage}
\hspace{0.5cm}
\begin{minipage}[b]{0.45\linewidth}
\centering
		 \scalebox{1.2}{\input{TropMorse1.pspdftex}}
	\vspace{1.2cm}
\end{minipage}
\caption{A tropical Morse tree and the associated polygon}
\label{fig:patience}
\end{figure}

Denote by $w$ the vertex adjacent to the root vertex. We have a one-parameter family of choices for $\phi(w)$, for which balancing condition in Definition \ref{tropicalMorsetree} is satisfied. Choosing $\phi(w)=7$, from the balancing condition we get $p_{04}=\frac{28}{5}\in \frac{1}{5} \ZZ$. 
For details see \cite[Example $8.8$]{MyThesis}. We illustrate the associated polygon one the right hand side of Figure \ref{fig:patience}. 

Note that the holomorphic disks contributing to the structure constants in the product in the Floer cohomology ring for the elliptic curve, in $0$-dimensional moduli spaces always have convex corners. However the polygon, corresponding to the tropical Morse tree in this example is in a $1$-dimensional moduli space, as there is a $1$-parameter choice of the image of the vertex $w\in \mathcal{R}$. This is expected in situations when there is more than one edge of the ribbon graph contracted. Indeed, it is shown in \cite[\S2.3]{Sl} that tropical Morse trees corresponding to convex polygons are those in which precisely one edge of $\mathcal{R}$ gets contracted.  
\end{example}

\subsection{Tropical Corals from Tropical Morse Trees}
\label{From Tropical Morse Trees to Tropical Corals}
In this section we show how to lift tropical Morse trees on $S^1$ to \emph{good types} of tropical corals on the truncated cone $\trcone S^1$ over $S^1$ and vice versa. 
\begin{definition}
\label{good type}
Let $(\Gamma,u)$ be the type of a coral graph as in Definition \ref{type of a coral graph}. We say $(\Gamma,u)$ is of \emph{good type} if for all edges $e \in E(\Gamma)$, the direction vector $u(e)$ projects to a non-zero element of $\RR$ under the projection map $\pr_2: \RR^2 \lra \RR$ onto the second component.
\end{definition} 

A tropical coral $h:\Gamma \to \trcone \RR$ of type $(\Gamma,u)$, of good type, uniquely determines a tropical Morse tree $\phi:\mathcal{R} \to \RR$, which we refer to as the tropical Morse tree associated to $h$. We construct $\phi$ from $h$ as follows. First, set $\mathcal{R}$ be the ribbon tree isomorphic to $\overline{\Gamma}$. Note that the orientation on $\overline{\Gamma}$ induces an orientation on $\Gamma$, and hence determines an orientation on $\mathcal{R}$. We refer to the orientation on an edge $e$ positive, if $h(e)$ points away from the boundary $\partial \trcone \RR$ and negative if it points towards the truncated cone. Endow the edges of $\mathcal{R}$ with accelerations given by $n_e = w_\Gamma(e)$, if the orientation on the edge $e$ is positive, and $n_e = - w_\Gamma(e)$ otherwise, where $w_\Gamma:E(\Gamma) \to \NN\setminus \{0\}$ is the weight function on $\Gamma$. Without loss of generality assuming $n_0=0$, this determines the decoration on $\mathcal{R}$. For any vertex $v\in V\setminus V^-(\bar{\Gamma})$, set 
\[  u_{v,e}:= (u^1,u^2) \]
the direction vector as in Definition \ref{type of a tropical coral}. Identify $\partial \trcone \RR$ with $\RR$, and define the tropical Morse tree $\phi:\mathcal{R} \to \RR$, by setting
\[
\phi(v) = \begin{cases} h(v) & \mathrm{if ~} v\in V^-(\bar{\Gamma})
\\ u^1/u^2 &  \mathrm{if ~} v\in V\setminus V^-(\bar{\Gamma}) \end{cases}
\]

\begin{example}
In Figure \ref{fig:CoralToTMT} we illustrate a tropical coral and whose associated tropical Morse tree is the one in Figure \ref{fig:patience}.
\begin{figure}
	 \resizebox{.6\linewidth}{!}{\input{CoralToTMT.pspdftex}}
	\caption{A tropical coral for which the associated tropical Morse tree is as in Figure \ref{fig:polygon to tmt}}
\label{fig:CoralToTMT}
\end{figure}

\end{example}

Conversely, from a tropical Morse tree, we first deduce a coral graph, together with its type, in the following result.
\begin{lemma}
\label{TMTdeterminesType}
A tropical Morse tree $\phi:\mathcal{R} \to \RR$ uniquely determines $(\Gamma,u)$, where $\Gamma$ is a coral graph, and $u$ is the type of a tropical coral, as in Definition ~\ref{type of a tropical coral}.
\end{lemma}
\begin{proof}
Let $\phi:\mathcal{R} \to \RR$ be a tropical Morse tree as in Definition \ref{tropicalMorsetree}. Assume $\mathcal{R}$ has $d+1$ external vertices $\{v_{ji}\}$ with $\phi(v_{ji})=p_{ji}$.  Denote by $e_{ij}$ the edge of $\mathcal{R}$ lying between regions~$i$ and~$j$. The graph $\mathcal{R}$ defines a bilateral graph $\overline{\Gamma}$, where the partition on the set of vertices $V(\bar{\Gamma})=V(\mathcal{R}),$ into sets of negative, positive and interior vertices are obtained as follows.
\begin{align*}
 V^- (\bar{\Gamma}) &= \{ v ~|~ \text{$v$ is an external vertex of $\mathcal{R}$, adjacent to an edge which is contracted under $\phi$   }  \}\\
 V^+ (\bar{\Gamma}) &= \{ v ~|~ \text{$v$ is an external vertex of $\mathcal{R}$ adjacent to an edge which is not contracted under $\phi$   }  \} \\
V^0 (\bar{\Gamma}) &= \{ v ~|~ \text{$v$ is an internal vertex of $\mathcal{R}$ } \} 
\end{align*}
We set $\Gamma = |\bar{\Gamma}| \setminus V^+(\Gamma)$, and define a weight function on $E(\Gamma)$ by 
\[w_\Gamma(e)= |n_e|\]
where $n_e$ is the acceleration of $e$.

Let $\trcone \RR$ be the truncated cone over $\RR$ defined as in Definition \ref{truncated cone}. For any tropical Morse tree $\phi:\mathcal{R} \to \RR$, we identify $\RR = \partial \trcone \RR$, so that $\phi(\mathcal{R}) \subset \trcone \RR$.  Without loss of generality we assume, after translating $\phi$ if necessary, that $\phi$ maps the root vertex, denoted by $v_{0d}$, to $( 0,1 )$. Let $v\in V^0(\Gamma)$, and
\[p=\phi(v) \in \partial \trcone \RR.\] 
Define the direction vector $u_{(v,e)}$ assigned to the flag $(v,e) \in F(\Gamma)$ by
\[ u_{(v,e)}:= \begin{cases} 
      \,\ \,\ \langle p,1 \rangle & \mathrm{if}~ e ~\mathrm{is ~ an ~ outgoing ~ edge ~ adjacent ~ to} ~ v \\
        - ~~ \langle p,1 \rangle &\mathrm{if}~ e ~\mathrm{is ~ an ~ incoming ~ edge ~ adjacent ~ to} ~ v 
   \end{cases}
\]
For a negative vertex $v_{ij} \in V^-(\Gamma)$, define $u_{v_{ij}} = - \langle p_{ij},1 \rangle$, where $p_{ij}=\phi(v_{ij})$. The direction vectors $\{u_{(v,e)}\}$ and $\{u_{v_{ij}}\}$ determine the type $(\Gamma,u)$ of the coral graph $\Gamma$, hence the type of a tropical coral. 
\end{proof}

\begin{theorem}
\label{SurjectionOfCoralsToTMT}
Let $\mathcal{T}$ be the set of tropical corals in $\trcone \RR$ which are of good type. For any tropical coral $h$, denote by $\psi_h$ the tropical Morse tree obtained from $h$. Then, the map
\begin{eqnarray}
\label{SurjectionTMT}
\Psi: \mathcal{T} & \lra &   \mathcal{TMT}\\
\nonumber
h & \longmapsto & \psi_{h} 
\end{eqnarray}
is a surjection, and the fiber over $\phi \in \mathcal{TMT}$ is the set of tropical corals of type $(\Gamma,u)$ determined by $\phi$.
\end{theorem}
\begin{proof}
We show this for the general case (which suffices for our purposes in this paper, since we set up the tropical counting problem for general tropical corals in \S \ref{A tropical counting problem}). The proof for non-general cases follows analogously. Let $\phi:\mathcal{R} \to \RR$ be a general tropical Morse tree and let $(\Gamma,u)$ be the type of the coral graph determined by $\phi$ by Lemma \ref{TMTdeterminesType}. Note that $(\Gamma,u)$ is a good type. We will construct a general tropical coral $(h:\trcone \RR \to \RR) \in\mathcal{T}_{(\Gamma,u)}$ with $\Psi(h)=\phi$. Let $l$ be the number of positive vertices of $\Gamma$ and let
\[ (r_0,\ldots,r_{l-2}) \in \RR_{>0}^{l-1}   \]
be an arbitrary $l-1$-tuple of positive real numbers. Note that there always exists at least one external edge which is contracted under $\phi$ -- see Remark \ref{contraction}. We fix a negative vertex $v \in V^-(\Gamma)$ with $h(v)=(p,1) \in \trcone \RR$. Let $v^0$ be the interior vertex of $\Gamma$ adjacent to $v$ and let $\{v_1,\ldots,v_{l-1}\}$ be the set of interior vertices of $\Gamma$ not adjacent to any negative vertex $V^-(\Gamma)$. Or aim is to construct a tropical coral $h:\Gamma \to \trcone \RR$ with 
\[\Phi(h)= (\rho(v^0),\rho(v_1),\ldots,\rho(v_{l-2}) ) = (r_0,\ldots,r_{l-2})  \]
where $\Phi$ is the map defined in the proof of Proposition \ref{EmbeddingOfT}. We will in a moment determine $h(v^0)$. Let $\pi_u$ be the canonical projection map $\RR^2 \to \RR^2/\RR \cdot u$ where $u=(p,1)$, and write $R_{p}$ for the ray in $\trcone \RR$ which maps to $[p]$ under $\pi_{u}$.\footnote{Note that the broken lines appearing in the scattering procedure in \cite{GHK} are the rays in $\trcone \RR$ obtained as the inverse image of $[p_{ji}]$ under $\pi_u$, where $u=(p_{ji},1) \in \partial \trcone \RR$ and $p=\phi(v)$ for an external vertex $v_{ji}$ of $\mathcal{R}$.} Fix $h(v^0)$ on $R_{p}$ with $\rho(h(v^0))=r_1 \in \RR_{>0}$. Construct the rest of $h$ inductively.  Let $v'$ be an interior vertex of $\Gamma$ such that $h(v')$ is already determined, let $v''$ be a vertex adjacent to $v'$. Assume $v''$ is also adjacent to a negative vertex ${v'}^-$. Then knowing the type and the positions of $h({v'}^-)$ and $h(v')$ determines $h(v'')$ uniquely (note that we initially choose $r_0 \in \RR_{>0}$ so that $h(v'') \in \trcone \RR$). If $v''$ is not adjacent to any negative vertex then $v''=v_k$ for some $v_k\in \{v_1,\ldots,v_{l-1}\}$. In this case, $h(v'')$ is determined uniquely by $\rho(v_k)$. Hence, the choice of $(r_0,\ldots,r_{l-2})$ determines $h\in \mathcal{T}_{(\Gamma,u)}$ uniquely. Note that the balancing condition for $h$ (\ref{parameterized tropical coral},ii) follows from the equation \eqref{eq:balancing_tropical_Morse}. Thus, the result follows.
\end{proof}
The following is an immediate corollary of Theorem \ref{SurjectionOfCoralsToTMT}.
\begin{corollary}
\label{thm: bijective correspondence corals and TMT}
Let $\lambda$ be a general constraint in the stable range as in Definition \ref{stable range}, for a good type $(\Gamma,u)$. Then, there is a bijective correspondence between tropical Morse trees in $S^1$, and tropical corals of type $(\Gamma,u)$, which is determined as in Lemma \ref{TMTdeterminesType} from the tropical Morse trees, matching $\lambda$.
\end{corollary}
Note that the choice of a constraint in the stable range in Corollary \ref{thm: bijective correspondence corals and TMT} has also a Floer theoretic interpretation. Indeed while defining the symplectic cohomology the fact that we restrict to the stable range, will impose certain energy bounds on Floer trajectories. Further details of these symplectic aspects will be studied in future work.

\begin{figure}
	\input{CoralsBoth.pspdftex}
	\caption{A tropical coral obtained from the tropical Morse tree in Figure \ref{fig:patience}}.
\label{fig:cool coral}
\end{figure}

\section{Counts of Log Corals}
\label{A curve counting problem}
Throughout this section we assume basic familiarity with log
geometry \cite{Kk, Ogus}. We denote a log scheme $(X,\M_X)$ by $X^{\dagger}$ and the structure homomorphism of $X^{\dagger}$ by 
\[\alpha_X \colon \M_X\to\mathcal{O}_X.\]
We refer to $\overline{\M}_{X}:=\M_X / \mathcal{O}_X^{\times}$ as the \emph{ghost sheaf} on $X$.  We use $s_*$ for the sections of a log structure $\M_X$, defined by the monomial functions
indicated in the subscript. The corresponding sections of $\ol{\M}_X$ are denoted by $\bar{*}$.

If $X$ is a toric variety associated to a fan $\Sigma$ in $N_\RR$,then there is a canonical \emph{divisorial log structure} $\M_X=\M_{(X,D)}$, where $D\subset X$ is the toric boundary divisor \cite[ex. 3.8]{Mark}. Note that for any $\sigma \in \Sigma$, we have the affine toric subset 
\[U_\sigma= \Spec\CC[\sigma^\vee\cap M]\] 
of $X$, and the divisorial log structure $\M_X$ is locally generated by the monomial functions on this open subset. That is, the canonical map
\begin{eqnarray}
\label{toric chart}
\sigma^\vee\cap M & \lra & \CC[\sigma^\vee\cap M],\\
\nonumber
m & \longmapsto & z^m
\end{eqnarray}
is a \emph{chart for the log structure}, as in \cite[Def. $2.9(1)$]{Kk}, on $U_\sigma$. 
\begin{proposition}
\label{toric charts and ghost sheaves}
For any $x\in T_\sigma$ the toric chart~\eqref{toric chart} induces a canonical isomorphism
\[
\sigma^\vee\cap M/\sigma^\perp\cap M \stackrel{\sigma}{\lra} \ol \M_{X,x}.
\]
\qed
\end{proposition}
\begin{proof}
The proof is straight forward and can be found in \cite[Prop A.30]{MyThesis}. 
\end{proof}
\subsection{The log structure on the Tate curve}
\label{The log structures on the Tate curve and its degeneration}
Fix $b \in \ZZ_{>0}$, and let $\tilde{\pi}: Y\to \Spec\CC[s,t]$ be the degeneration of the unfolded Tate curve, obtained from $(\trcone \RR, \trcone \P_b)$ as in \S \ref{The unfolded Tate curve}. We endow $Y$ with the divisorial log structure $\alpha_Y:\M_Y\to \mathcal{O}_Y$, where 
\[\tilde{D}:=\tilde{\pi}^{-1}(st=0)\subset Y \]
The affine cover \eqref{stb} for $Y$ induces an affine cover of $Y_0$ by restricting to $t=0$, which is given by a countable union of the open sets
\[  \mathcal{U} =  \Spec \mathbb{C}[x,y,s]/ (xy)   \]
A toric chart for the log structure $\M_{Y_0}$ is given by Proposition \ref{toric charts and ghost sheaves}. The stalks of $\overline{\mathcal{M}}_{Y_0}$ are classified in Table \ref{StalksOfTheGhostSheaf},
\begin{table}
\begin{tabular}{cc} \toprule
Points $p_i$ & $\overline{\M}_{Y_0,p_i}$  \\ \midrule
 $p_1\in\mathcal{U}\setminus\{(y=0)\cup(s=0)\}$ & $  \langle  \overline{t} \rangle $ \\ 
$p'_1 \in \mathcal{U}\setminus\{(x=0)\cup(s=0)\}$ & $  \langle  \overline{t} \rangle $  \\
$p_2  =  \mathcal{U} \cap (x=y=0) $ & $\langle \overline{x},\overline{y},\overline{t} ~|~\overline{x}\overline{y}=\overline{t}^b \rangle $ \\ 
$p_3 = \mathcal{U} \cap (x=s=0)$ & $\langle \overline{x},\overline{s},\overline{t} ~|~\overline{x}=\overline{s}\overline{t}^b \rangle $  \\
$p'_3 = \mathcal{U} \cap (y=s=0)$ &  $\langle \overline{y},\overline{s},\overline{t} ~|~\overline{y}=\overline{s}\overline{t}^b \rangle$ \\
$ p_4 = \mathcal{U} \cap (x=y=s=0)$ & $\langle \overline{x},\overline{y},\overline{s},\overline{t} ~|~\overline{x}\overline{y}=(\overline{s}\overline{t})^b \rangle$
\\ \bottomrule \\
\end{tabular}
\caption{Stalks of $\overline{\mathcal{M}}_{Y_0}$}
\label{StalksOfTheGhostSheaf}
\end{table}
where we present a monoid with a set of generators $G$ and relations $R$ among elements of $G$ by $\langle  G~|~R \rangle$.

\subsection{Log curves}
\label{Sec: Log curves}
Gromov-Witten theory has a generalisation to the setting of logarithmic geometry \cite{logGW, logGWbyAC}. In log Gromov-Witten theory one works over a base log scheme $(S,\mathcal{M}_S)$. The scheme $S$ in
practice could be the spectrum of a discrete
valuation ring with the log structure induced by the closed point
(one-parameter degeneration), or it could be $\Spec \kk$, for $\kk$ an algebraically closed field of characteristic zero, endowed with the trivial
log structure (absolute situation), or $\Spec \kk$ endowed with \emph{the standard log structure} (central fibre of one-parameter
degeneration). The standard log structure up to isomorphism is given uniquely by a monoid $Q$ with $Q^\times=\{0\}$ giving rise to the log structure
\[
Q\oplus\kk^\times\lra \kk,\quad
(q,a)\longmapsto \begin{cases} a,&q=0\\ 0,&q\neq0.\end{cases}
\]
on $\Spec\kk$. 
We will restrict our attention to the latter case and take the log point endowed with the standard log structure as a base scheme. Throughout this paper we assume
\[ \kk=\CC \,\ \mathrm{and} \,\ Q:=\NN \]
and denote the standard log point by
\[{\Spec \CC}^{\dagger}:=(\Spec \CC,\NN \oplus \CC^{\times}).\]
One generalises the notion of a stable map to the log setting as follows. Consider an ordinary stable map with
a number, say $\ell$, of marked points. Thus we have a proper
curve $C$ with at most nodes as singularities, a regular map $f: C\to X$, a tuple $\mathbf{x}=(x_1,\ldots,x_\ell)$ 
of closed points in the non-singular locus of $C$. Moreover the triple $(C,\mathbf{x},f)$ is supposed to fulfill the stability condition of finiteness of the group of automorphisms of $(C,\mathbf{x})$ commuting with $f$. To promote such a stable map to a stable log map amounts to endow all spaces with (fine, saturated) log structures and lift all morphisms to morphisms of log spaces. Then $C\to\Spec\CC$ is promoted to a smooth morphism of log spaces
\[
\pi:{C}^{\dagger}\lra \Spec\CC^{\dagger}.
\]
and we have a log morphism $f: {C}^{\dagger}\lra X^{\dagger}$,
where $X^{\dagger}=(X,\M_X)$ and $C^{\dagger}=(C,\M_C)$ are log schemes. Given a morphism of log spaces $f: {C}^{\dagger}\lra X^{\dagger}$, we denote by $f:C \lra X$
the underlying morphism of schemes. 
Throughout this paper we will assume that the arithmetic genus of the domain curve $C$ is zero and thus will work on the Zariski site, rather than the \'etale site which would be needed for more general cases.

Let $x\in X$ be a closed point in  and let $f^{\flat}_{x}:\M_{X,f(x)}\to \M_{C,x}$ be the morphism of monoids induced by $f: {C}^{\dagger}\lra X^{\dagger}$. Then, by the definition of a log morphism \cite[pg 99]{Mark} we obtain the following commutative diagram on the level of stalks
\begin{equation} 
\label{local ring level}
\xymatrix@C=30pt
{
\M_{X,f(x)} \ar[r]^{f^{\flat}_{x}}\ar[d]_{\alpha_{X,f(x)}}
&\M_{C,x} \ar[d]^{\alpha_{C,x}}\\
\mathcal{O}_{X,f(x)} \ar[r]^{f^{\sharp}_{x}}&\mathcal{O}_{C,x}
}
\end{equation}
Let $\kappa:\M_X \xrightarrow{/\mathcal{O}^{\times}_X} \overline{\M}_X$ be the quotient homomorphism. By the commutativity of the above diagram there is a morphism induced by $f$ on the level of ghost sheaves, denoted by 
\[\overline{f}^{\flat}_x:\overline{\mathcal{M}}_{X,f(x)}\rightarrow \overline{\mathcal{M}}_{C,x}\] for a closed point $x\in C$. By abuse of notation, we denote the morphism $\overline{\mathcal{M}}^{gp}_{X,f(x)}\rightarrow {\overline{\mathcal{M}}^{gp}_{C,x}}$
on group level is also by $\overline{f}^{\flat}_x$.

We demand the morphism $f:{C}^{\dagger}\to \Spec \CC^{\dagger}$ to be \emph{log smooth} \cite[Defn. $3.23$]{Mark}. We furthermore demand that the regular points of $C$ where $\pi$ is not strict are exactly the marked points. We will recall the precise shape of such log structures on nodal curves instantly. 
\begin{remark}
After a moment of thought one may conclude that an algebraic stack based on this notion of a log smooth map over the log point $(\Spec \CC, Q\oplus \CC^{\times})$ can never be of finite type, because for a given log map one can always enlarge the monoid $Q$, for example by embedding $Q$ into $Q\oplus\NN^r$. To solve this issue, a basic insight in \cite{logGW, logGWbyAC} is that there is a universal, minimal choice of $Q$. In this \emph{basic monoid} there are just enough generators and relations to lift $f:C\to X$ to a morphism of log spaces while maintaining log smoothness of ${C}^{\dagger}\to \Spec\CC^{\dagger}$. After the usual fixing of topological data (genus, homology class, degree) the corresponding stack of \emph{basic} stable log maps turns out to be a proper Deligne-Mumford stack. For the present paper this general theory is both a bit too general and still a bit too limited. It is too general because we will end up with a finite list of unobstructed stable log maps over the standard log point. In particular, we have a one-parameter smoothing of log maps and there is always a distinguished morphism 
\[ (\Spec\CC,\NN)\to (\Spec\CC, Q)\] 
just coming from our degeneration situation as in \cite{NS}. Therefore, throughout this paper we comfortably assume the basic monoid $Q$ is given by the natural numbers. Moreover, there is no need for working with higher dimensional moduli spaces or with virtual fundamental classes, as the moduli space of the stable log maps over $\Spec \CC^{\dagger}$ form a proper Deligne-Mumford stack of finite type \cite[Cor. $2.8$]{logGW}. The general theory is also too restricted because we will have to admit non-complete domains. The presence of non-complete components requires an ad hoc treatment of compactness of our moduli space that is special to our situation. 
\end{remark}
\begin{definition}
\label{logCurve}
A \emph{log smooth curve} over the standard log point ${\Spec\CC}^{\dagger}$ consists of a fine saturated log scheme $C^{\dagger}:=(C,\mathcal{M}_C)$
with a log smooth, integral morphism $\pi:C^{\dagger} \rightarrow {\Spec\CC}^{\dagger}$
of relative dimension $1$ such that every fibre of $\pi$ is a reduced and connected curve. 
\end{definition}
Note that we do not demand $C$ to be \emph{proper} which is the case in \cite[Defn. $1.3$]{logGW}. If we specify a tuple of sections $\textbf{x}:=(x_1,\cdots,x_l)$ of $ \pi$ so that over the non-critical locus $U\subset  C$ of $\pi$ we have
\[\ol \M_C|_U\simeq \pi^* \ol\M_{{\Spec \CC}^{\dagger}}\oplus \bigoplus_i
{x_i}_* \NN_{\Spec  \CC}\] then we call the log smooth curve ${C}^{\dagger}$ \emph{marked} and denote it by $(C / {\Spec\CC}^{\dagger},\textbf{x})$.
Before going through more details we would first like to make a few remarks on the local structure of log smooth curves over $\Spec \CC^{\dagger}$. Let $0\in \Spec\CC$ be the closed point. Then we have an isomorphism
\begin{eqnarray}
\ol{\M}_{\Spec \CC,0} &\longrightarrow & \NN
\nonumber \\
 \overline{t}^a  & \longmapsto & a
\nonumber
\end{eqnarray} 
Let $\sigma: \NN\to \M_{\Spec \CC,0}$ be the chart for the log
structure on $\Spec \CC$ around $0\in \Spec\CC$, given by
\[ \sigma(q) =
\left\{
	\begin{array}{ll}
		1  & \mbox{if } q = 0 \\
		0 & \mbox{if } q > 0
	\end{array}
\right.
\]
The following crucial theorem is a special case of \cite [p.222]{Kf}, as we restrict our attention only to log smooth curves over the standard log point $\Spec \CC^{\dagger}$.
\begin{theorem} 
\label{Thm: structure of log curves}
Locally $C$ is isomorphic to one of
the following log schemes $V$ over $\Spec \CC$.
\begin{enumerate}
\item[(i)]
$\Spec(\CC[z])$ with the log structure induced from the homomorphism
\[
\NN\lra \O_V,\quad q\longmapsto \sigma(q).
\]
\item[(ii)]
$\Spec(\CC[z])$ with the log structure induced from the homomorphism
\[
\NN\oplus\NN\lra \O_V,\quad (q,a)\longmapsto z^a \sigma(q).
\]
\item[(iii)]
$\Spec(\CC[z,w]/(zw-t))$ with $t\in\maxid$, where $\maxid$ is the maximal ideal in $\CC$ and with the log structure
induced from the homomorphism
\[
\NN\oplus_\NN \NN^2\lra \O_V,\quad
\big(q,(a,b)\big)\longmapsto \sigma(q) z^a w^b.
\]
Here $\NN\to\NN^2$ is the diagonal embedding and $\NN\to \NN$,
$1\mapsto\rho_q$ is some homomorphism uniquely defined by $C\to \Spec\CC$.
Moreover, $\rho_q\neq 0$.
\end{enumerate}
In this list, the morphism $C^{\dagger}\to \Spec\CC^{\dagger}$ is represented by the
canonical maps of charts $\NN\to \NN$, $\NN\to \NN\oplus\NN$
and $\NN\to \NN\oplus_\NN \NN^2$, respectively where we identify the domain $\NN$ always with the first factor of the image. 
\end{theorem}
In Theorem \ref{Thm: structure of log curves} cases $\mathrm{(i),(ii),(iii)}$ correspond to neighbourhoods of general points, marked points and nodes of $C$ respectively. 
Nodes and marked points of a log smooth curve $C^{\dagger} \to \Spec \CC^{\dagger}$ are referred to as \emph{special points} of $C$.
Now we are ready to generalise in the next section the notion of a log map over $\Spec \CC^{\dagger}$ as in \cite[\S $1$]{logGW} to the case where the domain includes some special non-complete components. 

\subsection{Log corals}
\label{Log corals on $Y_0$}
Let $Y_0=\widetilde{T}_0 \times \AA^1$ be the central fiber of the degeneration of the unfolded Tate curve defined in \S \ref{The Tate curve and its unfolding}. In this section we define and discuss the properties of a special kind of log maps on $Y_0$, with non-complete components.
\begin{definition}
\label{log map}
Let $(C/\Spec \CC^{\dagger},\textbf{x})$ be a marked log smooth curve. A morphism of log schemes $f: C^{\dagger} \rightarrow Y_0^{\dagger}$ is called a \emph{log map} if the following holds:
\begin{itemize}
\item[(i)] The morphism $f$ fits into the following commutative diagram
\begin{equation} 
\xymatrix@C=30pt
{
{C}^{\dagger} \ar[r]^{f}\ar[rd]_{\pi}
&(Y_0, \M_{Y_0})\ar[d]\\
&{\Spec\CC}^{\dagger}
}
\nonumber
\end{equation}
\item[(ii)] For each non-complete irreducible component $C'\subset C$, there is an isomorphism $C'\cong \mathbb{A}^1$.
\item[(iii)] The map $s\circ \restr{f}{C'}:C'\rightarrow \mathbb{A}^1$ is dominant, where $s:Y_0\rightarrow \AA^1_s$
is the projection map.
\item[(iv)] For each marked point $p_i\in C$, we have $f(p_i)\in (Y_0)_{s=0}$ where $(Y_0)_{s=0}$ is the fiber over $0\in \mathbb{A}^1_s$ under the map $s:Y_0 \rightarrow \mathbb{A}^1$. 
\item[(v)] The stability condition holds, that is, the automorphism group $\mathrm{Aut}(C / {\Spec\CC}^{\dagger},\textbf{x},f)$ is finite. 
\end{itemize}
\end{definition}
To set up the log counting problem in the next section, analogously as in the tropical counting problem, we restrict our attention to \emph{general} log maps defined as follows.
\begin{definition}
\label{general log map}
A log map $f:{C}^{\dagger}\to {Y_0}^{\dagger}$ is called \emph{general} if each non-complete component $C'\subset C$ has only one special point and each complete component contains at most three special points.
\end{definition}
For a general log map if a non-complete component $C'\subset C$ has a unique special point which is a marked point, then by the connectivity of the domain $C$ will have only one component $C'$. As this case can be treated easily throughout the next sections we omit it and assume the unique special point on each non-complete component on a general coral is a node.
\begin{remark}
\label{varphi}
Let $C'\subset C$ be a non-complete component with generic point $\eta$. Denote the function field of $C'=\mathbb{A}^1$ by  
\[\mathcal{O}_{C',\eta}=k(z)\]
Recall the log structure on $Y_0$ from \S \ref{The log structures on the Tate curve and its degeneration}. Let $Z\subset Y_0$ be the smallest toric stratum containing $f(C')$. Then, we have one of the two following cases.
\begin{itemize}
\item[i)] $\dim \overline{Z}=2$: in this case we call  $f$ \emph{transverse} at $C'$.
A chart for the log structure $\M_{Y_0}$ around $f(\eta)$ is given by 
\begin{align}
\overline{\M}_{Y_0}|_{\mathcal{U}}&\rightarrow\M_{Y_0}|_{\mathcal{U}}
\nonumber\\
\overline{t}&\mapsto s_t
\end{align}
for an open neighbourhood $\mathcal{U}$ of $f(\eta)$. Note that on $\mathcal{U}$ we have either $y \neq 0$ or $x \neq 0$. Throughout this section without loss of generality we assume we are in the former case. Recall the log structure on $Y_0$ described in \S \ref{The log structures on the Tate curve and its degeneration}. The morphism
\begin{eqnarray}
f^{\flat}_{\eta}:\M_{Y_0,f(\eta)} & \longrightarrow & \M_{C,\eta}
\nonumber \\
s_x &\longmapsto & \varphi_x(z)
\nonumber \\
s_s &\longmapsto & \varphi_s(z)
\nonumber 
\end{eqnarray}
is only determined by $\varphi_x,\varphi_s \in k(z)\setminus \{0\}$. Note that as $s$ is invertible on $\mathcal{U}$, we have  $f^{\sharp}_{\eta}(s_s)=f^{\flat}_{\eta}(s_s)$ where $f^{\sharp}_{\eta}:\mathcal{O}_{Y_0,f(\eta)}\to \mathcal{O}_{C,\eta}$ is the morphism on stalk level induced by the scheme theoretic map $\ul f$. Moreover, as $\dim \overline{Z}=2$, it is not contained in the double locus of $Y_0$ and hence the restriction of the morphism $(Y_0,\M_{Y_0}) \to \Spec \CC^{\dagger}$ to the complement of the double locus in $\overline{Z}$ in definition \ref{log map} is strict, which means $\M_{Y_0}$ is isomorphic to the log structure induced by the pull-back of the log structure on $\Spec \CC^{\dagger}$ and hence the image of $s_t$ is determined as $f^{\flat}_{\eta}(s_t)=s_t$.
\item[ii)] $\dim \overline{Z}=1$: in this case we call $f$ \emph{non-transverse} at $C'$. As $s_x s_y =s^b s_t^b$ where $s \in \O^{\times}_{Y_0}(\mathcal{U})$ is invertible, a chart for the log structure $\M_{Y_0}$ around $f(\eta)$ is given by
\begin{eqnarray}
\nonumber
{\overline{\M}_{Y_0}|}_{\mathcal{U}}&\longrightarrow & {\M_{Y_0}|}_{\mathcal{U}}
\nonumber\\
\overline{x}&\longmapsto & s_x
\nonumber \\
\overline{y}&\longmapsto & s^{-b} s_y
\nonumber \\
\overline{t}&\longmapsto & s_t
\nonumber
\end{eqnarray}
for an open neighbourhood $\mathcal{U}$ of $f(\eta)$, where $\overline{x}\cdot \overline{y}=\overline{t}^b$.  
The morphism
\begin{eqnarray}\label{alfabeta}
f^{\flat}_{\eta}:\M_{Y_0,f(\eta)} &\longrightarrow & \M_{C,\eta}
\nonumber \\
s_x &\longmapsto & \varphi_x(z)\cdot t^{\alpha}
\nonumber \\
s_y &\longmapsto & \varphi_y(z)\cdot t^{\beta}
\nonumber \\
s=s_s &\longmapsto & \varphi_s(z)
\end{eqnarray}
is determined by 
$\varphi_x,\varphi_y,\varphi_s \in k(z)\setminus \{0\}$ and $\alpha,\beta\in\NN\setminus \{0\}$ such that $\alpha+\beta=b$ and $\varphi_x\varphi_y=\varphi_s^b$.
\end{itemize}
\end{remark}
The following condition will be crucial while showing in the upcoming Theorem \ref{Log curves define tropical corals}, the necessary balancing requirements of the tropical analogues of the log invariants we work with in this paper.
\begin{definition}
\label{parallel}
Let $f:C^{\dagger} \to Y_0^{\dagger}$ be a log map and let $C'$ be a non-complete component of $C$. Let $val_{\infty}: k(z) \rightarrow \ZZ$ be the valuation at $z=\infty$. We call $f$ \emph{parallel} at $C'$ if the following holds.
\begin{itemize}
\item[i)] $val_{\infty}(\varphi_x)=0$, if $f$ is transverse at $C'$.
\item[ii)] $val_{\infty}(\varphi_x^{\beta} / \varphi_y^{\alpha})=0$, if $f$ is non-transverse at $C'$ .
\end{itemize}
where $\varphi_x, \varphi_x^{\beta}, \varphi_y^{\alpha}$ are defined as in Remark \ref{varphi}.
\end{definition}
\begin{definition}
\label{log coral}
A log map $f: C^{\dagger} \rightarrow {Y_0}^{\dagger}$ is called a \emph{log coral} if $f$ is \emph{parallel} at each non-complete component of $C$.
\end{definition}
We will sometimes replace $Y_0$ by an open subset of it containing $\im f$ or by the $\ZZ$-quotient of it in \S \ref{log count equals tropical count}. The definition of log corals then still makes sense with the obvious modifications.


\subsection{The count of log corals}
\label{The set up of the counting problem}
In this section we define the count of log corals $f:{C}^{\dagger}\to {Y_0}^{\dagger}$. We assume the charts for the log structure $\M_{Y_0}$ are given as in \S \ref{The log structures on the Tate curve and its degeneration}. We define charts for $\M_{C}$ using Theorem \ref{Thm: structure of log curves}. Note that, for generic points we have the isomorphism 
\begin{align*}
{\overline{\M}}^{gp}_{C,\eta}\rightarrow \ZZ
\nonumber\\
{\bar{t}}^a\mapsto a
\nonumber
\end{align*}
and for marked points $p$ of $C$ have the isomorphism 
\begin{align*}
{\overline{\M}}^{gp}_{C,p}\rightarrow \ZZ\oplus \ZZ
\nonumber\\
({\bar{t}}^a,{\bar{z}}^b)\mapsto (a,b).
\nonumber
\end{align*}
We denote by 
\begin{equation}
    \label{Pr for log}
    \pr_i:\overline{\M}_{C,p}^{gp} \to \ZZ
\end{equation}
the projection to the $i$-th factor for $i=1,2$, and consider the following charts for $\M_{C}$.
\begin{itemize}
\label{ChartsForC}
\item[(i)] For a generic point $\eta$ of an irreducible component of $C$;
\begin{eqnarray}
\overline{\M}_{C,\eta} & \longrightarrow & \M_{C,\eta}
\nonumber \\
\bar{t} & \longmapsto & s_t
\nonumber 
\end{eqnarray}
\item[(ii)] For a marked point $p$ of $C$;
\begin{eqnarray}
\overline{\M}_{C,p} & \longrightarrow & \M_{C,p}
\nonumber \\
\bar{t} & \longmapsto s_t
\nonumber \\
\bar{z} & \longmapsto s_z
\nonumber
\end{eqnarray}
\end{itemize}

To define the count of log corals, we describe incidence conditions given by a degree $\Delta$ and an asymptotic constraint $\lambda$. For this, given a log coral $f:C^{\dagger}\to Y_0^{\dagger}$, we first discuss how to assign an element of $N$ to marked points and to non-complete irreducible components of $C$, following the general scheme in \cite{logGW}. 

\begin{proposition}
\label{Prop:contact order}
Let $f:C^{\dagger}\to Y_0^{\dagger}$ be a log coral. 
\begin{itemize}
\item[(i)] For each marked point $p_i$ of $C$, there is a corresponding element $u_{i} \in N$ recording the logarithmic contact order of $p_i$ with $Y_0$.
\item[(ii)]  Each non-complete component $C'\simeq\AA^1$ of $C$ determines an integral point in the boundary of the truncated cone $ \trcone \RR$.
\end{itemize}
\end{proposition}
\begin{proof}
Recall the charts for the log structure on $Y_0$ from \S \ref{The log structures on the Tate curve and its degeneration}. First, given a marked point $p_i$, define the associated element of $N$ as follows. Let $\Xi\in\RR$ be the interval such that the corresponding toric chart
\[
(C(\trcone\Xi))^\vee\cap (M\oplus\ZZ)\lra \Gamma(U_\Xi,\M_Y)
\]
covers $f(p_i)$. Taking the composition with $\ol{f^\flat_{p_i}}$, yields a homomorphism of monoids
\[
(C(\trcone\Xi))^\vee\cap (M\oplus\ZZ)\lra \ol \M_{Y_0,f(p_i)}
\stackrel{\ol{f^\flat_{p_i}}}{\lra} \ol\M_{C,p_i}=\NN\oplus\NN.
\]
The two factors of $\NN\oplus\NN$ in $\ol\M_{C,p_i}$ are generated by the smoothing parameter $s_t$ and by the equation defining $p_i$ in $C$, respectively. Taking the composition with the projection to the second factor thus defines a homomorphism
\[
u_i:(C(\trcone\Xi))^\vee\cap (M\oplus\ZZ)\lra \NN,
\]
that is, an element of $C(\trcone\Xi)$. Since by definition this homomorphism maps $(0,1)\in M\oplus\ZZ$ to $0$, it lies in $(0,1)^\perp= N\oplus\{0\}\subset N\oplus\ZZ$. Note that, identifying $N\oplus\{0\}$ with $N$, the intersection of $C(\trcone\Xi)$ with $N_\RR\oplus\{0\}$ agrees with the asymptotic fan of $\trcone\RR$, endowed with the polyhedral decomposition $\trcone \P$. We have thus defined, for each marked point $p_i\in C$, an element $u_i\in N$, which is contained in the cone of the relevant cell of the asymptotic fan. The fact that $u_i$ records the contact orders follows basically from definitions -- see for instance \cite[pg. 12]{Utah}. 

To show the second part, given a non-complete component $C'\simeq\AA^1\subset C$, with generic point $\eta$, we define a point in $\partial \trcone \RR$ as follows. Let $\Xi\in\RR$ be the interval with the corresponding log chart containing $f(\eta)$. The composition with $\ol{f^\flat_\eta}$ defines a homomorphism
\begin{equation}
\label{phi-eta}
\phi_\eta: (C(\trcone \Xi))^\vee\cap (M\oplus\ZZ) \lra \ol\M_{Y,f(\eta)}
\stackrel{\ol{f^\flat_\eta}}{\lra} \ol\M_{C,\eta}=\NN.
\end{equation}
Since $f$ is a log morphism relative to the standard log point, $\phi_\eta(0,1)=1$. Thus viewed as an element of $N\oplus\ZZ$, we have $\phi_\eta\in N\oplus\{1\}$. 
\end{proof}
Let $f: C^{\dagger} \rightarrow {Y_0}^{\dagger}$ be a log coral and $C' \subset C$ be a non-complete irreducible component of $C$ with generic point $\eta$.
Let
\begin{equation}
\label{ueta}
u_\eta= \phi_\eta-(0,1)
\end{equation}
as an element of $N= N\oplus\{0\}\subset N\oplus\ZZ$, where $\phi_\eta$ is the map in \eqref{phi-eta}. We will define the degree by multiplying $u_\eta$ with the \emph{branching order at infinity}, which we discuss in a moment.
 The restriction ${f|}_{C'}: C' \to Y_0$ factors over an irreducible component $\mathbb{A}^1 \times \mathbb{P}^1 \subset Y_0$. Hence, we have a morphism $f':C' \rightarrow \mathbb{A}^1 \times \mathbb{P}^1$ satisfying
\begin{equation} 
\label{factoring}
\iota \circ f' = {f|}_{C'}
\end{equation}
where $\iota: \mathbb{A}^1 \hookrightarrow \mathbb{P}^1 \subset Y_0$ is the inclusion map. Let  
\begin{eqnarray}
\varphi&:=&{\pr}_1\circ f': C' \rightarrow \mathbb{A}^1
\nonumber\\
\phi&:=&{\pr}_2\circ f': C' \rightarrow \mathbb{P}^1
\nonumber
\end{eqnarray}
where $\pr_1$ and $\pr_2$ are the projection maps from $\AA^1\times \PP^1$ onto the first and second factors respectively. 
The map $\varphi:\AA^1\rightarrow \AA^1$ is identically equal with $s\circ f|_{C'}$ as in \ref{log map}, so it is dominant. Thus, it extends uniquely to a morphism $\tilde{\varphi}:\PP^1\rightarrow \PP^1$ where the domain of $\tilde{\varphi}$ is the completion $\{ \infty_{C'} \} ~ \cup ~ C' \cong \PP^1$ of $C'$ along the point at infinity on $C'\cong \AA^1$. We compactify $\AA^1\times \PP^1 \subset Y_0$ to  $\PP^1\times\PP^1$ along \emph{the divisor at infinity} $D_{\infty}=\PP^1$, and refer to
\begin{equation}
\label{point at infinity}
q_{\infty}  =  \tilde{\varphi}(\infty_{C'}) \in D_{\infty}.
\nonumber
\end{equation}
as \emph{the point at infinity}. By Hurwitz's formula $\tilde{\varphi}:\PP^1\rightarrow \PP^1$ is a covering totally branched at infinity. The branching order at infinity is defined as follows.
\begin{definition}
\label{branching order}
Let $f:C^{\dagger}\to Y_0^{\dagger}$ be a log coral with a non-complete component $C'\subset C$ of generic point 
$\eta$ and let $\varphi=s\circ f|_{C'}$ be the map defined in Definition \ref{log map},$(iii)$. We say $C'$ has \emph{branching order} $w_\eta\in \NN\setminus \{0\}$ if the degree of the covering $\tilde{\varphi}:\PP^1\rightarrow \PP^1$ equals $w_\eta$. Given $C' \subset C^\dagger$, the branching order at infinity is obtained by pulling back $s\in \O(Y)$, to $C'$ and taking the negative of the valuation at the point at infinity, $\infty_{C'}$, of $C'$:
\begin{equation}
\label{weta}
w_\eta  =  -\operatorname{val}_\infty(f^\sharp(s))
\end{equation}
Note that $w_\eta\in\NN\setminus\{0\}$ since $f^\sharp(s)$ is a non-constant regular function on $C'\simeq \AA^1$.
\end{definition}
\begin{remark}
Let
\begin{eqnarray}
\upsilon: \Spec \CC[w] \setminus \{0\} & \to &  \Spec \CC[z] \setminus \{0\}
\nonumber \\
w & \mapsto & \frac{1}{z} 
\nonumber
\end{eqnarray}
and 
\begin{eqnarray}
\upsilon^*: \CC(z) & \to &   \CC(w)
\nonumber \\
r & \mapsto & r \circ \upsilon 
\nonumber
\end{eqnarray}
be the morphism on fraction fields induced by $\upsilon$. Let $\varphi_x$ and $\varphi_x^{\beta} / \varphi_y^{\alpha}$ be defined as in Remark \ref{varphi}. Then, the position of $q_{\infty} \in D_{\infty}$, referred to as the \emph{virtual position at infinity}, is given by
\begin{itemize}
\label{virtual position}
\item[(i)] $ \upsilon^*(\varphi_x)(0)$ if $f$ is transverse at $C'$
\item[(ii) ] $\upsilon^*(\varphi_x^{\beta} / \varphi_y^{\alpha})(0)$ if $f$ is non-transverse at $C'$
\end{itemize}
\end{remark}
Now, we can define our log constraints.
\begin{definition}
Let $\Delta$ be a degree as in Definition 
\ref{degree} with $\ell+1$ positive and $m$ negative entries.
A log coral $f:{C}^{\dagger}\to{Y_0}^{\dagger}$ with $\ell+1$ marked points 
\[p_0,p_1,\ldots,p_\ell\] 
and $m$ non-complete components 
\[C'_1,\cdots,C'_m\] 
is said to be of degree $\Delta$ if the following conditions hold 
\begin{enumerate}
\item[(i)] For each marked point $p_i$ of $C$, the composition 
\[\mathrm{pr}_2\circ \overline{f}^{\flat}_{p_i}:\overline{\M}_{{Y_0},f(p_i)}^{gp} \rightarrow \overline{\M}_{C,p_i}^{gp}\rightarrow \ZZ,\]
where $\mathrm{pr}_2$ is the projection map to the second factor defined as in \eqref{Pr for log}, equals $\overline{\Delta}^i$.
\item[(ii)] For each non-complete component $C_j'\subset C$ with generic point $\eta_j$, we have $w_{\eta_j}\cdot u_{\eta_j}$ equals $\underline{\Delta}^j$, where $w_{\eta_j}$ and $u_{\eta_j}$ are defined as in \eqref{ueta} and \eqref{weta} .
\end{enumerate}
\end{definition}



\begin{definition}
\label{coral_match}
Let $\Delta$ be a degree as in Definition 
\ref{degree} with $\ell+1$ positive and $m$ negative entries. Let 
$\lambda$ be an asymptotic constraint for 
$\Delta$, as in Definition \ref{asymptotic constraint}.
We say that a log coral $f:{C}^{\dagger}\to {Y_0}^{\dagger}$ \emph{matches} $(\Delta,\lambda)$
if it is of degree $\Delta$, and denoting 
$p_0,\dots,p_\ell$ the marked points of $C$, 
and $\eta_0,\dots,\eta_\ell$ the generic points of the irreducible components $C_0,\dots,C_\ell$
of $C$ containing $p_0,\dots,p_\ell$,
the image of 
\[(\bar{f}_{\eta_j}:\overline{\M}_{Y_0,f(\eta_j)}\rightarrow \overline{\M}_{C,\eta_j}\cong \NN) \in N \]
under the quotient map 
\[ N\rightarrow N / (N \cap \RR \overline{\Delta}^j)\] 
is equal to $\lambda^j$ for all $j=1,\ldots,\ell$.
\end{definition}
\begin{remark}
An asymptotic constraint 
$\lambda$ constrains the marked points 
$p_1,\dots,p_\ell$, but not the marked point 
$p_0$.
\end{remark}

\begin{discussion}
\label{DiscussionOfLogConstraints}
Let $\Delta$ be a degree
with 
$\ell+1$ positive and $m$ negative entries, 
$\lambda$ an asymptotic constraint for 
$\Delta$, and $f:C^{\dagger} \to Y_0^{\dagger}$ a log coral matching $(\Delta,\lambda)$.
For every $1\leq i\leq \ell$, let $C_i \subset C$ be the complete irreducible component containing the marked point 
$p_i$. Let $\tilde{\lambda}_i\in N_\QQ$ be a lift of $\lambda_i$ under the quotient homomorphism $N_\QQ\to N_\QQ/\QQ\cdot \overline{\Delta}^i$. Then, the pair $(\overline{\Delta}^i,\lambda_i)$ defines the two-plane
\[
H_{\overline{\Delta}^i,\lambda_i}= \RR\cdot(\overline{\Delta}^i,0)+ \RR\cdot (\tilde{\lambda}_i,1)\subset N_\RR\oplus\RR.
\]
Let $m_{\overline{\Delta}^i,\lambda_i}$ span the rank one subgroup $H_{\overline{\Delta}^i,\lambda_i}^\perp \cap (M\oplus\ZZ)$ of $M\oplus\ZZ$. To fix signs we ask $m_{\overline{\Delta}^i,\lambda_i}$ to evaluate positively on $(1,0,0)$. The rational function $z^{m_{\overline{\Delta}^i,\lambda_i}}$ on $Y$ induces a section, denoted by
\[s_{\overline{\Delta}^i,\lambda_i}\in\Gamma(Y_0,\M_{Y_0}^\gp).\]
Since $m_{\overline{\Delta}^i,\lambda_i}$ spans $H_{\overline{\Delta}^i,\lambda_i}^\perp \cap (M\oplus\ZZ)$, we have $f^\flat_p(s_{\overline{\Delta}^i,\lambda_i}) \in \O_{C,p}^\times$. It thus makes sense to evaluate at $p_i$ to obtain a non-zero complex number
\[
f^\flat_p(s_{\overline{\Delta}^i,\lambda_i})(p_i)\in \CC^\times\,.
\]
Similarly, for every 
$1\leq j \leq m$, we consider the non-complete component $C'_j\subset C$, which determines a point $v_j^-\in \partial \trcone \RR$ by Proposition \ref{Prop:contact order}. Then, we take the two-plane spanned by $(\underline{\Delta}^j,0)$ and $(v_j^-,1)$. 
Analogously, as above we denote the monomial obtained from the pair $(\underline{\Delta}_j,v_j^-)$ by $m_{\underline{\Delta}_j,v_j^-}$ and the corresponding section by $s_{\underline{\Delta}_j,v_j^-}$. Now a log coral $f:C^{\dagger}\to Y_0^{\dagger}$ with a non-complete component $C'_j\subset C$ with generic point corresponding to $v_j^-$ is parallel,  as in Definition \ref{parallel}, if and only if $f^\flat(s_{\underline{\Delta}_j,v_j^-})$ extends over the point at infinity of $C'_j$. In this case we can thus obtain a well-defined complex number as
\[
f^\flat_p(s_{\underline{\Delta}_j,v_j^-})(\infty_{C'_j})\in \CC^\times\,.
\]
Indeed, $f^\flat_p(s_{\underline{\Delta}_j,v_j^-})=\varphi_x$ if $f$ is transverse at $C'_j$ and $f^\flat_p(s_{\underline{\Delta}_j,v_j^-})=\varphi_{x^{\beta}/y^{\alpha}}$ if $f$ is non-transverse at $C'$ where $\varphi_x$ and $\varphi_{x^{\beta}/y^{\alpha}}$ are defined as in Remark \ref{varphi} respectively. 
\end{discussion}
\begin{definition}
\label{log constraint}
We call a tuple 
\[\rho=(\overline{\rho}^{\ell},\underline{\rho}^m) \subset {(\CC^{\times})}^{\ell}\times {(\CC^{\times})}^{m}\] 
a \emph{log constraint of order $(\ell,m)$}. A log coral $f: C^{\dagger}\to \overline{Y_0}^{\dagger}$ with $\ell+1$ marked points and $m$ non-complete components, for $\ell,m\in \NN\setminus \{0\}$ matches a log constraint of order $(l,m)$ if the following holds. 
\begin{itemize}
\item[(i)] $C$ has $\ell+1$ marked points $p_0,\ldots,p_{\ell}$ such that \[f^\flat_{p_j}(s_{\overline{\Delta}^i,\lambda_i})(p_i)=\overline{\rho_i}\] for $1\leq i\leq \ell$, that is for all but one marked points.
\item[(ii)] $C$ has $m$ non-complete components $C'_1,\ldots,C'_m$ with generic points $\eta_j$ and 
\[f^\flat_{\eta_j}(s_{\underline{\Delta}_j,v_j^-})(\infty_{C'_j})=\underline{\rho}_j\,\] for $1\leq j\leq m$, where $\infty_{C_j'}$ is the point at infinity on $C'_j$.
\end{itemize}
where $s_{\overline{\Delta}^i,\lambda_i}$ and $s_{\underline{\Delta}_j,v_j^-}$ are defined as in Discussion \ref{DiscussionOfLogConstraints}.
\end{definition}
Note that the degree and the 
asymptotic constraint for log corals can be interpreted tropically as the tropical degree and asymptotic constraint for tropical corals. This will be obvious once we explain tropicalizations of log corals as tropical corals
in \S \ref{tropicalizations}. To get a well defined count, we assume we fix the degeneration order of log corals, so that after tropicalizing the corresponding tropical corals match asymptotic constraints that are inside the stable range defined in Definition \ref{stable range}.  
\begin{definition}
\label{log count}
Let $\Delta$ a degree with $\ell+1$ positive and $m$ negative entries, $\lambda$ an asymptotic constraint for $\Delta$ and $\rho$ a log constraint of order $(\ell,m)$. We say that a log coral $f:C^{\dagger}\to Y_0^{\dagger}$
\emph{matches} $(\Delta,\lambda,\rho)$
if it matches $(\Delta,\lambda)$ as in Definition \ref{coral_match} and $\rho$ as in Definition \ref{log constraint}.

Let $\mathfrak{L}_{\Delta,\lambda,\rho}$ be the set of all log corals $f:C^{\dagger}\to Y_0^{\dagger}$ matching 
$(\Delta,\lambda,\rho)$. Define $N^{log}_{\Delta,\lambda,\rho}$ to be the cardinality of $\mathfrak{L}_{\Delta,\lambda,\rho}$:
\begin{equation}\label{log_curves_count}
N^{log}_{\Delta,\lambda,\rho}:= |\mathfrak{L}_{\Delta,\lambda,\rho}|  \,.  
\end{equation}
\end{definition}
If $N^{log}_{\Delta,\lambda,\rho}$ is finite this defines a number giving us a count of log corals. In \ref{main theorem} we will see that, indeed it is a finite number which equals the tropical count defined in \ref{The count of tropical corals}.
\begin{remark}
A priori, $\mathfrak{L}_{\Delta,\lambda,\rho}$ has the structure of a stack, and it is a closed substack of a larger algebraic stack of log maps with some non-complete components. In the present case we are in the comfortable situation that due to unobstructedness of deformations, $\mathfrak{L}_{\Delta,\lambda,\rho}$ turns out to be a reduced scheme over $\CC$ of finite length, hence really does not carry more information than the underlying set.
\end{remark}

\section{Tropicalizations of Log Corals}
\label{tropicalizations}
In this section we prove that the tropicalization of a log coral is a tropical coral. We first recall the associated tropical space to a log scheme constructed as in \cite[Appendix B]{logGW}.

\begin{definition}
\label{tropical space}
Let $X^{\dagger}$ be a log scheme endowed with the Zariski topology. The tropical space associated to $X^{\dagger}$ denoted by $\mathrm{Trop}(X)$ referred to as the \emph{tropicalization of $X^{\dagger}$} is defined as
\[
\Trop(X):=\Bigg(\coprod_{ x\in \ul X} \Hom(\ol\M_{X, x},
\RR_{\ge 0})\Bigg)\bigg/\sim
\]
where the disjoint union is over all scheme-theoretic points $x$ of
$\ul X$ and the equivalence relation is generated by the
identifications of faces given by dualizing generization maps
$\ol\M_{X, x} \to\ol\M_{X, x'}$ when $x$ is a specialization of $x'$ \cite[Appendix B]{logGW}.
One then obtains for each $ x$ a map
\[
i_x:\Hom(\ol\M_{X, x},\RR_{\ge 0})\to \Trop(X)\,,
\]
which is injective since $\ol\M_X$ is fine in the Zariski topology.
\end{definition}
The tropicalization of log schemes is functorial.
Indeed, if $f \colon X_1^{\dagger} \rightarrow X_2^{\dagger}$
is a map of log schemes, then for every 
closed point $x \in X_1$, the map $\bar{f}^{\flat}_x:\overline\M_{X_2,f(x)} \rightarrow\overline\M_{X_1,x}$ induces a morphism $\Hom(\overline\M_{X_1,x},\RR_{\ge 0}) \rightarrow
\Hom(\overline\M_{X_2,f(x)},\RR_{\ge 0})$. 
These morphisms are compatible with the equivalence relations defining $\Trop(X_1)$ and $\Trop(X_2)$, and so we obtain a morphism
\begin{equation}
f^{\mathrm{trop}}: \mathrm{Trop}(X_1) \to \mathrm{Trop}(X_2)
\nonumber
\end{equation}
referred to as the tropicalization of $f$.

Now, we restrict our attention to the current case of interest in this paper, and consider $Y_0$, the fibre over $t=0$ of the degeneration of the unfolded Tate curve $Y\to\Spec \CC[s,t]$,
and a log coral $f: C^{\dagger} \rightarrow {Y_0}^{\dagger}$.
By functoriality of the tropicalization applied to the 
the commutative diagram in Definition \ref{log map}, and using that
$\mathrm{Trop}({\Spec\CC}^{\dagger})= \RR_{\geq 0}$
by Definition \ref{tropical space}, we obtain a commutative diagram \begin{equation} \label{diag}
\xymatrix@C=30pt
{
\mathrm{Trop}(C) \ar[r]^{f^{\mathrm{trop}}}\ar[rd]_{h_C}
&\mathrm{Trop}(Y_0) \ar[d]^{h_{Y_0}}\\
&\RR_{\geq 0}
}
\end{equation}

We describe below $\mathrm{Trop}(C)$
and $\mathrm{Trop}(Y_0)$ explicitly.
First, define the \emph{dual graph $\Gamma$ of C} as follows.
\begin{construction}
\label{dual graph}
\begin{itemize}
\item[(i)] Define the set of vertices $V(\Gamma)$ so that there is a vertex $v_j \in V(\Gamma)$ for each irreducible component $C_j\subset C$. We denote by $V^-(\Gamma) \subset 
V(\Gamma)$ the set of vertices $v$ such that the corresponding component $C_v$ is non-complete.
\item[(ii)] Define the set of bounded edges $E^b(\Gamma)$ so that an edge $e_{ij}\in E^b(\Gamma)$ connecting two vertices $v_i$ and $v_j$ exists if and only if the corresponding irreducible components $C_i\subset C$ and $C_j\subset C$ intersect.
\item[(iii)] Define the set of unbounded edges denoted by $E^{+}(\Gamma)$ so that for each edge $e_j\in E^{+}(\Gamma)$ adjacent to a vertex $v_j\in V(\Gamma)$ there exist a marked point $p_j\in C_j\subset C$ on the irreducible component $C_j$ corresponding to $v_j$.
\end{itemize} 
Note that the dual graph $\Gamma$ of $C$ has a natural structure of coral graph as in Definition 
\ref{Def: coral graph}, where $E^+(\Gamma)$ is the set of positive edges and $V^-(\Gamma)$ is the set of negative vertices.
\end{construction}

\begin{proposition} \label{prop_trop_C}
Let $f:C^{\dagger} \to Y_0^{\dagger}$ be a log coral. Then, the tropicalization $\mathrm{Trop}(C)$ of the domain $C^{\dagger}$ is the cone $C(\Gamma)$ over the dual graph $\Gamma$ of $C$. 
\end{proposition}
\begin{proof}
Recall from Theorem \ref{Thm: structure of log curves} that there are three possibilities of closed points $x\in C$, with the stalk $\overline{\M}_{C,x}$ is either $\NN$, $\NN\oplus\NN$ or $\NN\oplus_\NN \NN^2$, for $x$ a generic point, a marked point or a node, respectively. Then, the vertices of $\mathcal{G}$ correspond to generic points $\eta$, unbounded edges (or the flags of the graph $\mathcal{G}$) correspond to the marked points and bounded edges correspond to nodes. So, the result is a direct consequence of the Definition \ref{tropical space}. 
\end{proof}
\begin{proposition} \label{prop_trop_Y}
Let $Y\to\Spec \CC[s,t]$ be the degeneration of the unfolded Tate curve, and $Y_0$ be the central fiber over $t=0$. Then,
\begin{itemize} 
\item[(i)]The tropicalization $\mathrm{Trop}(Y)$ of $Y$ is the cone $C(\trcone \RR)$ over the truncated cone $\trcone \RR$ endowed with the polyhedral decomposition $\trcone \P_b$.
\item[(ii)] $ \mathrm{Trop}(Y)=\mathrm{Trop}(Y_0) $
\end{itemize}
\end{proposition}
\begin{proof}
Recall from \S \ref{The unfolded Tate curve} that $Y$ is the toric variety associated to $(\trcone \P_b,\trcone \RR)$. Since the log structure on Y is fine and constant along any open toric strata it follows from Definition \ref{tropical space} that $\mathrm{Trop}(Y)$ is the cone over the fan associated to $Y$. Note that this holds for any toric variety, hence in particular for $Y$. This proves $(i)$. To show $(ii)$, observe that, given any closed point $x\in Y$, we have $\overline{x}\cap Y_0 \neq \emptyset$. Hence, $(ii)$ follows. 
\end{proof}

Let $f:C^{\dagger} \to Y_0^{\dagger}$ be a log coral.
As $\mathrm{Trop}(C)=C(\Gamma)$ by Proposition \ref{prop_trop_C}
and $\mathrm{Trop}(Y_0)=C(\trcone \RR)$
by Proposition \ref{prop_trop_Y}, the restriction over 
$1 \in \RR_{\geq 0}$ of $f^{\mathrm{trop}}$ in the diagram 
\ref{diag}
defines a map 
\begin{equation}
\label{eq_h} h \colon \Gamma \rightarrow \trcone \RR\,.
\end{equation}

We also define a weight function $w$ on $\Gamma$ as follows.
\begin{construction}
\label{weight function}
\begin{itemize}
\item[(i)] Let $e_q \in E^b(\Gamma)$ be a bounded edge corresponding to a node $q\in C$. By Theorem 
\ref{Thm: structure of log curves}(iii), the log structure of $C$ at $q$ is determined by a positive integer $\rho_q$.
Let $v_1, v_2\in V(\Gamma)$ be the two vertices of $e_q$. 
Then, there exists $u_q \in N$ such that $h(v_1)-h(v_2)=\rho_q u_q$, and we define the \emph{weight} $w_\Gamma(e_q)$ as the divisibility 
of $u_q$ in $N$, that is the biggest positive integer such that $u_q/w_\Gamma(e_q) 
\in N$ (for details, see \cite[Section $1.4$]{logGW}).

\item[(ii)] Let $e_p \in E^+(\Gamma)$ be an unbounded edge corresponding to a marked point $p \in C$. 
Let $u_p \in N$ be the corresponding element given by 
Proposition
\ref{Prop:contact order} and recording the 
logarithmic contact order of $p$ with $Y_0$.
Then, we define the \emph{weight} $w_\Gamma(e_p)$ as the divisibility 
of $u_p$ in $N$, that is the biggest positive integer such that $u_p/w_\Gamma(e_p) \in N$. 
\end{itemize}
\end{construction}

Let $f:C^{\dagger} \to Y_0^{\dagger}$ be a log coral. 
Then, we refer to the map $h \colon \Gamma \rightarrow \trcone \RR$
in \eqref{eq_h} as the \emph{tropicalization of the log coral $f:C^{\dagger} \to Y_0^{\dagger}$}, where $\Gamma$ is the coral graph defined as dual graph of $C$ in Construction \ref{dual graph}, and where we view $\Gamma$ endowed with the weight function $w$ defined in Construction \ref{weight function}.
Now, we are ready to prove the main theorem of this section.
\begin{theorem}
\label{Log curves define tropical corals}
The tropicalization $h:\Gamma\to \trcone \RR$
of a log coral $f:C^{\dagger} \to Y_0^{\dagger}$ is a tropical coral.
\end{theorem}
\begin{proof}
In order to show that $h:\Gamma\to \trcone \RR$
satisfies the Definition
\ref{parameterized tropical coral} of a tropical coral, 
it remains to check the balancing condition at each vertex of 
$\Gamma$.
The balancing condition at the interior vertices follows from the balancing result of \cite[Section 1.4]{logGW}. It remains to prove the balancing condition for the negative vertices of 
$\Gamma$, mapped by $h$ to the boundary $\partial \trcone \RR$ of the truncated cone $\trcone \RR$. 
Let $v \in V^{-}(\Gamma)$ a negative vertex, and $C' \subset C$
the corresponding non-proper irreducible component. We have 
$h(v) \in \partial \trcone \RR$.
Define the quotient map
\[ \pi_v:N\rightarrow N~/~(\RR\cdot h(v)\cap N)\,.  \] 
So, $\pi_v$ gives an element $m_v\in \Hom(N,\ZZ)=M$. Define 
\[ x_v:=z^{m_v} \in \CC[M]\,. \]
Observe that $f^*(x_v)=f^*(x)$ if $f$ is transverse at $C'$ and $f^*(x_v)=f^*(x^{\beta})/f^*(y^{\alpha})$ if $f$ is non-transverse at $C'$ where $x$ and $\alpha,\beta$ are as in Remark \ref{varphi}. In any case, we have
\[ \sum_{z\in C'} val_z (f^*x_v) = \sum_{z\in \PP^1 \setminus \{ \infty \}} val_z (f^*x_v) = 0 \]
since from the \emph{parallelness} condition
in Definition \ref{parallel} we know $val_{\infty} (f^*x_v)=0$ and the sum of orders of poles and zeroes of the rational function $f^*x_v$ on $\PP^1$ is zero.

Let $z_1,\ldots,z_k\in C'$ be the special points and let
$\xi_1,\ldots,\xi_k$ be the direction vectors (with weights)
of the corresponding edges emanating from $v$. Note that if $z_i$ is a marked point $p$ then $\xi_i=u_p$ and if $z_i$ is a node $q$ then $z_i=u_q$ where $u_p,u_q$ are defined as in Construction  \ref{weight function} (where $u_q$ is defined emanating from $v$). Then, $\pi_v(\sum_{i=1}^k\xi_i)= \sum_{z\in C'} val_z (f^*x_v) =0$ and hence the tropical balancing condition (\ref{parameterized tropical coral},iii) at negative vertices is also satisfied.
\end{proof}
\section{Proof of the Main Theorem}
\label{log count equals tropical count}
In this section, we prove our main result, showing that the count of tropical corals equals the count of log corals.

\subsection{Set-up} \label{set-up}
Let  $\Delta$ be a degree as in Definition 
\ref{degree} with $\ell+1$ positive and 
$m$ negative entries, 
and $\lambda$ an asymptotic constraint in the  stable range 
$\mathcal{S}_\Delta$ as in Definition 
\ref{stable range}.
Let $\mathcal{H}_{\Delta,\lambda}$ be the set of tropical corals 
$h:\Gamma \rightarrow \overline{C}\RR$ 
matching $(\Delta, \lambda)$. This set is finite by Proposition \ref{finiteness of types} and Remark \ref{no dependence}. 

\begin{construction}
\label{rescaling}
As $\mathcal{H}_{\Delta,\lambda}$
is finite, we can rescale $N$ such that for every $(h:\Gamma \rightarrow \overline{C}\RR) \in \mathcal{H}_{\Delta,\lambda}$ the following conditions hold:
\begin{itemize}
    \item[(i)] For every vertex $v \in V(\Gamma)$, we have $h(v)\in N$.
    \item[(ii)] For every bounded edge $e\in E^b(\Gamma)$, connecting vertices $v_1, v_2 \in V(\Gamma)$, $h(v_1)-h(v_2)$ is divisible in $N$ by the weight $w_\Gamma(E)$. 
    \item[(iii)] For every negative vertex $v \in V^-(\Gamma)$, $h(v)$
    is divisible in $N$ by $w_v$ (where $w_v$ is as in Definition \ref{parameterized tropical coral}(iii)).
    \item[(iv)] For every 
    $1\leq i \leq \ell$, $\lambda_i \in N/(N \cap \QQ \overline{\Delta}^i)$.
\end{itemize}
We fix $b\in \NN\setminus \{0\}$ and we still denote by 
$\trcone \RR$ and $\trcone \P_b$ the truncated cone
and the polyhedral decomposition in the rescaled $\RR^2$.
The toric degeneration defined by the rescaled pair $(\trcone \RR, \trcone \P_b)$ only differs from the toric degeneration defined by the original pair $(\trcone \RR, \trcone \P_b)$ by a base change $t \mapsto t^n$ for some positive integer $n$. We still denote by $Y_0$
the central fiber: the base change does not change the scheme structure of $Y_0$ and only scales the log structure. 
\end{construction}


It follows from the description of the tropicalization in 
\S \ref{tropicalizations} that if 
$h:\Gamma \rightarrow \overline{C}\RR$ is the tropicalization of a log coral 
$f: C^\dagger \rightarrow Y_0^\dagger$, then
$h:\Gamma \rightarrow \overline{C}\RR$
matches $(\Delta,\lambda)$ as in Definition
\ref{trop_coral_match} if and only if $f: C^\dagger \rightarrow Y_0^\dagger$
matches $(\Delta,\lambda)$ as in Definition 
\ref{coral_match}.
In addition, let us fix a log constraint 
$\rho \in (\CC^{\times})^\ell \times {(\CC^{\times})}^{m}$ of order $(\ell,m)$
as in Definition \ref{log constraint}.
Our main result is a
correspondence between the count 
with multiplicities 
$N_{\Delta,\lambda}^{\mathrm{trop}}$ of the log corals 
matching 
$(\Delta,\lambda)$, as in
\eqref{NtropCurve}, and 
the count
$N^{log}_{\Delta, \lambda,\rho}$ of log corals in $Y_0$ matching $(\Delta,\lambda,\rho)$, as in \eqref{log_curves_count}.

\subsection{Extending the unfolded Tate curve}
\label{Extending the Tate curve}

As $\mathcal{H}_{\Delta,\lambda}$ is finite, we can choose
a big enough bounded closed interval $I\subset \RR$ such that for every
$(h:\Gamma \rightarrow \overline{C}\RR) \in \mathcal{H}_{\Delta,\lambda}$, the image $h(\Gamma)$ of the tropical coral is contained in 
$\trcone I \subset \trcone \RR$. Define a polyhedral decomposition $\P_I\subset \P_b$ such that cells $\sigma'\in \P_I$ are given by $\sigma'=\sigma \cap I$ for a cell $\sigma \in \P$. 
Let $\trcone \P_I$ be the polyhedral decomposition of the truncated cone over $I$, as in Definition \ref{truncated cone}, induced by $\P_I$. Let $U \to \AA^1$, be the toric degeneration associated to $(\trcone \P_I, \trcone I)$. The central fiber of this degeneration, $U_0$, is an open subset of $Y_0$ such that the image of any log coral $f:C^{\dagger}\to Y_0^{\dagger}$ whose tropicalization is 
$h:\Gamma \rightarrow \overline{C}\RR$
lies inside $\trcone I \subset \trcone\RR$. 

Corresponding to each vertex in the interior of the interval $I$, we have an irreducible component of $U_0$ given by $\PP^1 \times \AA^1$, and corresponding to the vertices that are endpoints of the interval $I$, there are irreducible components of $U_0$ given by $\AA^1 \times \AA^1$. The open subset $U_0 \subset Y_0$ is a union of these components pairwise glued along affine lines. See Figure \ref{U0} for an illustration.

\begin{figure}
	\resizebox{.9\linewidth}{!}{\input{UU0.pspdftex}}
	\caption{The truncated cone $\trcone I$ together with the polyhedral decomposition $\trcone \P_I$ on the left and its dual, corresponding to the momentum map images of components of $U_0$ on the right.}
	\label{U0}
\end{figure}

We define a compact variety $\overline{U}_0$ obtained from $U_0$ by extending the polyhedral decomposition $\trcone \P_I$ of $\trcone I$ to a polyhedral decomposition $\tilde{\P}$ of $\RR^2$ by applying the following steps: First add  the horizontal line $\partial 
\overline{C}\RR$ to $|\trcone \P_I|$. Then, extend each half line meeting $\partial 
\overline{C}\RR$ to a line which passes through the origin, and insert additional $2$-cells bounded by any of the new $1$-cells. The new polyhedral decomposition $\tilde{\P}$ is called \emph{the extended polyhedral decomposition} of $\trcone \P$, and defines a toric degeneration $\overline{U}\to \Spec \CC[s,t]$, whose fiber over $t=0$ is $\overline{U}_0$. See Figure \ref{ExtendedP} for an illustration of the extended polyhedral decomposition.
\begin{figure}
	\resizebox{.9\linewidth}{!}{\input{Z0.pspdftex}}
	\caption{The extended polyhedral decomposition $\tilde{\P}$ obtained from $\trcone \P_I$ in Figure \ref{U0} and a union of the momentum map images of the components of $\overline{U}_0$.}
	\label{ExtendedP}
\end{figure}
Note that $U_0$ is not dense in $\overline{U}_0$, as $\overline{U}_0$ has an additional irreducible component $Z_0$ corresponding to the vertex at the origin of the polyhedral decomposition $\tilde{\P}$. So, $Z_0$ is a complete toric variety associated to the fan in $\RR^2$ whose $1$-dimensional cones are given by $n$ lines passing through the origin, where $n$ is the number of vertices of $\tilde{\P}$ contained in the interval $I$, i.e. all vertices of $\tilde{\P}$ except the origin. As illustrated in Figure \ref{ExtendedP}, topologically $\overline{U}_0$ is given by the union of $n$ copies of $\PP^1 \times \PP^1$ glued to $Z_0$ along a union of projective lines;
\[\overline{U}_0=\big(\bigcup_n \PP^1 \times \PP^1 \big) \bigcup_{\cup_n \PP^1} Z_0.\]
We endow $\overline{U}$ with the divisorial log structure as usual, by considering the divisor $\overline{U}_0 \subset \overline{U}$. We define a log structure on the central fiber $\overline{U}_0$ over $t=0$, by pulling back the divisorial log structure on $\overline{U}$.

\subsection{Extending Log Corals to Log Maps}
\label{technical part}
We show that log corals on the central fiber of the degeneration of the unfolded Tate curve, $Y_0$, defined as in \S\ref{The unfolded Tate curve} can be extended to log maps in $ \overline{U}_0$, the central fiber of the degeneration defined \S\ref{Extending the Tate curve}. Note that log maps to $\overline{U}_0$ are maps from a complete log smooth curve, which we require to be log smooth over the standard log point throughout this section. For details of the theory of log maps we refer to \cite[\S $1$]{logGW}.
\begin{definition}
\label{log extension}
Let $(h:\Gamma \rightarrow \overline{C}\RR) \in \mathcal{H}_{\Delta,\lambda}$ be a tropical coral in 
$\overline{C}\RR$.
Let $f: C^\dagger \to Y_0^\dagger$ be a log coral with tropicalization $h:\Gamma \rightarrow \overline{C}\RR$, and let 
$\tilde{h}:\tilde{\Gamma}\rightarrow \RR^2$ be the extension of $h:\Gamma \rightarrow \overline{C}\RR$ as a tropical curve in 
$\RR^2$, defined as in \S \ref{Extending a tropical coral to a tropical curve}. A log map $\tilde{f}:\tilde{C}^\dagger \to  \overline{U}_0^\dagger$ is called a \emph{log extension} of $f$ if the following holds:
\begin{itemize}
    \item[(i)] The tropicalization of $\tilde{f}$, defined analogously as for a log coral in \S\ref{tropicalizations}, is $\tilde{h}:\tilde{\Gamma}\rightarrow \RR^2$.
    \item[(iii)] Viewing $C \subset \tilde{C}$ as the non-complete subcurve of $\overline{C}$ whose tropicalization is $\Gamma \subset \tilde{\Gamma}$, we have   
    $\tilde{f}|_C=f$.
\end{itemize}
\end{definition}

\begin{definition}
\label{log degree}
Let $\Delta$ be a degree with $\ell+1$ positive and $m$ negative entries, $\lambda$ an asymptotic constraint for $\lambda$, and 
$\rho$ a log constraint of order $(\ell,m)$.
We say that a log map $\tilde{f}:\tilde{C}^{\dagger}\to\overline{U}_0^{\dagger}$ with $(\ell+1)+m$ marked points $p_0,\dots,p_\ell, \dots, q_1,\dots,q_m$
 \emph{matches} $(\Delta,\lambda,\rho)$
 if:
 \begin{itemize}
     \item[(i)] For each marked point $p_i$, the composition 
\[\mathrm{pr}_2\circ \overline{\tilde{f}}^{\flat}_{p_i}:\overline{\M}_{{\overline{U}_0},\tilde{f}(p_i)}^{gp} \rightarrow \overline{\M}_{\tilde{C},p_i}^{gp}\rightarrow \ZZ,\]
equals $\overline{\Delta}^i$, where $\mathrm{pr}_2$ is the projection map to the second factor defined as in \eqref{Pr for log}.
\item[(ii)] For each marked point $q_j$, the composition 
\[\mathrm{pr}_2\circ \overline{\tilde{f}}^{\flat}_{q_j}:\overline{\M}_{{\overline{U}_0},\tilde{f}(q_j)}^{gp} \rightarrow \overline{\M}_{\tilde{C},q_j}^{gp}\rightarrow \ZZ,\]
equals $\underline{\Delta}^j$, 
 where $\mathrm{pr}_2$ is the projection map to the second factor defined as in \eqref{Pr for log}.
\item[(iii)]For each marked point $p_i$
with $1 \leq i\leq \ell$, 
let $C_i$ be the irreducible component
of $\tilde{C}$ containing $p_i$, of generic point $\eta_i$. Then, the image of 
\[(\bar{\tilde{f}}_{\eta_i}:\overline{\M}_{\overline{U}_0,\tilde{f}(\eta_i)}\rightarrow \overline{\M}_{\tilde{C},\eta_i}\cong \NN) \in N \]
under the quotient map 
\[ N\rightarrow N / (N \cap \RR \overline{\Delta}^i)\] 
is equal to $\lambda_i$.
\item[(iv)]For each marked point $q_j$, 
let $C_j$ be the irreducible component
of $\tilde{C}$ containing $q_j$, of generic point $\eta_j$. Then, the image of 
\[(\bar{\tilde{f}}_{\eta_j}:\overline{\M}_{\overline{U}_0,\tilde{f}(\eta_j)}\rightarrow \overline{\M}_{\tilde{C},\eta_j}\cong \NN) \in N \]
under the quotient map 
\[ N\rightarrow N / (N \cap \RR \underline{\Delta}^j)\] 
is equal to $0$.
\item[(v)]
for $1\leq i\leq \ell$,  $\tilde{f}^\flat_{p_i}(s_{\overline{\Delta}^i,\lambda_i})(p_i)=\overline{\rho_i}$, 
where $s_{\overline{\Delta}^i,\lambda_i}$ is defined as in Discussion \ref{DiscussionOfLogConstraints}.
\item[(vi)] for $1\leq j\leq m$, $\tilde{f}^\flat_{q_j}(s_{\underline{\Delta}_j,v_j^-})(q_j)=\underline{\rho}_j$, where $s_{\underline{\Delta}_j,v_j^-}$ is defined as in Discussion \ref{DiscussionOfLogConstraints}.
 \end{itemize}

\end{definition}
By construction of the tropicalization, 
if a log map $\tilde{f}: \tilde{C}^\dagger \rightarrow \overline{U}_0^\dagger$ matches 
$(\Delta,\lambda,\rho)$ as in Definition 
\ref{log degree}, then its tropicalization 
$\tilde{h}: \tilde{\Gamma} \rightarrow \RR^2$
matches $(\Delta,\lambda)$ as in Definition 
\ref{trop_const}.

\begin{lemma}
\label{Lem: stable  log map to log coral}
Let $U_0$ and $\overline{U}_0$ be defined as in \S\ref{Extending the Tate curve}, and let $\tilde f: \tilde{C}^\dagger \to   \overline{U}_0^\dagger$ be a log map matching $(\Delta,\lambda,\rho)$. Then the projection from the fibre product $\tilde f\times_{ \overline{U}_0} U_0\to U_0$ composed with the open embedding $U_0\hookrightarrow Y_0$ is a log coral matching $(\Delta,\lambda,\rho)$.
\end{lemma}

\begin{proof}
The result follows immediately by comparison of Definition \ref{log count} and Definition 
\ref{log degree}.
\end{proof}

\begin{theorem}
\label{LogExtension}
The forgetful map
\[
\widetilde{\mathfrak L}_{(\Delta,\lambda,\rho)} \to \mathfrak L_{(\Delta,\lambda,\rho)}
\]
from the set of isomorphism classes of log maps matching $(\Delta,\lambda,\rho)$, to the set of isomorphism classes of log corals matching  $(\Delta,\lambda,\rho)$ is bijective.
\end{theorem}
\begin{proof}
Let $(h:\Gamma \rightarrow \overline{C}\RR) \in \mathcal{H}_{\Delta,\lambda}$ be a tropical coral in 
$\overline{C}\RR$.
Let $f: C^\dagger \to Y_0^\dagger$ be a log coral with tropicalization $h:\Gamma \rightarrow \overline{C}\RR$, and let 
$\tilde{h}:\tilde{\Gamma} \rightarrow \RR^2$ be the extension of $h:\Gamma \rightarrow \overline{C}\RR$ as a tropical curve in 
$\RR^2$, defined as in \S \ref{Extending a tropical coral to a tropical curve}. 
We have to show that there exists a unique log map $\tilde{f}:\tilde{C}^\dagger \to  \overline{U}_0^\dagger$ which is a log extension of $f$ as in Definition \ref{log extension}.

We first show that there exists a unique 
log curve $\tilde{C}^{\dagger}$ compatible with Definition \ref{log extension}. 
From the description of 
$\tilde{\Gamma}$ in \S \ref{Extending a tropical coral to a tropical curve}
 and the description of tropicalization in \S\ref{tropicalizations}, we deduce that at the scheme level, 
$\tilde{C}$ is given as follows. 
Let $C'_1,\dots,C'_m$ be the non-proper components of $C$. Let $\overline{C}$
be the proper curve obtained from $C$ by adding a point at infinity $\infty_{C'_j}$ to each $C'_j$. 
For every 
$1 \leq j\leq m$, let $\PP^1_j$ be a copy of $\PP^1$ with two marked points $P_{1j}$, $P_{2j}$.
Then $\tilde{C}$ 
is obtained by gluing together 
$\overline{C}$ and the $m$ copies $\PP^1_j$
of $\PP^1$: 
for every $1 \leq j\leq m$, we glue transversally 
the marked point $P_{1j} \in \PP^1$
with $\infty_{C'_j} \in \overline{C}$. 
The resulting marked curve 
$\tilde{C}$ is unique up to isomorphism.
In fact, analogously as in construction \cite[Construction 4.4]{NS}, we also need to attach further $\PP^1$-components, corresponding to divalent vertices not mapping to the origin, which correspond to the intersection points of the image of $\tilde{h}$ with polyhedral decomposition $\tilde{\P}$ on $\mathbb{R}^2$, described in \S\ref{Extending the Tate curve}. However, the image of these components can be determined as in \cite{NS}, and as it does not effect the current proof, we ignore these further components. 

The log structure $\mathcal{M}_{\tilde{C}}$
is also uniquely determined by Definition 
\ref{log extension}. 
Indeed, as we want the marked curve $\tilde{C}$ to be log smooth over the standard log point $\Spec \CC^\dagger$,
the log structure on $\tilde{C}$ must be of the form given by Theorem \ref{Thm: structure of log curves} and the only choices
are the parameters $\rho_q$ of the log structures at the nodes.
All nodes of $\tilde{C}$ distinct from 
the nodes $\infty_{C'_j}$ are nodes of $C$ and so the log structure there is already fixed by
 Definition 
\ref{log extension}(ii).
On the other hand, by Definition \ref{log extension}(i) and the description of tropicalization in 
\S \ref{tropicalizations}, for every $1\leq j\leq m$, the parameter $\rho_{\infty_{C'_j}}$
of the node $\infty_{C'_j}$
has to be equal to the the divisibility in $N$ of $h(v_j^-)/w_{v_j^-}$
where $v_j^-$ is the negative vertex of $\Gamma$
corresponding to $C'_j$
and $w_{v_j^-}$ is the weight at
$v_j^-$ as in Definition 
\ref{parameterized tropical coral}(iii). 
Note that we have indeed $h(v_j^-)/w_{v_j^-} \in N$
by our choice of rescaling of $N$
in Construction \ref{rescaling}(iii).
Hence, we have a uniquely determined log structure $\mathcal{M}_{\tilde{C}}$
on $\tilde{C}$ and we denote $\tilde{C}^\dagger:=(\tilde{C},\mathcal{M}_{\tilde{C}})$. By construction, the log curve 
$\tilde{C}^\dagger$ is log smooth over the standard log point $\Spec \CC^\dagger$.

Next, we show that there exists a unique stable map extension $\tilde{f}: \tilde{C} \rightarrow \overline{U}_0$ compatible with Definition \ref{log extension}. 
We have to define 
$\tilde{f}$ at each of the added components 
$\PP^1_j$ of $\tilde{C}$ for $1 \leq j \leq m$.
First, note that from the tropical picture, such a component maps into the complete toric variety $Z_0\subset \overline{U}_0$, as in \eqref{U0}. Also, recall from Discussion \ref{DiscussionOfLogConstraints} that such a $\PP^1$ component $\PP^1_j$ determines a monomial $m_{\underline{\Delta}_j,v_j^-}$ and hence a section 
\[s_{\underline{\Delta}_j,v_j^-}\in \Gamma(Y_0,\M_{Y_0}^\gp)\]
with the property that $f^\flat(s_{\underline{\Delta}_j,v_j^-})$ is invertible at the point at infinity, $\infty_{C'_j}$. Moreover, the corresponding log constraint $\rho^j$ is the value 
\[f^\flat(s_{\underline{\Delta}_j,v_j^-})(\infty_{C'_j})\in \CC^\times\,.\] 
As a rational function, $z^{m_{\underline{\Delta}_j,v_j^-}}$ is defined on $Z_0\subset \overline{U}_0$ and has to restrict to the constant function on $\PP^1_j$, with value the log constraint $\rho^j\in\CC^\times$. We deduce that $\tilde{f}|_{\PP^1_j}: \PP^1\to Z_0$ is a cover of the line 
\[\{\rho^j\}\times\PP^1\subset \CC^\times\times\PP^1\] 
fully branched over the two intersection points with the toric boundary of $Z_0$. The covering degree agrees with the weight of the edges adjacent to the vertex of $\tilde\Gamma$ corresponding to $\mathbb{P}^1$. Hence, it is determined completely by the tropical extension $\tilde h$, and in fact equal to $w_{v_j^-}$.
Conversely, $z^{m_{\underline{\Delta}_j,v_j^-}}=\rho^j$ defines a rational curve in $Z_0$, and the $w_{v_j^-}$-fold branched cover defines the extension $\tilde{f}$ on $\mathbb{P}^1_j$. This way we determine uniquely the image of each inserted $\PP^1$ component
$\PP^1_j$, and hence we determine the stable map $\tilde{f}: \tilde{C}\to \overline{U}_0$. 

It remains to show that $\tilde{f}$ can be uniquely promoted to a log map 
$\tilde{C}^\dagger \rightarrow \overline{U}_0^\dagger$.
As $\tilde{C}$ is log smooth over the standard log point, it follows analogously as in the torically transverse case from the proof of \cite[Prop 7.1]{NS} that the extension of 
$\tilde{f}$ away from the nodes is uniquely determined by strictness.
For the extension at the node $\infty_{C_j'}$, it follows from the proof of Proposition~7.1 of \cite{NS} that there are locally 
$w_{v_j^-}$ ways to extend $\tilde{f}$
as a log map in a neighborhood of the node.
These $w_{v_j^-}$ choices are parametrized by a choice of $w_{v_j^-}$-th root of unity.
However, there is also an action of the Galois group $\ZZ/w_{v_j^-}\ZZ$ of the $w_{v_j^-}$-fold cyclic branched cover $\PP^1_j \to \tilde{f}(\PP^1_j)$ which acts simply transitively on the set of such choices. Thus, we obtain globally a unique extension of 
$\tilde{f}$ as a log map 
$\tilde{C}^\dagger \rightarrow \overline{U}_0^\dagger$.

\end{proof}
\subsection{Reduction to the torically transverse case}
\label{Reduction to the torically transverse case}
A correspondence theorem between tropical curves and log maps is proved in \cite{NS}, in torically transverse situations. Toric transversality \cite[Defn~4.1]{NS} means that the image of a log map is disjoint from toric strata of codimension larger than one. The tropicalization of such a map then maps to the $1$-skeleton of the polyhedral decomposition defining the considered degeneration.
The following general result will allow us to treat situations that are not torically transverse.
\begin{lemma}
\label{Interpolating family}
Let $\P$ be an integral polyhedral decomposition of $N_\RR$ and $\P'$ an integral refinement. Denote by $\pi:X\to\AA^1$ and $\pi':X'\to \AA^1$ be the associated toric degenerations of toric varieties. Assume that $\P$ and $\P'$ are regular, that is, they support strictly convex piecewise affine functions. Then there is a two-parameter degeneration
\[
\tilde\pi: \tilde X\lra \AA^2
\]
with restrictions to $\AA^1\times\GG_m$ and to $\GG_m\times\AA^1$ equal to $\pi\times\id_{\GG_m}$ and to $\id_{\GG_m}\times\pi'$, respectively.
\end{lemma}

\begin{proof}
Let $\varphi$, $\varphi'$ be strictly convex, piecewise affine functions with corner loci (or non-differentiability loci) $\P$, $\P'$, respectively (by this we mean that the non-extendable domains of linearity of $\varphi$, $\varphi'$ coincide with the maximal cells of $\P$, $\P'$, respectively). We can assume $\varphi,\varphi'$ are defined over the rational numbers and hence, after rescaling, that they are defined over the integers. Denote by $\Phi: N_\RR\oplus\RR_{>0}\to \RR$ the homogenization of $\varphi$:
\[
\Phi: N_\RR\lra\RR,\quad \Phi(x,\lambda) = \lambda\cdot \varphi\big(\frac{x}{\lambda}\big).
\]
Note that $\Phi$ is the restriction of a linear function on the cone $C(\sigma)$ for any $\sigma\in\P$. Now the fan $\Sigma=C(\P)$ can be defined by the corner locus of $\Phi$, that is, the maximal elements of $\Sigma$ are the domains of linearity of $\varphi$. Similarly, we define the homogenization $\Phi'$ of $\varphi'$.
\[
\Phi': N_\RR\lra\RR,\quad \Phi'(x,\mu) = \mu\cdot \varphi'\big(\frac{x}{\mu}\big).
\]
To define a two-parameter degeneration we use the fan $\tilde\Sigma$ in $N_\RR\oplus\RR\oplus\RR$ defined by the corner locus of the following piecewise linear function:
\[
\Psi(x,\lambda,\mu):= \Phi(x,\lambda)+\Phi'(x,\mu).
\]
Now if \[(x,\lambda,\mu)\in|\tilde\Sigma|\] then $\Psi$ is linear on \[(x,\lambda,\mu)+\{0\}\times\RR_{\ge0}^2.\] Thus the projection \[N_\RR\oplus \RR^2\to \RR^2\] induces a map of fans from $\tilde\Sigma$ to the fan of $\AA^2$. Define $\tilde X$ as the toric variety defined by $\tilde\Sigma$ and \[\tilde\pi: \tilde X\to\AA^2\] as the toric morphism defined by the map of fans just described. The restriction of $\tilde\pi$ to $\AA^1\times\GG_m$ is described by the intersection of $\tilde\Sigma$ with $N_\RR\times\RR\times \{0\}$. This intersection is the corner locus of \[\Psi|_{N_\RR\times\RR\times\{0\}}.\] But \[\Psi(x,\lambda,0)=\Phi(x,\lambda),\]so this corner locus defines $C(\P)$. Thus 
\[\tilde\pi|_{\AA^1\times\GG_m} = \pi \times \id_{\GG_m}.\] Analogously we conclude \[\tilde\pi|_{\GG_m\times\AA^1}= \id_{\GG_m}\times \pi'\]
\end{proof}
With the construction of Lemma~\ref{Interpolating family} we are now in position where we can refer to \cite{NS} to show main result in the following theorem. 
\begin{theorem}
\label{main theorem}
Let  $\Delta$ be a degree as in Definition 
\ref{degree} with $\ell+1$ positive and 
$m$ negative entries, $\lambda$ an asymptotic constraint in the  stable range 
$\mathcal{S}_\Delta$ as in Definition 
\ref{stable range}, 
and $\rho$ a
log constraint  of order $(\ell,m)$
as in Definition \ref{log constraint}.
Let $N^{\mathrm{trop}}_{\Delta,\lambda}$ be the count with multiplicities of tropical corals in $\trcone \RR$ 
matching $(\Delta, \lambda)$, defined as in \eqref{Ntrop}, and let $N^{log}_{\Delta,\lambda,\rho}$ be the count of log corals in $Y_0$ matching $(\Delta,\lambda,\rho)$, defined as in \eqref{log count}. Then,
\[ N^{\mathrm{trop}}_{\Delta,\lambda} = N^{log}_{\Delta,\lambda,\rho}\,. \]
Moreover, this count is a log Gromov--Witten invariant of the (unfolded) Tate curve.
\end{theorem}
\begin{proof}
Let $f:{C}^{\dagger}\to {Y_0}^{\dagger}$ be a log coral
with tropicalization 
$h \colon \Gamma \rightarrow \overline{C}\RR$. 
By Theorem \ref{LogExtension}, it extends uniquely to a log map $\tilde{f}:
\tilde{C}^\dagger \to \overline{U}_0^\dagger$ with tropicalization 
$\widetilde{h}:\widetilde{\Gamma}\to \RR^2$, where $\RR^2$ is endowed with the extended polyhedral decomposition $\tilde{\P}$ defined at the beginning of this subsection. This polyhedral decomposition is regular by inspection. 

Our aim is to achieve the toric transversality as in \cite{NS}, to use the correspondence theorem between tropical curves and log maps shown in \cite[\S 8]{NS}. For this we need to ensure that the image of $\widetilde{h}$ remains inside the $1$-skeleton of the polyhedral decomposition $\tilde\P$ refine $\tilde\P$ accordingly. To obtain an appropriate refinement we simply add $\widetilde{\Gamma}$ to the $1$-skeleton of $\tilde{\P}$. We do this refinement to $\tilde{\P}$ simultaneously for all tropical curves $\widetilde{h}_i:\widetilde{\Gamma}_i\to \RR^2$ of degree $\Delta$ matching $\lambda$ and of type $(\Gamma,u)$. Note that by \S \ref{The count of tropical corals}, there are finitely many such tropical curves. This way obtain a new polyhedral decomposition $\tilde{\P'}$ of $\RR^2$. By further refinement and rescaling we may assume that $\tilde\P'$ is integral and regular. Denote by $W_0$ the fiber over $t=0$ of the toric degeneration associated to $\tilde{\P'}$. So, $W_0$ carries the natural log structure $\M_{W_0}$ obtained by the pull-back of the divisorial log structure on the total space given by the divisor $W_0\subset W$. Then, by Lemma~\ref{Interpolating family} we obtain a two-parameter degeneration
\[
\tilde\pi: \tilde X\lra \AA^2
\]
isomorphic to $\pi:\ol U\to \AA^1$ over $\{1\}\times \AA^1$ and to the degeneration $\pi': W\to\AA^1$ defined by $\tilde\P'$ over $\AA^1\times\{1\}$. For the latter degeneration, all log curves in the central fibre $(W_0,\M_{W_0})$ are torically transverse. Note that the log constraints $\rho$ translate into point constraints along with multiplicities along toric divisors on the general fibre of $\pi'$, which is the situation treated in the refinement \cite{GPS} of \cite{NS}. In this refined case, the tropical count differs from the count in \cite{NS}, by dividing the weights on the unbounded edges of $\tilde{\Gamma}$ as shown in the Appendix of \cite{GPS}. Hence, from Theorems $3.4$ and $4.4$ in \cite{GPS}, the log maps on the central fiber of $\pi'$ matching analogous incidence conditions is a finite set with cardinality equal to $\tilde{N}^{\mathrm{trop}}_{\Delta,\lambda}$ we defined in  \eqref{NtropCurve}. 
By deformation invariance of log Gromov-Witten invariants we obtain the same virtual count for log maps on the central fibre $\overline{U}_0^\dagger$ of $\pi$, as a log space over the standard log point. By Theorem~\ref{LogExtension}, this count of log maps coincides with the count of log corals, and so 
\[ \tilde{N}^{\mathrm{trop}}_{\Delta,\lambda}=N^{log}_{\Delta,\lambda,\rho}\,.\]
Finally, by Theorem \ref{Independence Of Constraint} relating tropical corals and tropical curves, we have $\tilde{N}^{\mathrm{trop}}_{\Delta,\lambda}
=N^{\mathrm{trop}}_{\Delta,\lambda}$, and so 
\[ N^{\mathrm{trop}}_{\Delta,\lambda} = N^{log}_{\Delta,\lambda,\rho}\,. \]

Since we are in the situation where the total space $Y$ is toric, it follows analogously as in \cite[Prop 7.3]{NS} that this count is a log Gromov-Witten invariant of the general fiber of $Y$, which is the unfolded Tate curve. 
\end{proof}


\subsection{Lifting Log Corals on the Tate curve to its unfolding}
\label{Taking the quotient}
Recall that $Y_0$ is the central fiber of the degeneration of the unfolded Tate curve, and by taking the $\ZZ$-quotient 
\[\pi:Y_0\xrightarrow{/ \ZZ}
T_0 \times \AA^1\]
we obtain the central fiber of the degeneration of the Tate curve, as explained in \S \ref{The Tate curve and its unfolding}. Given a log coral $f:{C}^{\dagger}\to {Y_0}^{\dagger}$, one naturally obtains a log coral 
\[\overline{f}=\pi \circ f: \overline{C}^\dagger
\longrightarrow (T_0 \times \AA^1)^\dagger\,.\] 
In the following Theorem, we show that given a log coral $\overline{f}:\overline{C}^\dagger
\to (T_0 \times \AA^1)^\dagger$ 
there is a lift of it to a log coral  $f:{C}^{\dagger}\to {Y_0}^{\dagger}$.
\begin{theorem}
\label{lift}
Let $\overline{f}:\overline{C}^\dagger\to (T_0 \times \AA^1)^\dagger$ be a log coral fitting into the following Cartesian diagram
\begin{equation} 
\xymatrix@C=30pt
{
\tilde{C}\times_{T_0 \times \AA^1} Y_0 \ar[r]_{/ \ZZ} \ar[d]^{f}
&\overline{C} \ar[d]^{\overline{f}} \\\
Y_0 \ar[r]_{/ \ZZ} & T_0 \times \AA^1
}
\end{equation}
The fiber product  $\overline{C}\times_{T_0 \times \AA^1} Y_0$ is isomorphic to a disjoint union of copies of $\overline{C}$, that is
\[ \overline{C}\times_{T_0 \times \AA^1} Y_0= \coprod_{\ZZ} \overline{C} \]
and
$\overline{f}=\coprod_{n\in \ZZ} \Phi_n \circ f$ where $\Phi_n:Y_0\to Y_0$ is the $\ZZ$-action on $Y_0$ induced by the $\ZZ$-action on the Mumford fan  defined as in \S \ref{Tate fan}.
\end{theorem}
\begin{proof}
It is enough to show that each connected component $C'\subset \overline{C}\times_{T_0 \times \AA^1} Y_0$ is isomorphic to $\overline{C}$. This will follow from the fact that any \'etale proper map from a connected curve to a nodal rational curve is an isomorphism. The map $\pi:\overline{C}\times_{T_0 \times \AA^1} Y_0\to \overline{C}$ is \'etale since $Y_0 \to (T_0 \times \AA^1)$ is \'etale, hence it is a local isomorphism in the \'etale topology. 

We claim that $C'$ has finitely many irreducible components. 
Let $\Gamma_{C'}$ and $\Gamma_{\overline{C} }$ be the dual graphs of $C'$ and $\overline{C} $ respectively, defined as in construction \ref{dual graph}. The map $\pi|_{C'}:C'\to \overline{C} $ induces a map
\[ \Gamma_{\pi}:  \Gamma_{C'} \to \Gamma_{\overline{C} } \]
which is a local isomorphism. So, for every vertex
$v\in \Gamma_{C'}$ , the restriction $\Gamma_{\pi}|_{Star(v)}$ is an isomorphism onto $Star(\Gamma_{\pi}(v))$. Here, for any vertex $v$, $Star(v)$ denotes the subgraph consisting of the vertex v, edges containing $v$ and their vertices. Since $\Gamma_{C'}$ is connected and $\Gamma_{\overline{C}}$ is a tree, $\Gamma_{\pi}$ is a covering map. As any covering map factors through the universal covering and $\pi_1(\Gamma_{\overline{C}})=0$, it follows that $\Gamma_{\pi}$ is an isomorphism. Hence, the number of vertices of $\Gamma_{C'}$ is finite and the claim follows. This shows that the map $\pi|_{C'}:C'\to \overline{C} $ is proper.

By the Hurwitz formula we have $\pi^{*}\omega_{\overline{C} }=\omega_{C'}$ where $\omega_{\overline{C} }$ and $\omega_{C'}$ denote the canonical bundles over 
$\overline{C}$ and $C'$ respectively. Applying the Riemann Roch Theorem for nodal curves, we obtain
\[ 2g-2=\deg \omega_{C'}=d\cdot \deg (\omega_{\overline{C} })=-2d   \] 
where $g=g(C')=0$ as each component of $C'$ is a copy of $\PP^1$ and the dual graph is a tree. Hence, $d=1$. Therefore, $\pi|_{C'}:C'\to \overline{C} $ is an isomorphism.
\end{proof}
We will now discuss how to translate the equivalence of the tropical and log counts stated in Theorem \ref{main theorem} on the (degenerate) unfolded Tate curve to the $\ZZ$-quotient $T_0 \times \AA^1$. 
Let us first set up the counting problem on the tropical side. Denote the corresponding affine manifold by $\ol B= \trcone S^1$, which topologically is $S^1\times\RR_{\ge0}$. 

Each direction of positive or negative ends (the directions associated to positive edges, or to the edges adjacent to negative vertices of a coral graph respectively) that enter the definition of the degree of a tropical coral is now only defined up to the $\ZZ$-action. Thus specifying a degree $\Delta$ on $\ol B$ amounts to choosing an equivalence class under the $\ZZ$-action of the corresponding data on $B= \trcone \RR$. Given a degree $\Delta$ and asymptotic constraint $\rho$ on $\ol B$, there are either zero or infinitely many tropical corals with a fixed number of ends and with each direction and asymptotic constraint restricted only to an equivalence class under the $\ZZ$-action. Indeed, given any tropical coral fulfilling the given constraints on its ends, the composition with the action of any integer on $B$ will produce another one. These tropical corals related by the $\ZZ$-action induce equal tropical objects 
 on $\ol B$. Thus we should rather restrict to $\ZZ$-equivalence classes of tropical corals. A simple way to break the $\ZZ$-symmetry is to choose one representative of the $\ZZ$-equivalence class of directions of one of the ends.

Another source of infinity in the count is more fundamental and is part of the nature of the problem. It is related to the fact that the counting problem in symplectic cohomology produces an infinite sum with single terms weighted by the symplectic area of the pseudo-holomorphic curve with boundary. The logarithmic analogue runs as follows. For a log coral $f:C^{\dagger}\to Y_0^{\dagger}$ look at the underlying scheme theoretic morphism and take the composition 
\[ C\lra Y_0 \lra   T_0 \times\AA^1\] 
where $T_0$ is the central fiber of the Tate curve as in \S\ref{The Tate curve and its unfolding}.
This morphism extends uniquely to a morphism 
\[ \overline{C} \lra T_0\times\PP^1 \]
from a complete curve $\overline{C}$ to $T_0\times\PP^1$. The algebro-geometric analogue of the symplectic area is the degree of the composition 
\[\overline{C}\lra T_0\times\PP^1 \lra T_0\,.\] 
Since ``degree'' is already taken for something else, let us call the degree of $\tilde{\ul C}\to T_0$ the \emph{log-area} of the log coral. We denote by
\[N_{\Delta,\lambda,\rho}(A)\]
the cardinality of the set of log corals of fixed log-area $A$, given degree $\Delta$, matching the asymptotic constraint $\lambda$ and the log constraint $\rho$. We will see in \eqref{trop_log}
below that $N_{\Delta,\lambda,\rho}(A)$ is finite.

The tropical analogue of the area is given by a tropical intersection number as follows \cite{AR}. For each vertex $(a\cdot b,1)$ of $\P_b$ we have the line \[L_a=\RR\cdot(a,1)\subset N_\RR\] 
through the origin. Each tropical coral has a well-defined intersection number with $L_a$. Define the \emph{tropical area} of a tropical coral as the sum of the intersection numbers of its extension as a tropical curve in $N_\RR$ with $L_a$, for all $a\in \ZZ$. The tropical intersection number bounds the number of crossings of the tropical coral with unbounded $1$-cells of the polyhedral decomposition $\trcone \P_b$ of $\trcone \RR$, defined in \S \ref{The unfolded Tate curve}. Hence, with the $\ZZ$-action modded out as before, we also obtain a finite tropical count \[N^{\mathrm{trop}}_{\Delta,\lambda}(A)\] 
of tropical corals in $\trcone S^1$ of fixed tropical area $A$ of fixed degree $\Delta$ and matching a general asymptotic constraint $\lambda$. 
Using the fact that the tropical area of the tropicalization of a log-coral is the log-area of the log coral,
Theorem \ref{lift} and Theorem \ref{main theorem}, readily gives 
\begin{equation}
\label{trop_log}
N^{\mathrm{trop}}_{\Delta,\lambda}(A)= N^{log}_{\Delta,\lambda,\rho}(A)\end{equation} 
for any $A\in\NN$, $\Delta$, $\lambda$ and $\rho$. This shows in particular that $N^{log}_{\Delta,\lambda,\rho}(A)$ is finite.
Introducing generating series summing over $A\in\NN$:
\[
N^{log}_{\Delta,\lambda,\rho}(T_0 \times \AA^1) =\sum_{A\in\NN} N_{\Delta,\lambda,\rho}(A) q^A \in \CC\lfor q\rfor\,,
\]
\[
N^{\mathrm{trop}}_{\Delta,\lambda}(\trcone S^1) =\sum_{A\in\NN} N^{\mathrm{trop}}_{\Delta,\lambda}(A) q^A \in \CC\lfor q\rfor\,,
\]
and using the analogous unobstructedness results in \cite[\S 7]{NS}, we obtain:
\begin{theorem}
\label{Main Theorem for the Tate curve}
$N^{log}_{\Delta,\lambda,\rho}(T_0 \times \AA^1)=N^{\mathrm{trop}}_{\Delta,\lambda}(\trcone S^1)$
and this count is a log invariant of the Tate curve.
\end{theorem}
 



\section{Punctured Log Curves}
\label{The punctured invariants of the central fiber of the Tate curve}
The theory of log Gromov--Witten invariants \cite{logGW,logGWbyAC}, provides a vast generalisation of the theory of relative Gromov--Witten invariants \cite{LR, JL}. While in relative Gromov--Witten theory constraints are imposed relative to a smooth divisor, in log Gromov--Witten theory this is done relative to more general divisors, including normal--crossing divisors. Abramovich--Chen--Gross--Siebert, introduced \emph{punctured log Gromov--Witten theory}, in which one could also consider allow negative tangency orders with the divisors. 

Recall that we introduced log corals on $Y_0$, the central fiber of the degeneration of the (unfolded) Tate curve. We show in this section that log corals correspond to punctured log maps on $\widetilde{T}_0$, the central fiber of the (unfolded) Tate curve, defined as in \S \ref{The Tate curve and its unfolding}. As discussed in \S \ref{Taking the quotient}, since log corals on the Tate curve lift to the unfolded Tate curve, for convenience we again work on the unfolded case. 
\begin{remark}
\label{Rem: punctured curves}
From Proposition \ref{Prop:contact order} it follows that given a log coral 
$f:C^\dagger \to Y_0^\dagger$, the tangency order (or contact order) at a marked point $p\in C$ is given as the image of the composition of the following maps
\[ \overline{\M}_{Y_0,f(p)}\to \overline{\M}_{C,p}= \NN\oplus \NN \xrightarrow{\pr_2} \NN \,,\]
where the equality $\overline{\M}_{C,p}= \NN\oplus \NN$ follows from 
Theorem
\ref{Thm: structure of log curves}(ii), and 
$\pr_2$ is the projection on the second factor.
To achieve negative tangency orders, the idea in punctured log Gromov--Witten theory is to modify the log structure on the domain, so that we would have $\overline{\M}_{C,p} \subseteq \NN\oplus \ZZ$, and hence projecting to the second factor would give elements in $\ZZ$ capturing the tangency order, which could be negative -- for details see \cite{ACGS}. Such marked points $p$, where one allows negative contact orders are referred to as punctured points. For the case of the Tate curve recall that the domain of a log coral $f:C^\dagger \to Y_0^\dagger$ onto the central fiber of a degeneration of the Tate curve admits non-complete components. We will obtain punctured log maps from log corals by trading the non-complete components by punctured points. 
\end{remark}
Let $f:{C}^{\dagger}\to {Y_0}^{\dagger}$ be a log coral and let $\AA^1_k\subset C$ for $k=1,\ldots,m$ be the non-complete components of $C$ such that 
\begin{equation}
\label{eq:noncomplete}
C=\coprod_{k=1}^m \AA^1_k ~ \amalg_{q_k} ~ \tilde{C}
\end{equation}
where $q_k \subset \AA^1_k$ is a nodal intersection point of the non-complete component with another irreducible component. We omit all the non-complete components from $C$, and define a punctured log map
\[ \tilde{f}: (\tilde{C},\M^{\circ}_{\tilde{C}})\to (\widetilde{T}_0,\M_{\widetilde{T}_0})  \]
where the log structure $\alpha^{\circ}:\M^{\circ}_{\tilde{C}}\to \mathcal{O}_{\tilde{C}}$ is 
obtained by modifying the log structure on $C$ as pointed in Remark \ref{Rem: punctured curves}. For a detailed of this modification we refer to \cite[\S $2.1.4$]{ACGS}. 

Recall from \S \ref{The Tate curve and its unfolding}, that $Y_0$ is obtained as the product of the affine line with $\widetilde{T}_0$. We will first show that, given a log coral $f: C^{\dagger} \rightarrow {Y_0}^{\dagger}$, by projecting the image of non-complete component $\AA^1\subset Y_0$ onto $\widetilde{T}_0$, we obtain a single point in $\widetilde{T}_0$ -- this will correspond to a punctured point.
\begin{proposition}
\label{constant along noncomplete}
Let $f: C^{\dagger} \rightarrow {Y_0}^{\dagger}$ be a general log coral and let $C'\subset C$ be a non-complete component. Then, the map 
\begin{equation}
\label{Eq: position of puncture}
\pr_2\circ \tilde{f}_{|C'}:C'\cong \mathbb{A}^1  \rightarrow \mathbb{P}^1 
\end{equation} 
defined as in Equation \ref{factoring}, is constant.
\end{proposition}
\begin{proof}
Let $\tilde{f}(z):\PP^1\to\PP^1$ be the extension of $\pr_2\circ \tilde{f}_{|C'}:C'\cong \mathbb{A}^1  \rightarrow \mathbb{P}^1 $ as in \ref{factoring}. If $f$ is non-transverse at $C'$, the result follows trivially. Assume $f$ is transverse at $C'$, then 
\[\big(\tilde{f}_{|C'}(C')~\cap ~(Y_0)_{s=0}\big)\subset \PP^1\setminus \{ 0,\infty \}\]
where by $\{ 0,\infty \}$ we denote the lower dimensional toric strata in $\PP^1$.
Furthermore, by the parallel condition in Definition \ref{parallel}(i), the point at infinity intersects the divisor at infinity at a point 
\[\tilde{f}_{C'}(\infty_{C'})=q_{\infty} \in \CC^{\times}\subset \PP^1\] 
Hence we have $\tilde{f}(C' \cup \{\infty_{C'} \} ) \in \PP^1\setminus \{ 0\}$. The result follows, since a morphism from a complete variety to an affine variety is constant. 
\end{proof}
Proposition \ref{constant along noncomplete} shows that the non-complete components $C'\subset C$ do not carry any additional information except at the position of the nodal point $q_{\infty}$. By trading any non-complete component $C'$ by the point $\pr_2\circ \tilde{f}_{|C'}(C')$ defined as in \eqref{Eq: position of puncture}, we obtain a restriction of a log coral in $Y_0$ to $\widetilde{T}_0={(Y_0)}_{s=0}$. On the tropical side this amounts to restricting the corresponding tropical coral on the truncated cone $\trcone \RR$, obtained by tropicalization of a log coral, to the interior of the truncated cone. 

We define incidence conditions for punctured maps $\tilde{f}:(\tilde{C},\M^{\circ}_{\tilde{C}})\to (\widetilde{T}_0,\M_{\widetilde{T}_0})$ analogously to the case of extended log corals, as explained in \S \ref{technical part}. 


\begin{definition}
Let $\Delta$ be a degree with $\ell+1$ positive and $m$ negative entries, $\lambda$ an asymptotic constraint for $\Delta$, 
and $\rho$ a log constraint of order $(\ell,m)$.
We say that a punctured log map $\tilde{f}:(\tilde{C},\M^{\circ}_{\tilde{C}})\to \widetilde{T}_0^\dagger$ with $\ell+1$ marked points $p_0,\dots,p_\ell$ and $m$ punctured points 
$q_1,\dots,q_m$
\emph{matches} $(\Delta,\lambda,\rho)$
if:
\begin{itemize}
     \item[(i)] For each marked point $p_i$, the composition 
\[\mathrm{pr}_2\circ \overline{\tilde{f}}^{\flat}_{p_i}:\overline{\M}_{{\widetilde{T}_0},\tilde{f}(p_i)}^{gp} \rightarrow \overline{\M}_{\tilde{C},p_i}^{\circ,gp}\rightarrow \ZZ\]
equals $\overline{\Delta}^i$,
where $\mathrm{pr}_2$ is the projection map to the second factor.
\item[(ii)] For each punctured point $q_j$, the composition 
\[\mathrm{pr}_2\circ \overline{\tilde{f}}^{\flat}_{q_j}:\overline{\M}_{{\widetilde{T}_0},\tilde{f}(q_j)}^{gp} \rightarrow \overline{\M}_{\tilde{C},q_j}^{\circ,gp}\rightarrow \ZZ\]
equals $\underline{\Delta}^j$,
where $\mathrm{pr}_2$ is the projection map to the second factor.
\item[(iii)]For each marked point $p_i$
with $1 \leq i\leq \ell$, 
let $C_i$ be the irreducible component
of $\tilde{C}$ containing $p_i$, of generic point $\eta_i$. Then, the image of 
\[(\bar{\tilde{f}}_{\eta_i}:\overline{\M}_{\widetilde{T}_0,\tilde{f}(\eta_i)}\rightarrow \overline{\M}^{\circ}_{\tilde{C},\eta_i}\cong \NN) \in N \]
under the quotient map 
\[ N\rightarrow N / (N \cap \RR \overline{\Delta}^i)\] 
is equal to $\lambda_i$.
\item[(iv)]For each punctured point $q_j$, 
let $C_j$ be the irreducible component
of $\tilde{C}$ containing $q_j$, of generic point $\eta_j$. Then, the image of 
\[(\bar{\tilde{f}}_{\eta_j}:\overline{\M}_{\widetilde{T}_0,\tilde{f}(\eta_j)}\rightarrow \overline{\M}^{\circ}_{\tilde{C},\eta_j}\cong \NN) \in N \]
under the quotient map 
\[ N\rightarrow N / (N \cap \RR \underline{\Delta}^j)\] 
is equal to $0$.
\item[(v)]
for $1\leq i\leq \ell$,  $\tilde{f}^\flat_{p_i}(s_{u,\lambda})(p_i)=\overline{\rho_i}$, 
where $s_{u,\lambda}$ is defined as in Discussion \ref{DiscussionOfLogConstraints}.
\item[(vi)] for $1\leq j\leq m$, $\tilde{f}^\flat_{q_j}(s_{\underline{\Delta}_j,v_j^-})(q_j)=\underline{\rho}_j$, where $s_{\underline{\Delta}_j,v_j^-}$ is defined as in Discussion \ref{DiscussionOfLogConstraints}.
 \end{itemize}
\end{definition}

Denote the set of punctured log maps $\tilde{f}:(\tilde{C},\M^{\circ}_{\tilde{C}})\to \widetilde{T}_0^\dagger$ matching $(\Delta,\lambda,\rho)$
by ${\mathfrak{L}^{\circ}}_{\Delta,\lambda,\rho}$. It follows from Proposition \ref{constant along noncomplete} that given a punctured map $\tilde{f}:(\tilde{C},\M^{\circ}_{\tilde{C}})\to \widetilde{T}_0^\dagger$, we can extend it to a log coral $f:{C}^{\dagger}\to {Y_0}^{\dagger}$, defined by
\begin{equation}
\label{Eq: f}
  f =
\left\{
	\begin{array}{ll}
		\tilde{f}  & \mbox{on } \tilde{C}\subset C \\
		\tilde{f}(q_k) & \mbox{on } \AA^1_k \subset C
	\end{array}
\right. 
\end{equation}
for each $k\in\{ 1,\ldots,m\}$, where $q_k$ and $\AA^1_k$ are as in \eqref{eq:noncomplete}. The map $f$ naturally lifts to a log morphism, similarly as in the proof of Theorem $3.1.3$ of \cite{ACGS}. Hence, we obtain the following result.
\begin{theorem}
\label{PuncturedCorrespondence}
The forgetful map
\begin{eqnarray}
\mathfrak{L}_{\Delta,\lambda,\rho} &\lra
\mathfrak{L}^{\circ}_{\Delta,\lambda,\rho}
\nonumber\\
f &\mapsto \tilde{f}
\nonumber
\end{eqnarray}
is a bijection. Thus, the count of log corals $N^{log}_{\Delta,\lambda,\rho}$ in Theorem \ref{main theorem} is a punctured log Gromov--Witten invariant. 
\end{theorem}
The deformation theory with punctured point constraints has been worked out in a preliminary version of \cite{ACGS}, and in this particular case for the Tate curve it follows that the punctured stable maps in $\mathfrak{L}^{\circ}_{\Delta,\lambda,\rho}$ are unobstructed -- for details on this, we refer to \cite[\S $4$]{ACGS}. Therefore, from the bijection in Lemma \ref{PuncturedCorrespondence} we conclude  the equality of the count of log corals on the central fiber of the degeneration of the (unfolded) Tate curve with a punctured Gromov-Witten invariant in the sense of \cite{ACGS} on the central fibre of the (unfolded) Tate curve.
\bibliographystyle{plain}
\bibliography{bibliography}

\end{document}

%% file: HulyasCoral.pspdftex
\begin{picture}(0,0)%
\includegraphics{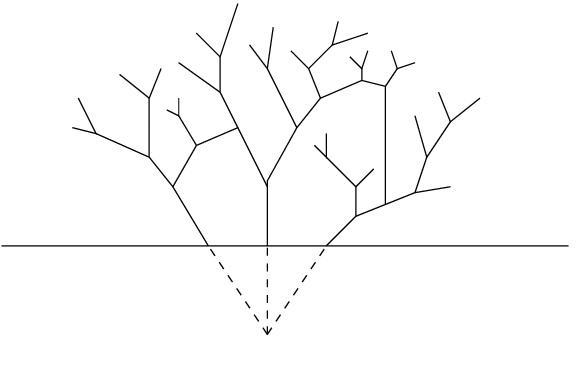}%
\end{picture}%
\setlength{\unitlength}{4144sp}%
\begingroup\makeatletter\ifx\SetFigFont\undefined%
\gdef\SetFigFont#1#2#3#4#5{%
  \reset@font\fontsize{#1}{#2pt}%
  \fontfamily{#3}\fontseries{#4}\fontshape{#5}%
  \selectfont}%
\fi\endgroup%
\begin{picture}(4392,2781)(1114,1715)
\put(5491,2639){\makebox(0,0)[lb]{\smash{{\SetFigFont{12}{14.4}{\familydefault}{\mddefault}{\updefault}{\color[rgb]{0,0,0}$\partial \trcone \RR$}%
}}}}
\put(2926,1784){\makebox(0,0)[lb]{\smash{{\SetFigFont{12}{14.4}{\familydefault}{\mddefault}{\updefault}{\color[rgb]{0,0,0}$(0,0)$}%
}}}}
\end{picture}%

%% file: MumfordFan.pspdftex
\begin{picture}(0,0)%
\includegraphics{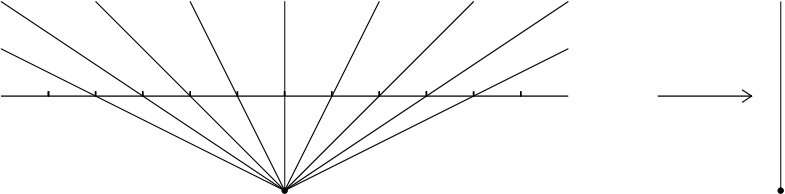}%
\end{picture}%
\setlength{\unitlength}{3315sp}%
\begingroup\makeatletter\ifx\SetFigFont\undefined%
\gdef\SetFigFont#1#2#3#4#5{%
  \reset@font\fontsize{#1}{#2pt}%
  \fontfamily{#3}\fontseries{#4}\fontshape{#5}%
  \selectfont}%
\fi\endgroup%
\begin{picture}(7470,1845)(1114,-2569)
\put(3826,-1906){\makebox(0,0)[b]{\smash{{\SetFigFont{10}{12.0}{\familydefault}{\mddefault}{\updefault}{\color[rgb]{0,0,0}$0$}%
}}}}
\put(4276,-1906){\makebox(0,0)[b]{\smash{{\SetFigFont{10}{12.0}{\familydefault}{\mddefault}{\updefault}{\color[rgb]{0,0,0}$b$}%
}}}}
\put(4726,-1906){\makebox(0,0)[b]{\smash{{\SetFigFont{10}{12.0}{\familydefault}{\mddefault}{\updefault}{\color[rgb]{0,0,0}$2b$}%
}}}}
\put(5176,-1906){\makebox(0,0)[b]{\smash{{\SetFigFont{10}{12.0}{\familydefault}{\mddefault}{\updefault}{\color[rgb]{0,0,0}$3b$}%
}}}}
\put(5626,-1906){\makebox(0,0)[b]{\smash{{\SetFigFont{10}{12.0}{\familydefault}{\mddefault}{\updefault}{\color[rgb]{0,0,0}$4b$}%
}}}}
\put(2881,-1906){\makebox(0,0)[b]{\smash{{\SetFigFont{10}{12.0}{\familydefault}{\mddefault}{\updefault}{\color[rgb]{0,0,0}$-2b$}%
}}}}
\put(2431,-1906){\makebox(0,0)[b]{\smash{{\SetFigFont{10}{12.0}{\familydefault}{\mddefault}{\updefault}{\color[rgb]{0,0,0}$-3b$}%
}}}}
\put(3331,-1906){\makebox(0,0)[b]{\smash{{\SetFigFont{10}{12.0}{\familydefault}{\mddefault}{\updefault}{\color[rgb]{0,0,0}$-b$}%
}}}}
\put(1981,-1906){\makebox(0,0)[b]{\smash{{\SetFigFont{10}{12.0}{\familydefault}{\mddefault}{\updefault}{\color[rgb]{0,0,0}$-4b$}%
}}}}
\put(6571,-1636){\makebox(0,0)[lb]{\smash{{\SetFigFont{10}{12.0}{\rmdefault}{\mddefault}{\updefault}{\color[rgb]{0,0,0}$B=\RR$}%
}}}}
\end{picture}%

%% file: CCXi.pspdftex
\begin{picture}(0,0)%
\includegraphics{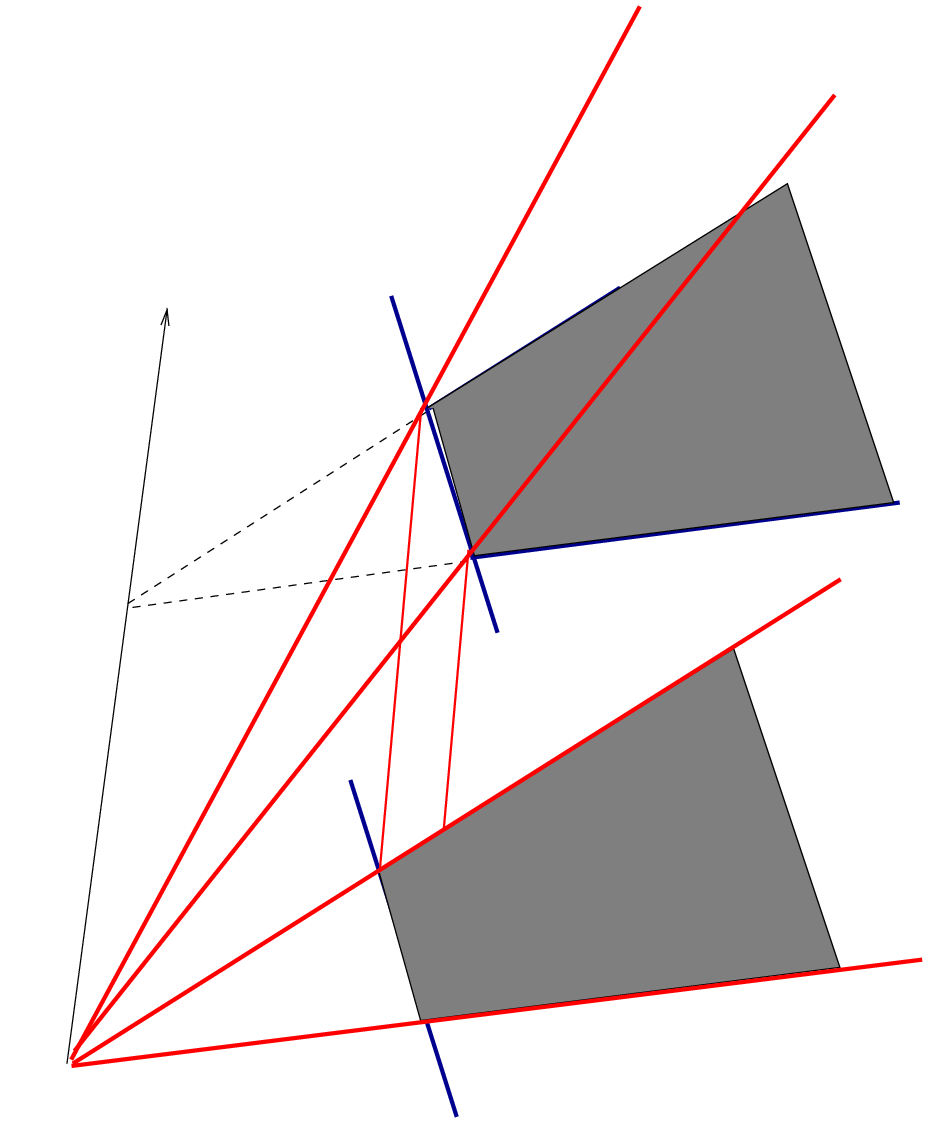}%
\end{picture}%
\setlength{\unitlength}{4144sp}%
\begingroup\makeatletter\ifx\SetFigFont\undefined%
\gdef\SetFigFont#1#2#3#4#5{%
  \reset@font\fontsize{#1}{#2pt}%
  \fontfamily{#3}\fontseries{#4}\fontshape{#5}%
  \selectfont}%
\fi\endgroup%
\begin{picture}(7059,8526)(10741,-5899)
\put(13771,-4111){\makebox(0,0)[lb]{\smash{{\SetFigFont{14}{16.8}{\rmdefault}{\mddefault}{\updefault}{\color[rgb]{0,0,0}$(a,1,0)$}%
}}}}
\put(14176,-5371){\makebox(0,0)[lb]{\smash{{\SetFigFont{14}{16.8}{\rmdefault}{\mddefault}{\updefault}{\color[rgb]{0,0,0}$(a+b,1,0)$}%
}}}}
\put(11971,434){\makebox(0,0)[lb]{\smash{{\SetFigFont{20}{24.0}{\rmdefault}{\mddefault}{\updefault}{\color[rgb]{0,0,0}$z$}%
}}}}
\put(10756,-1951){\makebox(0,0)[lb]{\smash{{\SetFigFont{14}{16.8}{\rmdefault}{\mddefault}{\updefault}{\color[rgb]{0,0,0}$(0,0,1)$}%
}}}}
\put(10756,-5821){\makebox(0,0)[lb]{\smash{{\SetFigFont{14}{16.8}{\rmdefault}{\mddefault}{\updefault}{\color[rgb]{0,0,0}$(0,0,0)$}%
}}}}
\put(14176,-646){\makebox(0,0)[lb]{\smash{{\SetFigFont{14}{16.8}{\rmdefault}{\mddefault}{\updefault}{\color[rgb]{0,0,0}$(a+b,1,1)$}%
}}}}
\put(14491,-1861){\makebox(0,0)[lb]{\smash{{\SetFigFont{14}{16.8}{\rmdefault}{\mddefault}{\updefault}{\color[rgb]{0,0,0}$(a+b,1,1)$}%
}}}}
\put(17146,-3661){\makebox(0,0)[lb]{\smash{{\SetFigFont{20}{24.0}{\rmdefault}{\mddefault}{\updefault}{\color[rgb]{0,0,0}$\overline{C}\Xi$}%
}}}}
\end{picture}%

%% file: ExtensionOfTheCoral.pspdftex
\begin{picture}(0,0)%
\includegraphics{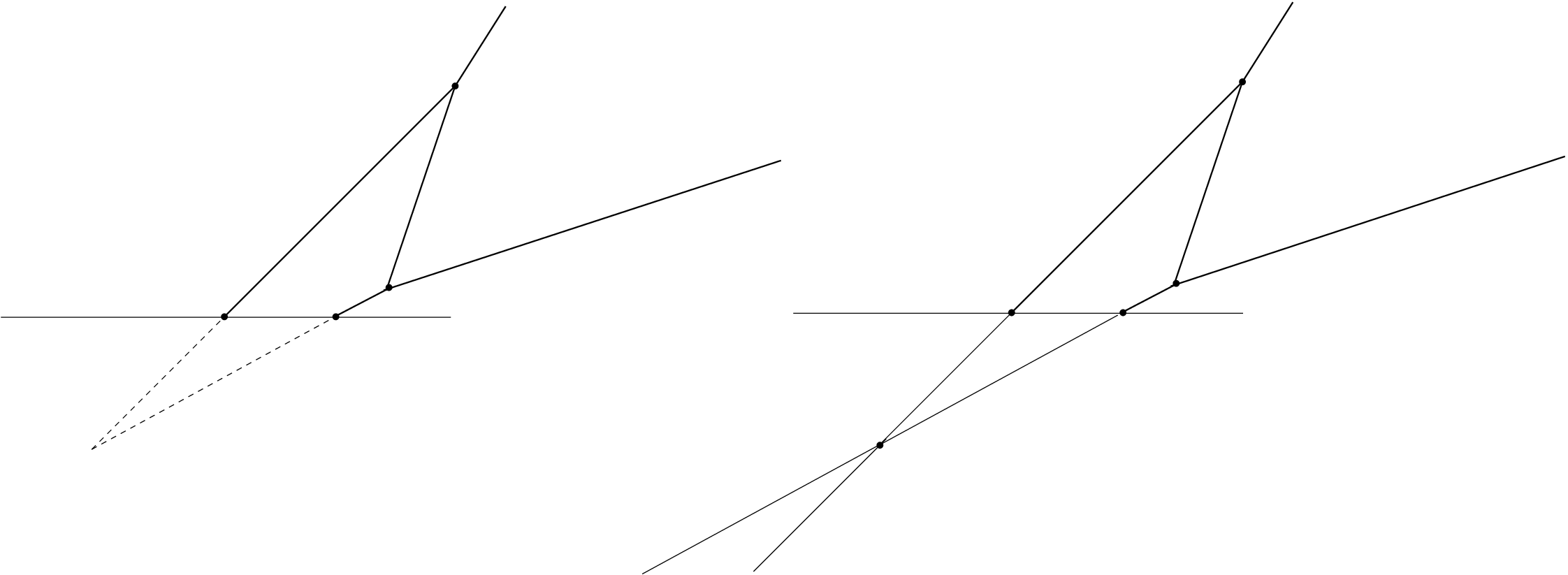}%
\end{picture}%
\setlength{\unitlength}{4144sp}%
\begingroup\makeatletter\ifx\SetFigFont\undefined%
\gdef\SetFigFont#1#2#3#4#5{%
  \reset@font\fontsize{#1}{#2pt}%
  \fontfamily{#3}\fontseries{#4}\fontshape{#5}%
  \selectfont}%
\fi\endgroup%
\begin{picture}(17089,6269)(2599,-3232)
\put(3781,-2131){\makebox(0,0)[lb]{\smash{{\SetFigFont{14}{16.8}{\rmdefault}{\mddefault}{\updefault}{\color[rgb]{0,0,0}$(0,0)$}%
}}}}
\put(12331,-2041){\makebox(0,0)[lb]{\smash{{\SetFigFont{14}{16.8}{\rmdefault}{\mddefault}{\updefault}{\color[rgb]{0,0,0}$(0,0)$}%
}}}}
\end{picture}%

%% file: buke.pspdftex
\begin{picture}(0,0)%
\includegraphics{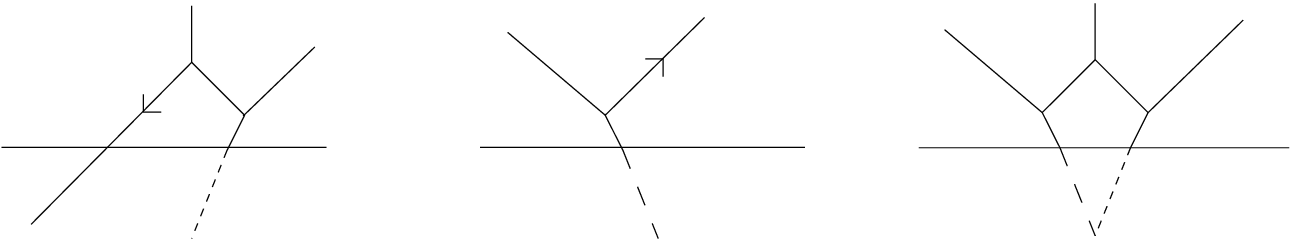}%
\end{picture}%
\setlength{\unitlength}{4144sp}%
\begingroup\makeatletter\ifx\SetFigFont\undefined%
\gdef\SetFigFont#1#2#3#4#5{%
  \reset@font\fontsize{#1}{#2pt}%
  \fontfamily{#3}\fontseries{#4}\fontshape{#5}%
  \selectfont}%
\fi\endgroup%
\begin{picture}(9834,1816)(-5591,1709)
\end{picture}%

%% file: Rescale.pspdftex
\begin{picture}(0,0)%
\includegraphics{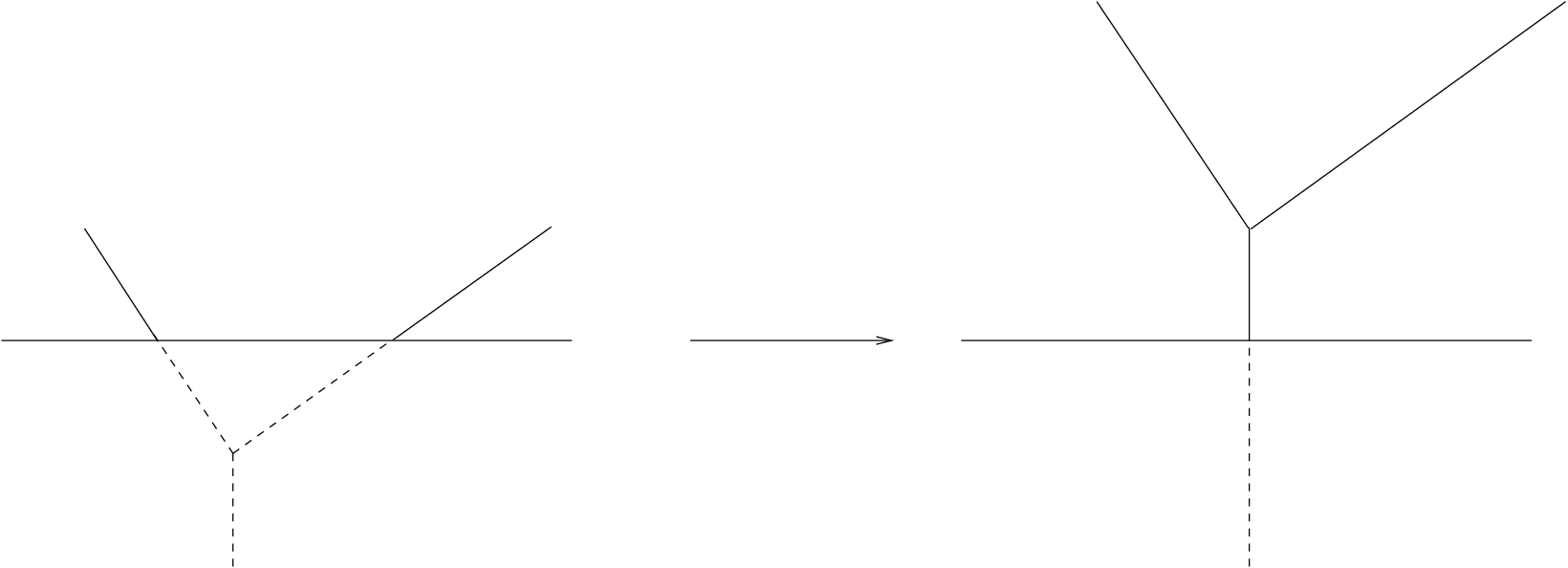}%
\end{picture}%
\setlength{\unitlength}{4144sp}%
\begingroup\makeatletter\ifx\SetFigFont\undefined%
\gdef\SetFigFont#1#2#3#4#5{%
  \reset@font\fontsize{#1}{#2pt}%
  \fontfamily{#3}\fontseries{#4}\fontshape{#5}%
  \selectfont}%
\fi\endgroup%
\begin{picture}(12489,4524)(3589,-5023)
\put(9361,-3616){\makebox(0,0)[lb]{\smash{{\SetFigFont{14}{16.8}{\rmdefault}{\mddefault}{\updefault}{\color[rgb]{0,0,0}Rescale}%
}}}}
\end{picture}%

%% file: polygon.pspdftex
\begin{picture}(0,0)%
\includegraphics{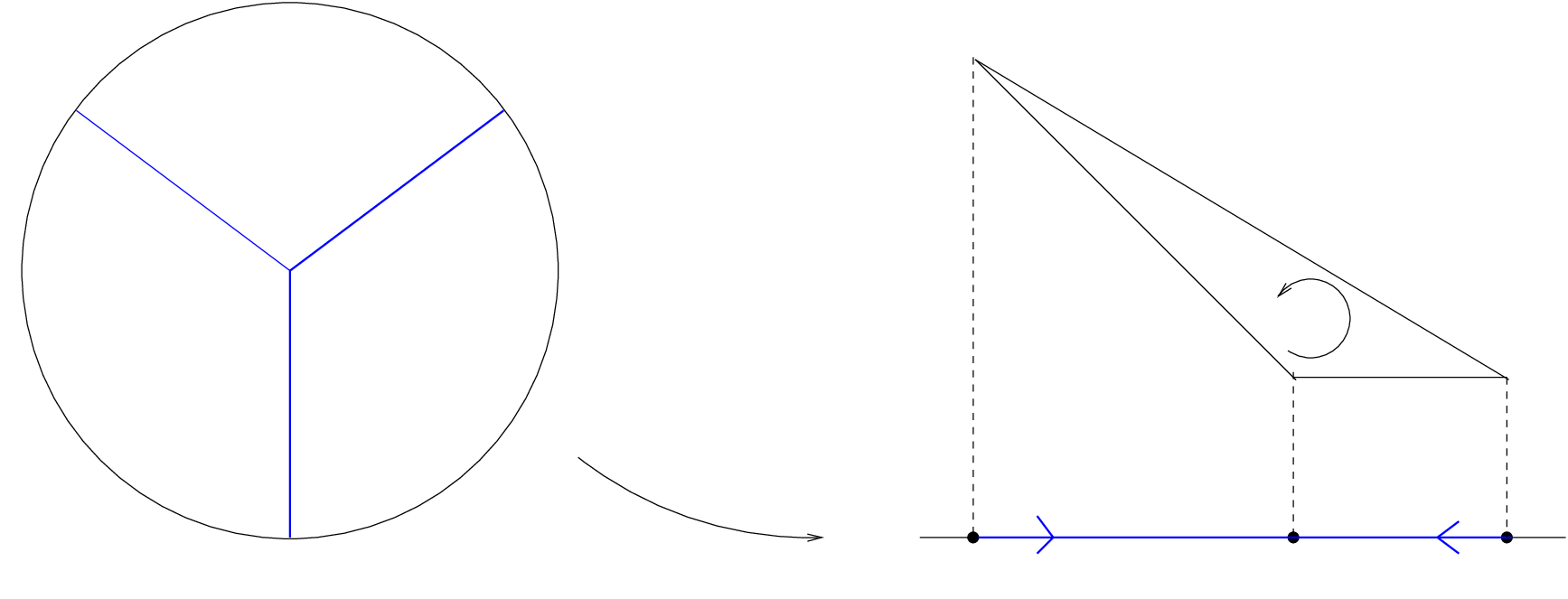}%
\end{picture}%
\setlength{\unitlength}{4144sp}%
\begingroup\makeatletter\ifx\SetFigFont\undefined%
\gdef\SetFigFont#1#2#3#4#5{%
  \reset@font\fontsize{#1}{#2pt}%
  \fontfamily{#3}\fontseries{#4}\fontshape{#5}%
  \selectfont}%
\fi\endgroup%
\begin{picture}(13212,5038)(2956,-5530)
\put(6481,-2311){\makebox(0,0)[lb]{\smash{{\SetFigFont{14}{16.8}{\rmdefault}{\mddefault}{\updefault}{\color[rgb]{0,0,0}$n_{e_{0,1}}=3$}%
}}}}
\put(3736,-2311){\makebox(0,0)[lb]{\smash{{\SetFigFont{14}{16.8}{\rmdefault}{\mddefault}{\updefault}{\color[rgb]{0,0,0}$n_{e_{1,2}}=2$}%
}}}}
\put(5446,-4291){\makebox(0,0)[lb]{\smash{{\SetFigFont{14}{16.8}{\rmdefault}{\mddefault}{\updefault}{\color[rgb]{0,0,0}$n_{e_{0,2}}=5$}%
}}}}
\put(7426,-1366){\makebox(0,0)[lb]{\smash{{\SetFigFont{14}{16.8}{\rmdefault}{\mddefault}{\updefault}{\color[rgb]{0,0,0}$v_{0,1}$}%
}}}}
\put(5356,-5461){\makebox(0,0)[lb]{\smash{{\SetFigFont{14}{16.8}{\rmdefault}{\mddefault}{\updefault}{\color[rgb]{0,0,0}$v_{0,2}$}%
}}}}
\put(8461,-5191){\makebox(0,0)[lb]{\smash{{\SetFigFont{20}{24.0}{\rmdefault}{\mddefault}{\updefault}{\color[rgb]{0,0,0}$\phi$}%
}}}}
\put(11791,-4741){\makebox(0,0)[lb]{\smash{{\SetFigFont{14}{16.8}{\rmdefault}{\mddefault}{\updefault}{\color[rgb]{0,0,0}$n_{e_{1,2}}=2$}%
}}}}
\put(14491,-4741){\makebox(0,0)[lb]{\smash{{\SetFigFont{14}{16.8}{\rmdefault}{\mddefault}{\updefault}{\color[rgb]{0,0,0}$n_{e_{0,1}}=3$}%
}}}}
\put(15661,-5461){\makebox(0,0)[lb]{\smash{{\SetFigFont{14}{16.8}{\rmdefault}{\mddefault}{\updefault}{\color[rgb]{0,0,0}$p_{0,1}=2$}%
}}}}
\put(13861,-5461){\makebox(0,0)[lb]{\smash{{\SetFigFont{14}{16.8}{\rmdefault}{\mddefault}{\updefault}{\color[rgb]{0,0,0}$p_{0,2}=0$}%
}}}}
\put(11161,-5461){\makebox(0,0)[lb]{\smash{{\SetFigFont{14}{16.8}{\rmdefault}{\mddefault}{\updefault}{\color[rgb]{0,0,0}$p_{1,2}=-3$}%
}}}}
\put(13546,-2176){\makebox(0,0)[lb]{\smash{{\SetFigFont{14}{16.8}{\rmdefault}{\mddefault}{\updefault}{\color[rgb]{0,0,0}$S_{L_1}=-3$}%
}}}}
\put(15751,-3706){\makebox(0,0)[lb]{\smash{{\SetFigFont{14}{16.8}{\rmdefault}{\mddefault}{\updefault}{\color[rgb]{0,0,0}$\tilde{p}_{0,1}$}%
}}}}
\put(14401,-4066){\makebox(0,0)[lb]{\smash{{\SetFigFont{14}{16.8}{\rmdefault}{\mddefault}{\updefault}{\color[rgb]{0,0,0}$S_{L_0}=0$}%
}}}}
\put(6346,-3436){\makebox(0,0)[lb]{\smash{{\SetFigFont{14}{16.8}{\rmdefault}{\mddefault}{\updefault}{\color[rgb]{0,0,0}$n_0=0$}%
}}}}
\put(2971,-1366){\makebox(0,0)[lb]{\smash{{\SetFigFont{14}{16.8}{\rmdefault}{\mddefault}{\updefault}{\color[rgb]{0,0,0}$v_{1,2}$}%
}}}}
\put(3826,-3391){\makebox(0,0)[lb]{\smash{{\SetFigFont{14}{16.8}{\rmdefault}{\mddefault}{\updefault}{\color[rgb]{0,0,0}$n_2=5$}%
}}}}
\put(5266,-1456){\makebox(0,0)[lb]{\smash{{\SetFigFont{14}{16.8}{\rmdefault}{\mddefault}{\updefault}{\color[rgb]{0,0,0}$n_1=3$}%
}}}}
\put(12061,-3076){\makebox(0,0)[lb]{\smash{{\SetFigFont{14}{16.8}{\rmdefault}{\mddefault}{\updefault}{\color[rgb]{0,0,0}$S_{L_2}=-5$}%
}}}}
\put(11026,-781){\makebox(0,0)[lb]{\smash{{\SetFigFont{14}{16.8}{\rmdefault}{\mddefault}{\updefault}{\color[rgb]{0,0,0}$\tilde{p}_{1,2}$}%
}}}}
\put(13276,-3706){\makebox(0,0)[lb]{\smash{{\SetFigFont{14}{16.8}{\rmdefault}{\mddefault}{\updefault}{\color[rgb]{0,0,0}$\tilde{p}_{0,2}$}%
}}}}
\end{picture}%

%% file: backforth.pspdftex
\begin{picture}(0,0)%
\includegraphics{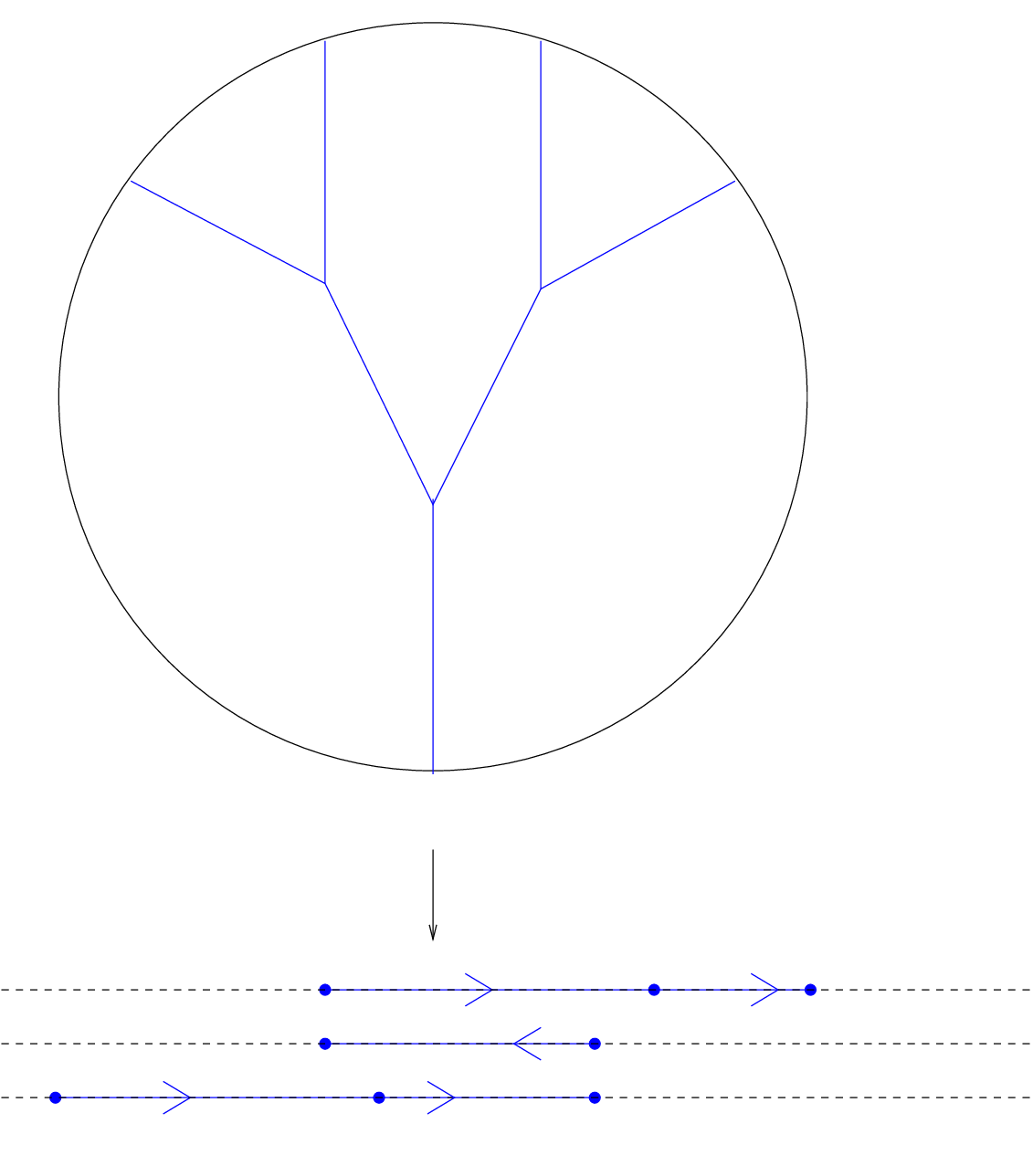}%
\end{picture}%
\setlength{\unitlength}{4144sp}%
\begingroup\makeatletter\ifx\SetFigFont\undefined%
\gdef\SetFigFont#1#2#3#4#5{%
  \reset@font\fontsize{#1}{#2pt}%
  \fontfamily{#3}\fontseries{#4}\fontshape{#5}%
  \selectfont}%
\fi\endgroup%
\begin{picture}(8619,9801)(1789,-7915)
\put(6346,-2671){\makebox(0,0)[lb]{\smash{{\SetFigFont{14}{16.8}{\rmdefault}{\mddefault}{\updefault}{\color[rgb]{0,0,0}$n_0=0$}%
}}}}
\put(6751,434){\makebox(0,0)[lb]{\smash{{\SetFigFont{14}{16.8}{\rmdefault}{\mddefault}{\updefault}{\color[rgb]{0,0,0}$n_1=-2$}%
}}}}
\put(5176,434){\makebox(0,0)[lb]{\smash{{\SetFigFont{14}{16.8}{\rmdefault}{\mddefault}{\updefault}{\color[rgb]{0,0,0}$n_2=1$}%
}}}}
\put(3556,434){\makebox(0,0)[lb]{\smash{{\SetFigFont{14}{16.8}{\rmdefault}{\mddefault}{\updefault}{\color[rgb]{0,0,0}$n_3=3$}%
}}}}
\put(3781,-2671){\makebox(0,0)[lb]{\smash{{\SetFigFont{14}{16.8}{\rmdefault}{\mddefault}{\updefault}{\color[rgb]{0,0,0}$n_4=-5$}%
}}}}
\put(8101,389){\makebox(0,0)[lb]{\smash{{\SetFigFont{14}{16.8}{\rmdefault}{\mddefault}{\updefault}{\color[rgb]{0,0,0}$v_{0,1}$}%
}}}}
\put(6301,1694){\makebox(0,0)[lb]{\smash{{\SetFigFont{14}{16.8}{\rmdefault}{\mddefault}{\updefault}{\color[rgb]{0,0,0}$v_{1,2}$}%
}}}}
\put(4411,1739){\makebox(0,0)[lb]{\smash{{\SetFigFont{14}{16.8}{\rmdefault}{\mddefault}{\updefault}{\color[rgb]{0,0,0}$v_{2,3}$}%
}}}}
\put(2161,434){\makebox(0,0)[lb]{\smash{{\SetFigFont{14}{16.8}{\rmdefault}{\mddefault}{\updefault}{\color[rgb]{0,0,0}$v_{3,4}$}%
}}}}
\put(5356,-4966){\makebox(0,0)[lb]{\smash{{\SetFigFont{14}{16.8}{\rmdefault}{\mddefault}{\updefault}{\color[rgb]{0,0,0}$v_{0,4}$}%
}}}}
\put(3331,-7666){\makebox(0,0)[lb]{\smash{{\SetFigFont{14}{16.8}{\rmdefault}{\mddefault}{\updefault}{\color[rgb]{0,0,1}$2$}%
}}}}
\put(5716,-7036){\makebox(0,0)[lb]{\smash{{\SetFigFont{14}{16.8}{\rmdefault}{\mddefault}{\updefault}{\color[rgb]{0,0,1}$-5$}%
}}}}
\put(5536,-7621){\makebox(0,0)[lb]{\smash{{\SetFigFont{14}{16.8}{\rmdefault}{\mddefault}{\updefault}{\color[rgb]{0,0,1}$-6$}%
}}}}
\put(4861,-7846){\makebox(0,0)[lb]{\smash{{\SetFigFont{14}{16.8}{\rmdefault}{\mddefault}{\updefault}{\color[rgb]{0,0,0}$p_{3,4}$}%
}}}}
\put(2206,-7846){\makebox(0,0)[lb]{\smash{{\SetFigFont{14}{16.8}{\rmdefault}{\mddefault}{\updefault}{\color[rgb]{0,0,0}$p_{2,3}$}%
}}}}
\put(7201,-6631){\makebox(0,0)[lb]{\smash{{\SetFigFont{14}{16.8}{\rmdefault}{\mddefault}{\updefault}{\color[rgb]{0,0,0}$p_{0,1}$}%
}}}}
\put(5536,-5641){\makebox(0,0)[lb]{\smash{{\SetFigFont{14}{16.8}{\rmdefault}{\mddefault}{\updefault}{\color[rgb]{0,0,0}$\phi$}%
}}}}
\put(8686,-6226){\makebox(0,0)[lb]{\smash{{\SetFigFont{14}{16.8}{\rmdefault}{\mddefault}{\updefault}{\color[rgb]{0,0,0}$p_{1,2}$}%
}}}}
\put(8056,-6091){\makebox(0,0)[lb]{\smash{{\SetFigFont{14}{16.8}{\rmdefault}{\mddefault}{\updefault}{\color[rgb]{0,0,1}$3$}%
}}}}
\put(5941,-6091){\makebox(0,0)[lb]{\smash{{\SetFigFont{14}{16.8}{\rmdefault}{\mddefault}{\updefault}{\color[rgb]{0,0,1}$1$}%
}}}}
\put(4366,-6676){\makebox(0,0)[lb]{\smash{{\SetFigFont{14}{16.8}{\rmdefault}{\mddefault}{\updefault}{\color[rgb]{0,0,0}$p_{0,4}$}%
}}}}
\put(6706,-7126){\makebox(0,0)[lb]{\smash{{\SetFigFont{14}{16.8}{\rmdefault}{\mddefault}{\updefault}{\color[rgb]{0,0,0}$\phi(w)$}%
}}}}
\end{picture}%

%% file: TropMorse1.pspdftex
\begin{picture}(0,0)%
\includegraphics{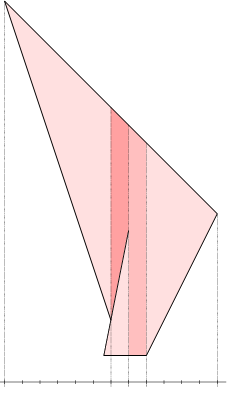}%
\end{picture}%
\setlength{\unitlength}{622sp}%
\begingroup\makeatletter\ifx\SetFigFont\undefined%
\gdef\SetFigFont#1#2#3#4#5{%
  \reset@font\fontsize{#1}{#2pt}%
  \fontfamily{#3}\fontseries{#4}\fontshape{#5}%
  \selectfont}%
\fi\endgroup%
\begin{picture}(11499,20298)(-236,-9526)
\put(  1,-9511){\makebox(0,0)[b]{\smash{{\SetFigFont{8}{9.6}{\familydefault}{\mddefault}{\updefault}{\color[rgb]{0,0,0}0}%
}}}}
\put(5401,-9511){\makebox(0,0)[b]{\smash{{\SetFigFont{8}{9.6}{\familydefault}{\mddefault}{\updefault}{\color[rgb]{0,0,0}6}%
}}}}
\put(6301,-9511){\makebox(0,0)[b]{\smash{{\SetFigFont{8}{9.6}{\familydefault}{\mddefault}{\updefault}{\color[rgb]{0,0,0}7}%
}}}}
\put(7201,-9511){\makebox(0,0)[b]{\smash{{\SetFigFont{8}{9.6}{\familydefault}{\mddefault}{\updefault}{\color[rgb]{0,0,0}8}%
}}}}
\put(10801,-9511){\makebox(0,0)[b]{\smash{{\SetFigFont{8}{9.6}{\familydefault}{\mddefault}{\updefault}{\color[rgb]{0,0,0}12}%
}}}}
\put(9271,-5461){\makebox(0,0)[b]{\smash{{\SetFigFont{8}{9.6}{\familydefault}{\mddefault}{\updefault}{\color[rgb]{0,0,0}L1}%
}}}}
\put(7741,3989){\makebox(0,0)[b]{\smash{{\SetFigFont{8}{9.6}{\familydefault}{\mddefault}{\updefault}{\color[rgb]{0,0,0}L2}%
}}}}
\put(6121,-8161){\makebox(0,0)[b]{\smash{{\SetFigFont{8}{9.6}{\familydefault}{\mddefault}{\updefault}{\color[rgb]{0,0,0}L0}%
}}}}
\put(4501,-6451){\makebox(0,0)[b]{\smash{{\SetFigFont{8}{9.6}{\familydefault}{\mddefault}{\updefault}{\color[rgb]{0,0,0}L4}%
}}}}
\put(2341,929){\makebox(0,0)[b]{\smash{{\SetFigFont{8}{9.6}{\familydefault}{\mddefault}{\updefault}{\color[rgb]{0,0,0}L3}%
}}}}
\end{picture}%

%% file: CoralToTMT.pspdftex
\begin{picture}(0,0)%
\includegraphics{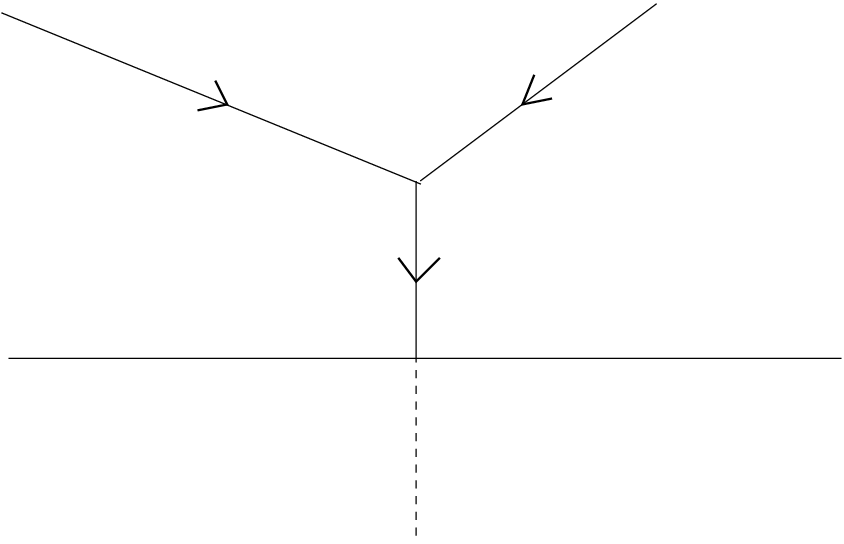}%
\end{picture}%
\setlength{\unitlength}{4144sp}%
\begingroup\makeatletter\ifx\SetFigFont\undefined%
\gdef\SetFigFont#1#2#3#4#5{%
  \reset@font\fontsize{#1}{#2pt}%
  \fontfamily{#3}\fontseries{#4}\fontshape{#5}%
  \selectfont}%
\fi\endgroup%
\begin{picture}(6423,4128)(16675,-5525)
\put(21061,-2266){\makebox(0,0)[lb]{\smash{{\SetFigFont{14}{16.8}{\rmdefault}{\mddefault}{\updefault}{\color[rgb]{0,0,0}$3(-2,-1)$}%
}}}}
\put(17326,-2266){\makebox(0,0)[lb]{\smash{{\SetFigFont{14}{16.8}{\rmdefault}{\mddefault}{\updefault}{\color[rgb]{0,0,0}$2(3,-1)$}%
}}}}
\put(20161,-3616){\makebox(0,0)[lb]{\smash{{\SetFigFont{14}{16.8}{\rmdefault}{\mddefault}{\updefault}{\color[rgb]{0,0,0}$5(0,-1)$}%
}}}}
\put(19981,-5461){\makebox(0,0)[lb]{\smash{{\SetFigFont{14}{16.8}{\rmdefault}{\mddefault}{\updefault}{\color[rgb]{0,0,0}$(0,0)$}%
}}}}
\put(19936,-4381){\makebox(0,0)[lb]{\smash{{\SetFigFont{14}{16.8}{\rmdefault}{\mddefault}{\updefault}{\color[rgb]{0,0,0}$(0,1)$}%
}}}}
\end{picture}%

%% file: CoralsBoth.pspdftex
\begin{picture}(0,0)%
\includegraphics{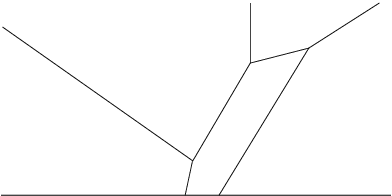}%
\end{picture}%
\setlength{\unitlength}{2486sp}%
\begingroup\makeatletter\ifx\SetFigFont\undefined%
\gdef\SetFigFont#1#2#3#4#5{%
  \reset@font\fontsize{#1}{#2pt}%
  \fontfamily{#3}\fontseries{#4}\fontshape{#5}%
  \selectfont}%
\fi\endgroup%
\begin{picture}(4974,2464)(11059,-4573)
\end{picture}%

%% file: UU0.pspdftex
\begin{picture}(0,0)%
\includegraphics{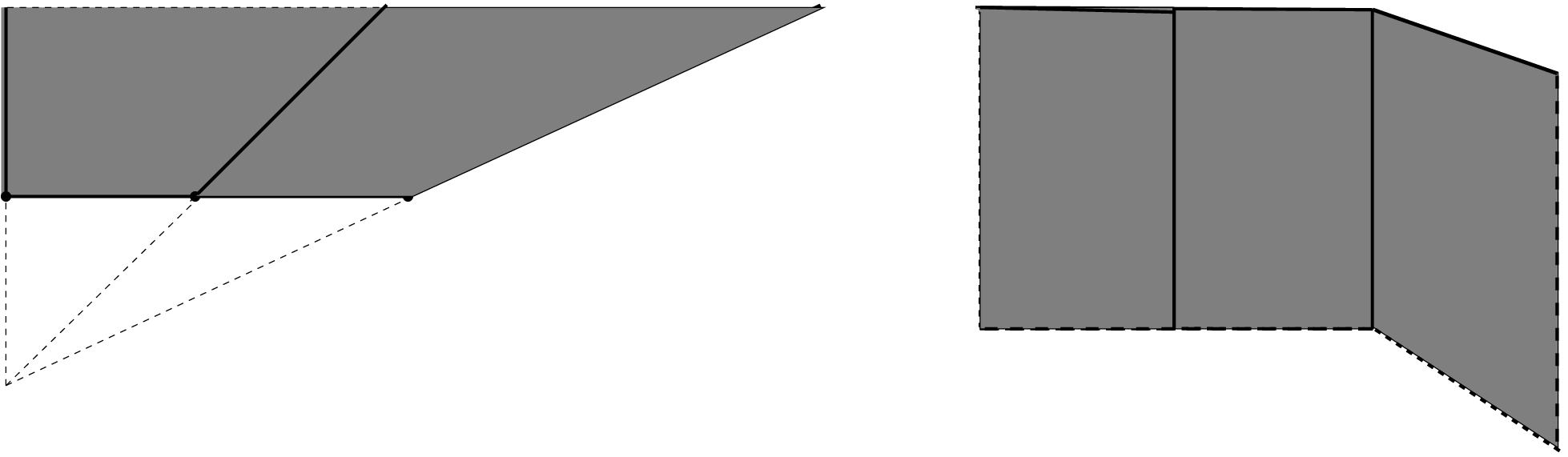}%
\end{picture}%
\setlength{\unitlength}{4144sp}%
\begingroup\makeatletter\ifx\SetFigFont\undefined%
\gdef\SetFigFont#1#2#3#4#5{%
  \reset@font\fontsize{#1}{#2pt}%
  \fontfamily{#3}\fontseries{#4}\fontshape{#5}%
  \selectfont}%
\fi\endgroup%
\begin{picture}(14913,4337)(5794,-6142)
\put(19351,-3931){\makebox(0,0)[lb]{\smash{{\SetFigFont{14}{16.8}{\rmdefault}{\mddefault}{\updefault}{\color[rgb]{0,0,0}$\AA^1 \times \AA^1$}%
}}}}
\put(17461,-3526){\makebox(0,0)[lb]{\smash{{\SetFigFont{14}{16.8}{\rmdefault}{\mddefault}{\updefault}{\color[rgb]{0,0,0}$\PP^1 \times \AA^1$}%
}}}}
\put(15571,-3526){\makebox(0,0)[lb]{\smash{{\SetFigFont{14}{16.8}{\rmdefault}{\mddefault}{\updefault}{\color[rgb]{0,0,0}$\AA^1 \times \AA^1$}%
}}}}
\end{picture}%

%% file: Z0.pspdftex
\begin{picture}(0,0)%
\includegraphics{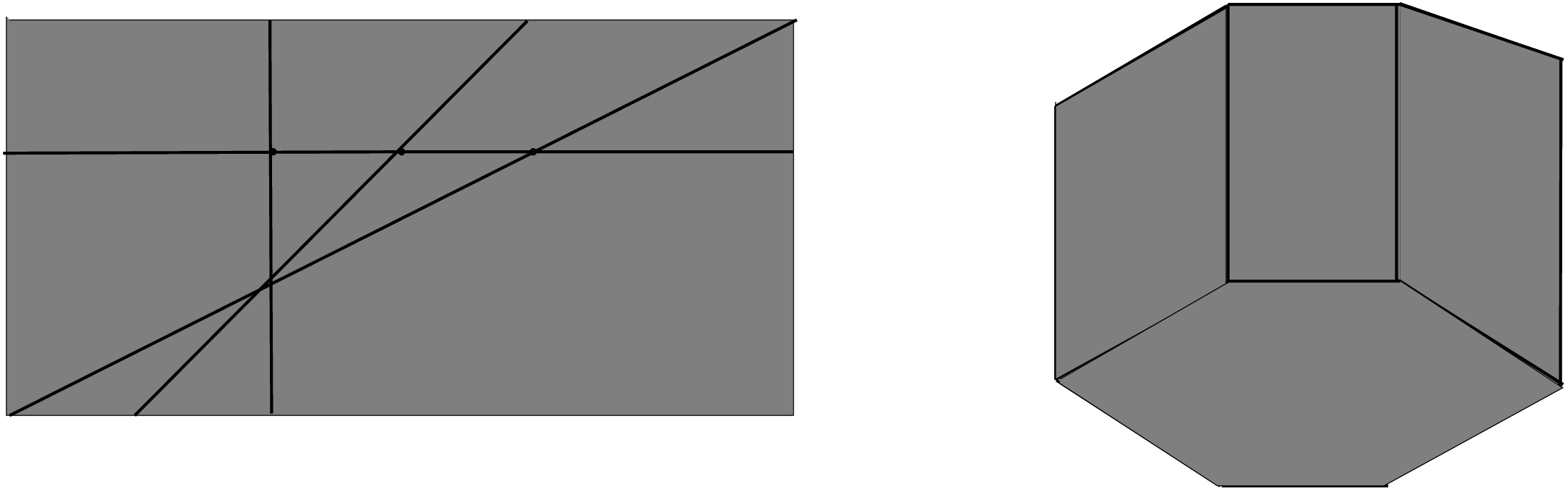}%
\end{picture}%
\setlength{\unitlength}{4144sp}%
\begingroup\makeatletter\ifx\SetFigFont\undefined%
\gdef\SetFigFont#1#2#3#4#5{%
  \reset@font\fontsize{#1}{#2pt}%
  \fontfamily{#3}\fontseries{#4}\fontshape{#5}%
  \selectfont}%
\fi\endgroup%
\begin{picture}(16543,5146)(2173,-6799)
\put(15596,-3376){\makebox(0,0)[lb]{\smash{{\SetFigFont{20}{24.0}{\rmdefault}{\mddefault}{\updefault}{\color[rgb]{0,0,0}$\PP^1\times \PP^1$}%
}}}}
\put(17416,-3684){\makebox(0,0)[lb]{\smash{{\SetFigFont{20}{24.0}{\rmdefault}{\mddefault}{\updefault}{\color[rgb]{0,0,0}$\PP^1\times \PP^1$}%
}}}}
\put(15796,-5776){\makebox(0,0)[lb]{\smash{{\SetFigFont{25}{30.0}{\rmdefault}{\mddefault}{\updefault}{\color[rgb]{0,0,0}$Z_0$}%
}}}}
\put(13771,-3976){\makebox(0,0)[lb]{\smash{{\SetFigFont{20}{24.0}{\rmdefault}{\mddefault}{\updefault}{\color[rgb]{0,0,0}$\PP^1\times \PP^1$}%
}}}}
\end{picture}%